\documentclass[11pt]{article}
\usepackage[left=2.7cm,right=2.7cm,top=2.2cm,bottom=2.2cm]{geometry}
\usepackage{amssymb}
\usepackage{amsmath}
\usepackage{amsthm}
\usepackage{enumerate}
\usepackage{mathrsfs}
\usepackage{stmaryrd}
\usepackage{bbm}
\usepackage{xcolor}
\usepackage[hyperfootnotes=false]{hyperref}

\theoremstyle{definition}
\newtheorem{definition}{Definition}[section]

\newtheorem{example}[definition]{Example}

\theoremstyle{plain}
\newtheorem{proposition}[definition]{Proposition}
\newtheorem{theorem}[definition]{Theorem}
\newtheorem{lemma}[definition]{Lemma}

\theoremstyle{remark}
\newtheorem{remark}[definition]{Remark}

\numberwithin{equation}{section}

\newcommand{\E}{\mathbb{E}}
\newcommand{\N}{\mathbb{N}}
\renewcommand{\P}{\mathbb{P}}

\newcommand{\R}{\mathbb{R}}

\newcommand{\X}{\mathbb{X}}

\newcommand{\cB}{\mathcal{B}}

\newcommand{\cF}{\mathcal{F}}
\newcommand{\cG}{\mathcal{G}}

\newcommand{\cI}{\mathcal{I}}

\newcommand{\cL}{\mathcal{L}}

\newcommand{\cP}{\mathcal{P}}
\newcommand{\cQ}{\mathcal{Q}}

\newcommand{\cV}{\mathcal{V}}

\newcommand{\bX}{\mathbf{X}}

\newcommand{\bZ}{\mathbf{Z}}

\newcommand{\1}{\mathbf{1}}
\renewcommand{\d}{\mathrm{d}}
\newcommand{\dd}{\,\mathrm{d}}

\renewcommand{\epsilon}{\varepsilon}

\newcommand{\edot}{\boldsymbol{\cdot}}
\newcommand{\tw}{\widetilde{w}}
\newcommand{\bw}{\bar{w}}
\newcommand{\tY}{\widetilde{Y}}
\newcommand{\tbX}{\widetilde{\mathbf{X}}}
\newcommand{\tM}{\widetilde{M}}
\newcommand{\tX}{\widetilde{X}}
\newcommand{\tbbX}{\widetilde{\mathbb{X}}}
\newcommand{\hw}{\hat{w}}
\newcommand{\D}{\mathrm{D}}
\newcommand{\sV}{\mathscr{V}}
\newcommand{\tZ}{\widetilde{Z}}
\newcommand{\ty}{\widetilde{y}}
\DeclareMathOperator*{\esssup}{ess\,sup}
\DeclareMathOperator*{\Sym}{Sym}
\DeclareMathOperator{\tr}{tr}

\begin{document}

\title{Rough Stochastic Analysis with Jumps}
\author{Andrew L.~Allan\footnote{andrew.l.allan{\fontfamily{ptm}\selectfont @}durham.ac.uk} \hspace{-2pt} and Jost Pieper\footnote{jost.pieper{\fontfamily{ptm}\selectfont @}durham.ac.uk}\\
~\vspace{-5pt}\\
{\normalsize Department of Mathematical Sciences}\\
{\normalsize Durham University}\\}
\date{\today}
\maketitle

\vspace{-12pt}

\begin{abstract}
We present a new version of the stochastic sewing lemma, capable of handling multiple discontinuous control functions. This is then used to develop a theory of rough stochastic analysis in a c\`adl\`ag setting. In particular, we define rough stochastic integration and establish solutions to rough stochastic differential equations with jump discontinuities in both sources of noise, along with an It\^o formula for stochastic controlled paths with jumps.
\end{abstract}

\vspace{2pt}



\vspace{5pt}

\tableofcontents

\section{Introduction}

Over the last few decades, the theory of rough paths has both established itself as a cornerstone of modern stochastic analysis, and successfully branched out in a multitude of research directions, from abstract algebra and financial mathematics to statistics and machine learning. One of its fundamental contributions, which has been at the heart of the theory since its inception by Lyons in the seminal paper \cite{Lyons1998}, was to identify the precise informational structure of a potentially highly oscillatory path which is required to guarantee continuity of the solution to a nonlinear differential equation driven by such a path. In the context of stochastic analysis, this insight led naturally to a robust representation of a stochastic differential equation (SDE) $\d Y_t = b(Y_t) \dd t + f(Y_t) \dd X_t$ as a (random) rough differential equation (RDE)
\begin{equation*}
\d Y_t = b(Y_t) \dd t + f(Y_t) \dd \bX_t
\end{equation*}
driven by a rough path $\bX$, viewed as the process $X$ enhanced with its iterated integrals. Unlike its purely stochastic counterpart, such an RDE is well-defined for a fixed realization of the driving noise, and comes with strong local Lipschitz estimates for the solution with respect to suitable rough path topologies. For details and a full historical account, see one of the many monographs on the subject, such as \cite{FrizHairer2020} or \cite{FrizVictoir2010}.

Beyond classical SDEs, it is common to encounter differential equations with two (or more) stochastic drivers of the form
\begin{equation}\label{eq: doubly stochastic differential equation}
\dd Y_t = b(Y_t) \dd t + \sigma(Y_t) \dd M_t + f(Y_t) \dd X_t,
\end{equation}
where $M$ and $X$ are stochastic processes, which in particular may \emph{not} both be semimartingales. Moreover, an intricate and increasingly important class of equations are obtained by \emph{conditioning} on one of the sources of noise. A common example of such a scenario is that of conditional It\^o diffusions, in which the driving noises in \eqref{eq: doubly stochastic differential equation} are taken to be independent Brownian motions, though more general classes of noise are of increasing interest. From a rough path perspective, it is natural to fix a realization of the noise on which the solution has been conditioned, $X$ say, and to lift this to its robust representation as a (deterministic) rough path $\bX$. This leads to a so-called \emph{rough stochastic differential equation} (RSDE) of the form
\begin{equation}\label{eq: RSDE introduction}
\dd Y_t = b(Y_t) \dd t + \sigma(Y_t) \dd M_t + f(Y_t) \dd \bX_t,
\end{equation}
which is driven by both rough and stochastic noise.

One area in which such equations appear naturally is in the context of stochastic filtering, in which one is interested in the conditional distribution of a signal process given observations of another process whose dynamics depend on the signal. After a suitable Girsanov transformation and the lifting of the observation process to a rough path, the system dynamics turn into an RSDE, in which the rough (resp.~stochastic) noise represents the observed (resp.~unobserved) part of the system. The continuity of the solution with respect to the rough path then translates to robustness of the posterior distribution with respect to the observed data, which in general only holds in rough path topology; see, e.g., Crisan, Diehl, Friz and Oberhauser \cite{CrisanDiehlFrizOberhauser2013}, Diehl, Oberhauser and Riedel \cite{DiehlOberhauserRiedel2015} or Bugini, Friz, L{\^e} and Zhang \cite{BuginiFrizLeZhang2025}, or Allan, Pieper and Teichmann \cite{AllanPieperTeichmann2025} for the recent extension to robust filtering for jump-diffusions.

Equations of the form in \eqref{eq: doubly stochastic differential equation} also appear in the study of interacting particle systems with common noise. Consider, for instance, a swarm of interacting particles, each one with its own independent noise, and all of them subject to an additional universal noise. Under suitable assumptions it can be shown that a conditional propagation of chaos holds, such that, conditioned on the common noise, the empirical measure converges to a conditional law governed by suitable McKean--Vlasov dynamics. Cases with independent noise, e.g., Kurtz and Xiong \cite{KurzXiong1999}, as well as with correlated noise, e.g., Coghi and Flandoli \cite{CoghiFlandoli2016}, have been considered. Conditional McKean--Vlasov dynamics with jump discontinuities appear in particular in systemic risk modelling, e.g., Hambly and S\o{}jmark \cite{HamblySojmark2019}. Such particle systems also give rise to stochastic Fokker--Planck equations, as in Coghi and Gess \cite{CoghiGess2019}, and indeed to rough Fokker--Planck equations, as in Coghi and Nilssen \cite{CoghiNilssen2021}. They have moreover been applied in the filtering context, where the empirical measure provides a numerical approximation of the conditional density of the signal; see, e.g., Bain and Crisan \cite[Chapter~9]{BainCrisan2009} or Pathiraja, Reich and Stannat \cite{PathirajaReichStannat2021}. Recently, Coghi, Nilssen, N\"usken and Reich \cite{CoghiNilssenNuskenReich2023} studied an RSDE of McKean--Vlasov-type in the development of a rough path version of an ensemble Kalman filter, in which the continuity of solutions to the RSDE with respect to the rough driver leads to a propagation of chaos result for the associated particle system.

Another advantage of the rough path framework is its ability to handle general classes of stochastic noise beyond semimartingales. Correspondingly, RSDEs as in \eqref{eq: RSDE introduction} allow one to study systems driven by both semimartingale and non-semimartingale noise. In particular, dynamics involving both standard and fractional Brownian motion appear frequently, for instance in slow-fast systems (e.g., Hairer and Li \cite{HairerLi2020}), and in rough volatility models (e.g., Bayer, Friz and Gatheral \cite{BayerFrizGatheral2016}). Relatedly, Bank, Bayer, Friz and Pelizzari \cite{BankBayerFrizPelizzari2025} derived an RSDE from the conditional dynamics of a local stochastic volatility model, and used a rough Feynman--Kac formula to represent the corresponding pricing problem as a rough PDE. Recently, Malliavin differentiability of solutions, as well as existence of solutions to RSDEs with linear coefficients have been established by Bugini, Coghi and Nilssen \cite{BuginiCoghiNilssen2024}, and Li, Zhang and Zhang \cite{LiZhangZhang2024} have also considered rough BSDEs with reflection.

Several approaches have been proposed to make sense of solutions to RSDEs of the form in \eqref{eq: RSDE introduction}, at least for continuous driving noise. In \cite{CrisanDiehlFrizOberhauser2013}, solutions are constructed via a flow transformation which, in the spirit of a Doss--Sussmann transformation, effectively turns the RSDE into a classical SDE. Among several drawbacks of this approach, no intrinsic meaning is given to solutions of the RSDE, which in particular renders the approach suboptimal from a numerical analysis viewpoint. It is also a rigid approach, which requires excessive regularity conditions on the vector fields and does not cope well with jumps.

Another approach to RSDEs of the form in \eqref{eq: RSDE introduction} involves jointly lifting $M$ and $\bX$ to a new random rough path, $\bZ$ say, and treating the RSDE as a random RDE driven by $\bZ$, as in, e.g., Diehl, Oberhauser and Riedel \cite{DiehlOberhauserRiedel2015} and Friz and Hairer \cite[Chapter~12]{FrizHairer2020}. More recently, a significant generalization of this idea was developed by Friz and Zorin-Kranich \cite{FrizZorinKranich2023}. In order to construct such a joint lift $\bZ$, it is necessary to define the integral $\int M_u \dd \bX_u$, which does not exist as a (classical) rough integral since $M$ is not controlled by $X$ in the rough path sense. This is overcome by imposing integration by parts, noting that $\int X_u \dd M_u$ is a well-defined It\^o integral. In the general setting of \cite{FrizZorinKranich2023}, this requires the construction of a suitable bracket process $[M,\bX]$ to capture the quadratic covariation between $M$ and $\bX$. The authors moreover identify the novel class of \emph{rough semimartingales} as a suitable solution space for RSDEs driven by a general class of c\`adl\`ag martingales and rough paths. On the other hand, all such \emph{random rough path}-style approaches require excessive regularity assumptions on the vector field $\sigma$, and construct solutions to RSDEs in a fundamentally pathwise manner through predominantly analytical tools, which do not fully exploit the martingale structure of the stochastic noise.

A new insight on, among other fields, rough stochastic analysis was provided by the introduction of the stochastic sewing lemma by L\^e \cite{Le2020}. This allowed for the development by Friz, Hocquet and L\^e \cite{FrizHocquetLe2021} of a unified theory of RSDEs driven by Brownian motion and a H\"older continuous rough path, providing intrinsic solutions under optimal regularity assumptions on the vector fields. In this setting, the notion of stochastic controlled rough paths was introduced to make sense of so-called \emph{rough stochastic integration} via a generalized stochastic sewing lemma.

Liang and Tang \cite{LiangTang2025} extended the notion of rough stochastic integration to the $p$-variation setting, though still restricted to continuous paths, in order to treat rough BSDEs. Due to the lack of a stochastic sewing lemma capable of handling multiple control functions, they rely on a decomposition of stochastic controlled paths into a martingale and a deterministic controlled path and treat these separate parts individually, via stochastic sewing and classical (deterministic) sewing respectively.

Beginning with the work of Friz and Shekhar \cite{FrizShekhar2017}, the theory of rough paths with jumps has developed over the past few years. In the spirit of Marcus canonical SDEs, Chevyrev and Friz \cite{ChevyrevFriz2019} developed a theory of Marcus RDEs, in which the solution is continuous with respect to both the driving rough path, and the so-called path function used over artificial time intervals inserted at the jump times to effectively remove the discontinuities. Another approach, developed by Friz and Zhang \cite{FrizZhang2018}, constructs ``forward'' rough integrals and solutions to RDEs driven by general rough paths, using a generalized sewing lemma to incorporate discontinuous control functions.

The main motivation of this paper is to provide a theory of rough SDEs of the form in \eqref{eq: RSDE introduction}, driven by a c\`adl\`ag martingale $M$ and a c\`adl\`ag rough path $\bX$. This amounts to a significant generalization of the results of \cite{FrizHocquetLe2021} in terms of the permissible classes of driving noise, whilst still imposing essentially optimal regularity assumptions on the vector fields (in particular with $\sigma$ only assumed to be Lipschitz and bounded).

Introduced by Gubinelli \cite{Gubinelli2004} and Feyel and de La Pradelle \cite{FeyelDeLaPradelle2006}, the sewing lemma provides a simple criterion to guarantee the convergence of some abstract Riemann sums $\sum_{[s,t] \in \cP} \Xi_{s,t}$ along sequences of partitions $\cP$, along with corresponding estimates on the limiting object. In particular, suitable versions of the sewing lemma provide a convenient means to construct Young and rough integrals, even in the discontinuous setting of \cite{FrizZhang2018}. The stochastic sewing lemma \cite{Le2020} is a powerful extension of the sewing lemma, which exploits the stochastic cancellation of martingales in order to weaken the regularity assumptions on the summands $\Xi_{s,t}$.

Beyond RSDEs, the stochastic sewing lemma has proved to be a powerful tool in stochastic analysis more widely, particularly for establishing results on regularization by noise. In particular, Harang and Perkowski \cite{HarangPerkowski2021} used stochastic sewing to show that strong solutions to ODEs with any Schwartz distribution as drift coefficient exist for a large class of Gaussian processes as regularizing noise, and Harang and Ling \cite{HarangLing2021} found similar results for Volterra--L\'evy noise. Based on the stochastic sewing lemma, Kremp and Perkowski \cite{KrempPerkowski2023} developed an alternative notion of rough stochastic sewing integral to show well-posedness of SDEs with Besov drift and $\alpha$-stable L\'evy noise. Athreya, Butkovsky, L\^e and Mytnik \cite{AthreyaButkovskyLeMytnik2024} also recently studied a regularization by noise problem in an SPDE setting, again crucially relying on stochastic sewing.

Butkovsky, Dareiotis and Gerencs\'er \cite{ButkovskyDareiotisGerencser2021} developed a stochastic sewing approach to the numerical approximation of SDEs, and have since expanded these results to the numerics of SPDEs \cite{ButkovskyDareiotisGerencser2023} and of SDEs with L\'evy noise \cite{ButkovskyDareiotisGerencser2024}. Further applications of stochastic sewing include in the study of Euler schemes for singular SDEs (L\^e and Ling \cite{LeLing2022} and Dareiotis, Gerencs\'er and L\^e \cite{DareiotisGerencserLe2022}), homogenization problems for slow-fast systems (Hairer and Li \cite{HairerLi2020}), a stochastic reconstruction theorem (Kern \cite{Kern2023}), and obtaining regularity estimates for the local time of stochastic fields and fractional Brownian motion (Bechtold, Harang and Kern \cite{BechtoldHarangKern2023}, Das, \L{}ochowski, Matsuda and Perkowski \cite{DasLochowskiMatsudaPerkowski2023} and Matsuda and Perkowski \cite{MatsudaPerkowski2024}).

Sewing lemmas are typically written in terms of a control function which captures the regularity of the underlying paths. It turns out that one of the most significant barriers to extending the stochastic sewing lemma to the c\`adl\`ag setting is to be able to incorporate more than one such control function. Accordingly, our first main result (Theorem~\ref{theorem: mild stochastic sewing lemma}) is a novel stochastic sewing lemma, which can simultaneously handle multiple control functions, and which can also incorporate jumps. It may be viewed as a generalization, either of the generalized sewing lemma in \cite[Theorem~2.5]{FrizZhang2018} to the stochastic setting, or of the stochastic sewing lemma in \cite[Theorem~2.9]{FrizHocquetLe2021} to the c\`adl\`ag setting.

With this new c\`adl\`ag version of the stochastic sewing lemma at hand, we introduce notions of stochastic controlled paths and rough stochastic integration, analogous to those in \cite{FrizHocquetLe2021}. This relies on the identification of suitable $p$-variation-type norms for stochastic processes. Our second main result then establishes the existence, uniqueness and stability of solutions to the RSDE in \eqref{eq: RSDE introduction}, driven by a c\`adl\`ag martingale and a c\`adl\`ag rough path. Moreover, we provide the fundamental results of It\^o calculus, including the existence of a bracket process and an associated It\^o formula for stochastic controlled paths with the canonical jump corrections.

We note that our approach to incorporating jumps follows a generalization of the ``forward'' rough integral of \cite{FrizZhang2018}. It is conceivable that the Marcus canonical RDE approach of \cite{ChevyrevFriz2019} may also be applicable, but this is beyond the scope of the present work.

The paper is organized as follows. In Section~\ref{section: Preliminaries} we introduce the relevant $p$-variation spaces for stochastic processes, along with the various properties thereof required for the subsequent analysis. In Section~\ref{section: stochastic sewing lemma} we present our stochastic sewing lemma, and then proceed to define rough stochastic integration in our c\`adl\`ag framework in Section~\ref{section: rough stochastic integration}. Our main result on solutions to RSDEs is then presented in Theorem~\ref{theorem: existence and estimates for solutions to RSDEs}. In Section~\ref{section: rough stochastic calculus} we include some results on the wider calculus of stochastic controlled paths, and the proof of the stochastic sewing lemma is then given in Section~\ref{section: proof of sewing lemma}.

\section{Preliminaries}\label{section: Preliminaries}

\subsection{Frequently used notation}

Given an interval $I \subset [0,\infty)$, we write $\Delta_I := \{(s,t) \in I^2 : s \leq t\}$. For $0 < T < \infty$, we call a function $w \colon \Delta_{[0,T]} \to [0,\infty)$ a \emph{control} if it is superadditive, in the sense that $w(s,u) + w(u,t) \leq w(s,t)$ for all $s \leq u \leq t$. We recall that, for any control $w$, the map $s \mapsto w(s,t)$ is non-increasing for every $t$, the map $t \mapsto w(s,t)$ is non-decreasing for every $s$, and $w(t,t) = 0$ for all $t$.

Given a control $w$, we write $w(s,t+) := \lim_{u \searrow t} w(s,u)$ and
\begin{equation*}
w(s,t-) := \begin{cases}
\lim_{u \nearrow t} w(s,u) & \text{if } \, s < t,\\
0 & \text{if } \, s = t,
\end{cases}
\end{equation*}
with $w(s-,t)$ and $w(s+,t)$ defined analogously, as well as
\begin{equation*}
w(s+,t-) := \begin{cases}
\lim_{v \nearrow t} \lim_{u \searrow s} w(u,v) & \text{if } \, s < t,\\
0 & \text{if } \, s = t.
\end{cases}
\end{equation*}

We say that a control $w$ is \emph{right-continuous in its first argument} if $w(s+,t) = w(s,t)$ for every $(s,t) \in \Delta_{[0,T]}$, and is \emph{right-continuous in its second argument} if $w(s,t+) = w(s,t)$ for every $(s,t) \in \Delta_{[0,T]}$. If $w$ is right-continuous in both its arguments, then we will simply say that it is \emph{right-continuous}. We also say that a control $w$ is \emph{continuous from the inside} if $w(s+,t) = w(s,t) = w(s,t-)$ for all $(s,t) \in \Delta_{[0,T]}$.

\smallskip

By a \emph{partition} $\cP$ of a given interval $[s,t]$, we mean a finite sequence of times $\cP = \{s = t_0 < t_1 < \cdots < t_n = t\}$. We also denote by $|\cP| := \max_{0 \leq i < n} |t_{i+1} - t_i|$ the mesh size of a partition $\cP$.

If $\cP_1$ and $\cP_2$ are two partitions such that $\cP_1 \subseteq \cP_2$ then we say that $\cP_2$ is a \emph{refinement} of $\cP_1$. We say that a sequence of partitions $(\cP^n)_{n \in \N}$ is \emph{nested} if $\cP^n \subseteq \cP^{n+1}$ for every $n \in \N$.

For $u < v$, we write $\cP|_{[u,v]} := \cP \cap [u,v]$ for the restriction of $\cP$ to those times which lie in the subinterval $[u,v]$. We will also sometimes abuse notation slightly by writing $[u,v] \in \cP$ to identify two consecutive times $u, v \in \cP$, i.e., when $u = t_i$ and $v = t_{i+1}$ for some $i$.

\smallskip

Given a path $A \colon [0,T] \to E$, taking values in any normed vector space $(E,|\cdot|)$, we write $\delta A$ for its increment function, so that
\begin{equation*}
\delta A_{s,t} = A_t - A_s
\end{equation*}
for all $(s,t) \in \Delta_{[0,T]}$.

For $p \in [1,\infty)$, given a two-parameter function $F \colon \Delta_{[0,T]} \to E$, the \emph{$p$-variation} of $F$ over the interval $[s,t] \subseteq [0,T]$ is defined as
\begin{equation*}
\|F\|_{p,[s,t]} := \bigg(\sup_{\cP \subset [s,t]} \sum_{[u,v] \in \cP} |F_{u,v}|^p\bigg)^{\hspace{-2pt}\frac{1}{p}},
\end{equation*}
where the supremum is taken over all possible partitions $\cP$ of the interval $[s,t]$. We say that $F$ has finite $p$-variation if $\|F\|_{p,[0,T]} < \infty$. We will sometimes also consider $p$-variation over half-open intervals, e.g., $\|F\|_{p,[s,t)} := \lim_{u \nearrow t} \|F\|_{p,[s,u]}$.

Given a path $A \colon [0,T] \to E$, the $p$-variation of $A$ over $[s,t]$ is defined as the $p$-variation of its increment process, i.e., $\|A\|_{p,[s,t]} := \|\delta A\|_{p,[s,t]}$.

For $p \in [1,\infty)$, we write $V^p = V^p([0,T];E)$ for the space of (deterministic) c\`adl\`ag paths with finite $p$-variation.

\smallskip

For $k \in \N$ and a function $f$, we write $\D^k f$ for the $k$-th order Fr\'echet derivative of $f$. We write $C^n_b$ for the space of functions $f$ which are $(n-1)$-times continuously differentiable, such that $f$ and all its derivatives up to order $n-1$ are bounded, and such that the $(n-1)$-th order derivative $\D^{n-1} f$ is Lipschitz continuous. We define a corresponding norm by
\begin{equation*}
\|f\|_{C^n_b} := \sum_{k=0}^{n-1} \sup_x \|\D^k f(x)\| + \sup_{x \neq y} \frac{\|\D^{n-1} f(x) - \D^{n-1} f(y)\|}{|x - y|}.
\end{equation*}
We will also write, e.g., $\cL(\R^\ell;\R^m)$ for the space of linear maps from $\R^\ell \to \R^m$.

\smallskip

During proofs, we will often use the symbol $\lesssim$ to indicate inequality up to a multiplicative constant. When deriving estimates, this implicit constant will depend on the same variables as the constant specified in the statement of the corresponding estimate.

\subsection{Mixed moment norms for random variables}

We will assume throughout that we have a fixed finite time horizon $T > 0$, and an underlying filtered probability space $(\Omega,\cF,(\cF_t)_{t \in [0,T]},\P)$. Except where stated otherwise, the filtration $(\cF_t)_{t \in [0,T]}$ will \emph{not} be assumed to satisfy the usual conditions. All random variables $Y$ in this section will be assumed to take values in some normed vector space $(E,|\cdot|)$.

\smallskip

For $r \in [1,\infty]$, we write $\|\cdot\|_{L^r}$ for the standard Lebesgue norm on $(\Omega,\cF,\P)$, i.e., $\|Y\|_{L^r} = \E[|Y|^r]^{\frac{1}{r}}$ when $r \in [1,\infty)$, and $\|Y\|_{L^\infty} = \esssup |Y|$.

We write $\E_s[\, \cdot \,] := \E[\, \cdot \,|\,\cF_s]$ for the conditional expectation at time $s \in [0,T]$.

\begin{definition}
For $q \in [1,\infty)$, $r \in [q,\infty]$, $s \in [0,T]$ and any random variable $Y$, we define
\begin{equation*}
\|Y\|_{q,r,s} := \big\|\E_s[|Y|^q]^{\frac{1}{q}}\big\|_{L^r}
\end{equation*}
and write $L^{q,r}_s$ for the space of random variables $Y$ such that $\|Y\|_{q,r,s} < \infty$.
\end{definition}

In particular, we see that $\|\cdot\|_{q,q,s} = \|\cdot\|_{L^q}$, so that $L^{q,q}_s = L^q$. The following proposition recalls various useful properties of the $\|\cdot\|_{q,r,s}$ norm, the proofs of which are either straightforward or may be found (in slightly different notation) in \cite[Proposition~2.3 and Lemma~2.4]{FrizHocquetLe2021}.

\begin{proposition}\label{prop: Lqrs is Banach space}
Let $q \in [1,\infty)$, $r \in [q,\infty]$, $s, t \in [0,T]$, and let $Y, Z$ be random variables.
\begin{enumerate}[(i)]
\item We have that
\begin{equation}\label{eq:q,r,s embeds into L^q}
\|Y\|_{L^q} \leq \|Y\|_{q,r,s} \leq \|Y\|_{L^{r}}.
\end{equation}

\item If $s \leq t$, then
\begin{equation*}
\|Y\|_{q,r,s} \leq \|Y\|_{q,r,t}.
\end{equation*}

\item We have that
\begin{equation}\label{eq: L^qr norm of E_s bounded by q,r,s}
\big\| \E_s[Y] \big\|_{L^{r}} \leq \|Y\|_{q,r,s}.
\end{equation}

\item The following version of H\"older's inequality,
\begin{equation}\label{eq:Hoelder's inequality for q,r,s norm}
\big\|\E_s[|Y||Z|]\big\|_{L^\ell} \leq \big\|\E_s[|Y|^p]^{\frac{1}{p}}\big\|_{L^{p\ell}} \big\|\E_s[|Z|^q]^{\frac{1}{q}}\big\|_{L^{q\ell}} = \|Y\|_{p,p\ell,s} \|Z\|_{q,q\ell,s},
\end{equation}
holds for any $\ell \in [1,\infty]$ whenever $p, q \in (1,\infty)$ with $\frac{1}{p} + \frac{1}{q} = 1$.

\item We have the following version of Fatou's lemma. If $(Y_n)_{n \in \N}$ is a sequence of non-negative random variables, then
\begin{equation*}
\big\|\liminf_{n \to \infty} Y_n\big\|_{q,r,s} \leq \liminf_{n \to \infty} \|Y_n\|_{q,r,s}.
\end{equation*}

\item The space $L^{q,r}_s$ becomes a Banach space when equipped with the norm $\|\cdot\|_{q,r,s}$ (where as usual we identify random variables as being the same if they are equal almost surely).
\end{enumerate}
\end{proposition}

We note, by Jensen's inequality and part~(ii) of Proposition~\ref{prop: Lqrs is Banach space}, that, for any random variable $Y$, the norm $\|Y\|_{q,r,s}$ is non-decreasing in each of the variables $q, r$ and $s$.

\begin{lemma}\label{lemma: closedness of set of bounded L^q,infinity random variables in L^q}
For any $q \in [1,\infty)$, $s \in [0,T]$ and $\nu > 0$, the set
$$\cB^{q,\infty}_s(\nu) := \{Y \in L^{q,\infty}_s : \|Y\|_{q,\infty,s} \leq \nu\}$$
is a closed subset of $L^q$.
\end{lemma}

\begin{proof}
Let $Y \in L^q$, and let $(Y^n)_{n \in \N}$ be a sequence in $\cB^{q,\infty}_s(\nu)$ such that $Y^n \to Y$ in $L^q$. Using the fact that $|b^q - a^q| = |\int_a^b q x^{q-1} \dd x| \leq q |b - a| (a^{q-1} + b^{q-1})$ for any $a, b \geq 0$, and H{\"o}lder's inequality, we have that
\begin{align*}
\E \big[\big| \E_s[|Y^n|^q] - \E_s[|Y|^q] \big|\big] &\leq q \E \big[\big|\E_s\big[|Y^n - Y| \big(|Y^n|^{q-1} + |Y|^{q-1}\big)\big]\big|\big]\\
&\leq q \E \big[ |Y^n - Y| \big(|Y^n|^{q-1} + |Y|^{q-1}\big) \big]\\
&\leq q \|Y^n - Y\|_{L^q} \big(\|Y^n\|^{q-1}_{L^q} + \|Y\|^{q-1}_{L^q}\big).
\end{align*}
Since $Y^n \to Y$ in $L^q$, we see that $\E_s[|Y^n|^q] \to \E_s[|Y|^q]$ in $L^1$, and hence also in probability.

For each $n \in \N$, we have that, almost surely, $\E_s[|Y^n|^q] \leq \|Y^n\|_{q,\infty,s}^q \leq \nu^q$. It follows that, for each $\epsilon > 0$,
\begin{equation*}
\P\big(\E_s[|Y|^q] > \nu^q + \epsilon\big) \leq \P\big(\big| \E_s[|Y^n|^q] - \E_s[|Y|^q] \big| > \epsilon\big) \, \longrightarrow \, 0 \qquad \text{as} \quad n \, \longrightarrow \, \infty.
\end{equation*}
Since then $\P(\E_s[|Y|^q] > \nu^q + \epsilon) = 0$ for every $\epsilon > 0$, we have that $\P(\E_s[|Y|^q] \leq \nu^q) = 1$, which means precisely that $Y \in \cB^{q,\infty}_s(\nu)$.
\end{proof}

\subsection{Variation norms for stochastic processes}

We say that a stochastic process $Y = (Y_t)_{t \in [0,T]}$ is \emph{right-continuous in probability}, if $Y_t \to Y_s$ in probability as $t \searrow s$, for every (deterministic) time $s \in [0,T)$.

\begin{definition}\label{defn: V^pL^q and V^pL^q,r spaces}
Let $p, q \in [1,\infty)$ and $r \in [q,\infty]$. Given a two-parameter stochastic process $F = (F_{s,t})_{(s,t) \in \Delta_{[0,T]}}$, we define
\begin{equation*}
\|F\|_{p,q,[s,t]} := \bigg(\sup_{\cP \subset [s,t]} \sum_{[u,v] \in \cP} \|F_{u,v}\|_{L^q}^p\bigg)^{\hspace{-2pt}\frac{1}{p}} = \bigg(\sup_{\cP \subset [s,t]} \sum_{[u,v] \in \cP} \E[|F_{u,v}|^q]^\frac{p}{q}\bigg)^{\hspace{-2pt}\frac{1}{p}}
\end{equation*}
and
\begin{equation*}
\|F\|_{p,q,r,[s,t]} := \bigg(\sup_{\cP \subset [s,t]} \sum_{[u,v] \in \cP} \|F_{u,v}\|_{q,r,u}^p\bigg)^{\hspace{-2pt}\frac{1}{p}} = \bigg(\sup_{\cP \subset [s,t]} \sum_{[u,v] \in \cP} \big\|\E_u[|F_{u,v}|^q]^{\frac{1}{q}}\big\|_{L^r}^{p}\bigg)^{\hspace{-2pt}\frac{1}{p}}
\end{equation*}
for $(s,t) \in \Delta_{[0,T]}$. For a stochastic process $Y = (Y_t)_{t \in [0,T]}$, we then let
\begin{equation*}
\|Y\|_{p,q,[s,t]} := \|\delta Y\|_{p,q,[s,t]} \qquad \text{and} \qquad \|Y\|_{p,q,r,[s,t]} := \|\delta Y\|_{p,q,r,[s,t]}.
\end{equation*}
We also write $\|Y\|_{p,q,[s,t)} := \lim_{u \nearrow t} \|Y\|_{p,q,[s,u]}$ and $\|Y\|_{p,q,r,[s,t)} := \lim_{u \nearrow t} \|Y\|_{p,q,r,[s,u]}$.

We write $V^p L^q = V^p L^q([0,T];E)$ for the space of stochastic processes $Y \colon \Omega \times [0,T] \to E$ which are right-continuous in probability and are such that $Y_0 \in L^q$ and $\|Y\|_{p,q,[0,T]} < \infty$. Similarly, we write $V^p L^{q,r} = V^p L^{q,r}([0,T];E)$ for the space of stochastic processes $Y \colon \Omega \times [0,T] \to E$ which are right-continuous in probability and are such that $Y_0 \in L^q$ and $\|Y\|_{p,q,r,[0,T]} < \infty$.
\end{definition}

In the case when $r = q$, we simply have that $\|Y\|_{p,q,q,[s,t]} = \|Y\|_{p,q,[s,t]}$, so that $V^p L^{q,q} = V^p L^q$. For general $r \in [q,\infty]$, it follows from \eqref{eq:q,r,s embeds into L^q}, that
\begin{equation}\label{eq: p,q norm less than p,q,r}
\|Y\|_{p,q,[s,t]} \leq \|Y\|_{p,q,r,[s,t]},
\end{equation}
so that $V^p L^{q,r} \subseteq V^p L^q$. Moreover, the map $Y \mapsto \|Y_0\|_{L^q} + \|Y\|_{p,q,[0,T]}$ (resp.~$Y \mapsto \|Y_0\|_{L^q} + \|Y\|_{p,q,r,[0,T]}$) is a norm on $V^p L^q$ (resp.~on $V^p L^{q,r}$), where we identify processes as being the same if they are a modification of each other.

\begin{remark}
It is sufficient to consider only $V^p L^q$ spaces if one is interested in rough stochastic differential equations with linear vector fields; see, e.g., \cite{BuginiCoghiNilssen2024} or \cite{LiangTang2025}. However, as we will see in detail later, the composition of processes with nonlinear vector fields leads to a loss of integrability, which renders the use of the mixed moment spaces $V^p L^{q,r}$ essential.
\end{remark}

\begin{example}\label{examples: BM in mixed moment spaces}
For Brownian motion $W = (W_t)_{t \in [0,T]}$, we have that $\|\delta W_{s,t}\|_{q,\infty,s} = \|\delta W_{s,t}\|_{L^q} \lesssim (t - s)^{\frac{1}{2}}$, so that $W \in V^2 L^{q,\infty}$ for any $q \in [1,\infty)$.

More generally, for any L\'evy process $L = (L_t)_{t \in [0,T]}$ with $\E [L_1^2] < \infty$, we have (from, e.g., \cite[Proposition~2.3]{BrockwellSchlemm2012}) that $\|\delta L_{s,t}\|_{2,\infty,s} = \|\delta L_{s,t}\|_{L^2} \lesssim (t - s)^{\frac{1}{2}}$, so that $L \in V^2 L^{2,\infty}$.
\end{example}

We will discuss further classes of processes which live in $V^p L^{q,r}$ spaces in Section~\ref{section: martingales in mixed moment spaces}.

\begin{lemma}\label{lemma: V^pL^q implies L^q-regulated}
Let $p, q \in [1,\infty)$ and $r \in [q,\infty]$, and let $Y$ be a stochastic process such that $\|Y\|_{p,q,r,[0,T]} < \infty$. Then $Y$ is \emph{$L^{q,r}$-regulated}, by which we mean that, for every $s \in [0,T)$, the map
$$[s,T] \ni t \mapsto \delta Y_{s,t}$$
admits left and right limits in $L^{q,r}_s$.
\end{lemma}

\begin{proof}
We recall from part~(vi) of Proposition~\ref{prop: Lqrs is Banach space} that $L^{q,r}_s$ is a Banach space. To show the existence of right limits, it thus suffices to show that for any $0 \leq s \leq t < T$, and any sequence $(u_k)_{k \in \N}$ of times with $u_k \searrow t$ as $k \to \infty$, the sequence of random variables $(\delta Y_{s,u_k})_{k \in \N}$ is a Cauchy sequence in $L^{q,r}_s$, and that the corresponding limit is independent of the choice of sequence $(u_k)_{k \in \N}$.

To this end, let us assume for a contradiction that there exists such a sequence $(u_k)_{k \in \N}$ with $u_k \searrow t$, such that $(\delta Y_{s,u_k})_{k \in \N}$ is not Cauchy in $L^{q,r}_s$. This implies that there exists an $\epsilon > 0$, and sequences $(m_j)_{j \in \N}$ and $(n_j)_{j \in \N}$, such that $m_j < n_j < m_{j+1}$ and
\begin{equation*}
\|\delta Y_{u_{n_j},u_{m_j}}\|_{q,r,s} > \epsilon
\end{equation*}
for each $j \in \N$. Then, recalling part~(ii) of Proposition~\ref{prop: Lqrs is Banach space}, we have that
\begin{equation*}
\|Y\|_{p,q,r,[0,T]}^p \geq \sum_{j=1}^\infty \|\delta Y_{u_{n_j},u_{m_j}}\|_{q,r,u_{n_j}}^p \geq \sum_{j=1}^\infty \|\delta Y_{u_{n_j},u_{m_j}}\|_{q,r,s}^p = \infty,
\end{equation*}
which gives the desired contradiction.

Thus, for every sequence $(u_k)_{k \in \N}$ with $u_k \searrow t$ as $k \to \infty$, the sequence $(\delta Y_{s,u_k})_{k \in \N}$ is Cauchy, and hence convergent in $L^{q,r}_s$. As we may always combine two such sequences of times to obtain a new sequence, it is immediate that the limit is independent of the choice of sequence $(u_k)_{k \in \N}$. The existence of left limits for $t \in (s,T]$ can be shown by an almost identical argument, by considering sequences of times such that $u_k \nearrow t$ as $k \to \infty$.
\end{proof}

\begin{remark}\label{remark: V^p L^q,r processes are cadlag in L^q}
Since $\|\cdot\|_{L^q} = \|\cdot\|_{q,q,s} \leq \|\cdot\|_{q,r,s}$, the result of Lemma~\ref{lemma: V^pL^q implies L^q-regulated} also implies that, for any $p, q \in [1,\infty)$, $r \in [q,\infty]$ and any $Y \in V^p L^{q,r}$, the process $Y$ is \emph{$L^q$-regulated}, by which we mean that the map $t \mapsto Y_t$ admits left and right limits in $L^q$.
\end{remark}

\begin{lemma}\label{lemma: V^pL^q is a Banach space}
For any $p, q \in [1,\infty)$ and $r \in [q,\infty]$, the space $V^p L^{q,r}$ (as defined in Definition~\ref{defn: V^pL^q and V^pL^q,r spaces}) is a Banach space, with the norm given by
\begin{equation}\label{eq: norm on V^p L^q}
Y \, \mapsto \, \|Y_0\|_{L^q} + \|Y\|_{p,q,r,[0,T]}.
\end{equation}
Moreover, any process $Y \in V^p L^{q,r}$ is \emph{c\`adl\`ag in $L^q$}, by which we mean that the map $t \mapsto Y_t$ is a c\`adl\`ag path from $[0,T] \to L^q$.
\end{lemma}

\begin{proof}
It is straightforward to check that the map defined in \eqref{eq: norm on V^p L^q} is indeed a norm on $V^p L^{q,r}$, where we identify processes as being the same if they are a modification of each other.

Let $(Y^n)_{n \in \N}$ be a Cauchy sequence with respect to the norm in \eqref{eq: norm on V^p L^q}. We then have that the sequence $(Y^n_0)_{n \in \N}$ is Cauchy in $L^q$, and, for any $(s,t) \in \Delta_{[0,T]}$, the sequence $(\delta Y^n_{s,t})_{n \in \N}$ is Cauchy in $L^{q,r}_s$. We can then define a new process $Y$, by letting
$$Y_t := \lim_{n \to \infty} Y^n_0 + \lim_{n \to \infty} \delta Y^n_{0,t}.$$
It is then also straightforward to see that $\delta Y_{s,t} = \lim_{n \to \infty} \delta Y^n_{s,t}$ as a limit in $L^{q,r}_s$, for every pair $(s,t) \in \Delta_{[0,T]}$.

The seminorm $\|\cdot\|_{p,q,r,[0,T]}$ is, by definition, the composition of the $p$-variation norm with the time-parameterized family of norms $L^{q,r}_{\edot} = (L^{q,r}_s)_{s \in [0,T]}$. Using the lower semi-continuity of the $p$-variation norm (see, e.g., \cite[(P7) in Sect.~2]{ChistyakovGalkin1998}\footnote{Strictly speaking, here we use a slightly more general version of this lower semi-continuity result, since we have a time-parameterized family of norms, but the same proof works in this case almost verbatim.}), we then have that
\begin{align*}
\|Y^n - Y\|_{p,q,r,[0,T]} &= \big\| \|(\delta Y^n - \delta Y)_{(\cdot,\cdot)}\|_{L^{q,r}_{\edot}} \big\|_{p,[0,T]} = \big\| \lim_{m \to \infty} \|(\delta Y^n - \delta Y^m)_{(\cdot,\cdot)}\|_{L^{q,r}_{\edot}} \big\|_{p,[0,T]}\\
&\leq \liminf_{m \to \infty} \big\| \|(\delta Y^n - \delta Y^m)_{(\cdot,\cdot)}\|_{L^{q,r}_{\edot}} \big\|_{p,[0,T]}\\
&= \liminf_{m \to \infty} \|Y^n - Y^m\|_{p,q,r,[0,T]} \, \longrightarrow \, 0 \qquad \text{as} \quad n \, \longrightarrow \, \infty,
\end{align*}
from which it follows, both that $\|Y\|_{p,q,r,[0,T]} < \infty$, and that $Y^n \to Y$ with respect to the norm in \eqref{eq: norm on V^p L^q}.

It remains to check that the limiting process $Y$ is also right-continuous in probability. Let $\eta > 0$ and $\epsilon > 0$. Since
\begin{equation*}
\sup_{(s,t) \in \Delta_{[0,T]}} \eta^q \P \big(|\delta Y^n_{s,t} - \delta Y_{s,t}| > \eta\big) \leq \sup_{(s,t) \in \Delta_{[0,T]}} \|\delta Y^n_{s,t} - \delta Y_{s,t}\|_{L^q}^q \leq \|Y^n - Y\|_{p,q,r,[0,T]}^q,
\end{equation*}
we can fix an $n \in \N$ sufficiently large such that
\begin{equation*}
\sup_{(s,t) \in \Delta_{[0,T]}} \P \big(|\delta Y^n_{s,t} - \delta Y_{s,t}| > \eta\big) < \frac{\epsilon}{2}.
\end{equation*}
Since $Y^n$ is right-continuous in probability, for any $s \in [0,T)$, there exists a $v \in (s,T]$ such that
\begin{equation*}
\P \big(|\delta Y^n_{s,t}| > \eta\big) < \frac{\epsilon}{2}
\end{equation*}
for every $t \in (s,v]$. Then, for any such $t \in (s,v]$, we have that
\begin{equation*}
\P \big(|\delta Y_{s,t}| > 2\eta\big) \leq \P \big(|\delta Y^n_{s,t}| > \eta\big) + \P \big(|\delta Y^n_{s,t} - \delta Y_{s,t}| > \eta\big) < \epsilon.
\end{equation*}
We have thus established that $\P (|Y_t - Y_s| > 2\eta) \to 0$ as $t \searrow s$ for every $s \in [0,T)$ and every $\eta > 0$, which means precisely that $Y$ is right-continuous in probability. It follows that $V^p L^{q,r}$ is indeed a Banach space.

Finally, we recall from Remark~\ref{remark: V^p L^q,r processes are cadlag in L^q} that any process $Y \in V^p L^{q,r}$ is $L^q$-regulated, so that in particular $\lim_{t \searrow s} Y_t$ exists as a limit in $L^q$ for every $s \in [0,T)$. Since $Y$ is right-continuous in probability, it follows from the uniqueness of limits that actually $Y_t \to Y_s$ in $L^q$ as $t \searrow s$. Combined with the existence of left limits in $L^q$, this implies that $Y$ is c\`adl\`ag in $L^q$.
\end{proof}

\begin{lemma}\label{lemma: closedness of set of bounded paths in V^pL^q}
For any $p, q \in [1,\infty)$ and $\nu > 0$, the set
$$\cB^{p,q,\infty}(\nu) := \{Y \in V^p L^{q,\infty} : \|Y\|_{p,q,\infty,[0,T]} \leq \nu\}$$
is a closed subset of $V^p L^q$.
\end{lemma}

\begin{proof}
It is immediate from \eqref{eq: p,q norm less than p,q,r} that $\cB^{p,q,\infty}(\nu)$ is a subset of $V^p L^q$. Let $Y \in V^p L^q$, and let $(Y^n)_{n \in \N}$ be a sequence in $\cB^{p,q,\infty}(\nu)$ such that $\lim_{n \to \infty} \|Y^n - Y\|_{p,q,[0,T]} = 0$.

It follows from the definition of the seminorm $\|\cdot\|_{p,q,[0,T]}$ that, for any interval $[u,v] \subseteq [0,T]$,
\begin{equation}\label{eq: L^q convergence of V^p L^q processes}
\lim_{n \to \infty} \|\delta Y^n_{u,v} - \delta Y_{u,v}\|_{L^q} = 0.
\end{equation}
Assume for a contradiction that $\|Y\|_{p,q,\infty,[0,T]}^p > \nu^p + \epsilon$ for some $\epsilon > 0$. Then there exists a partition $\cP$ of $[0,T]$ such that $\sum_{[u,v] \in \cP} \|\delta Y_{u,v}\|_{q,\infty,u}^p > \nu^p + \epsilon$. For each $n \in \N$, since
\begin{equation*}
\sum_{[u,v] \in \cP} \|\delta Y^n_{u,v}\|_{q,\infty,u}^p \leq \|Y^n\|_{p,q,\infty,[0,T]}^p \leq \nu^p,
\end{equation*}
we have that
\begin{equation*}
\sum_{[u,v] \in \cP} \Big(\|\delta Y_{u,v}\|_{q,\infty,u}^p - \|\delta Y^n_{u,v}\|_{q,\infty,u}^p\Big) > \epsilon.
\end{equation*}
Thus, for each $n \in \N$, there exists at least one interval $[u,v] \in \cP$ such that
\begin{equation*}
\|\delta Y_{u,v}\|_{q,\infty,u}^p - \|\delta Y^n_{u,v}\|_{q,\infty,u}^p > \frac{\epsilon}{N},
\end{equation*}
where $N$ is the number of intervals in the partition $\cP$. By the pigeon hole principle, we can fix a single interval $[u,v] \in \cP$ and a subsequence $(n_k)_{k \in \N}$, such that
\begin{equation*}
\|\delta Y_{u,v}\|_{q,\infty,u}^p - \|\delta Y^{n_k}_{u,v}\|_{q,\infty,u}^p > \frac{\epsilon}{N}
\end{equation*}
for every $k \in \N$. Setting $\widetilde{\nu} := \sup_{k \in \N} \|\delta Y^{n_k}_{u,v}\|_{q,\infty,u}$, it follows that
\begin{equation*}
\|\delta Y_{u,v}\|^p_{q,\infty,u} \geq \widetilde{\nu}^p + \frac{\epsilon}{N}.
\end{equation*}
We have from \eqref{eq: L^q convergence of V^p L^q processes} that $\delta Y^{n_k}_{u,v} \to \delta Y_{u,v}$ in $L^q$ as $k \to \infty$. Since $\|\delta Y^{n_k}_{u,v}\|_{q,\infty,u} \leq \widetilde{\nu}$ for every $k \in \N$, it then follows from Lemma~\ref{lemma: closedness of set of bounded L^q,infinity random variables in L^q} that
$$\Big(\widetilde{\nu}^p + \frac{\epsilon}{N}\Big)^{\hspace{-1pt}\frac{1}{p}} \leq \|\delta Y_{u,v}\|_{q,\infty,u} \leq \widetilde{\nu},$$
which gives the desired contradiction.
\end{proof}

\begin{lemma}\label{lemma: right-continuity of controls}
Let $Y \in V^p L^{q,r}$ for any $p, q \in [1,\infty)$ and $r \in [q,\infty]$, and let $w$ be the control defined by $w(s,t) = \|Y\|_{p,q,r,[s,t]}^p$ for $(s,t) \in \Delta_{[0,T]}$. Then $w$ is right-continuous.
\end{lemma}

\begin{proof}
We will first show right-continuity in the second argument, arguing similarly to the proof of \cite[Lemma~7.1]{FrizZhang2018}. Let $0 \leq s \leq t < T$ and $\epsilon \in (0,1]$.

By Lemma~\ref{lemma: V^pL^q implies L^q-regulated}, we have that $\lim_{v \searrow t} \delta Y_{t,v}$ exists as a limit in $L^{q,r}_t$, and, since $Y$ is right-continuous in probability (by the definition of the space $V^p L^{q,r}$), this limit is equal to $0$. We can thus find an $\eta > 0$ such that
\begin{equation}\label{eq: delta Y t,v q,r,t less epsilon}
\|\delta Y_{t,v}\|_{q,r,t} < \epsilon
\end{equation}
for all $v \in [t,t + \eta]$. We now fix a partition $\cP$ of the interval $[s,t + \eta]$ such that
\begin{equation}\label{eq: approx partition for Y p,q,r,[s,t+delta]}
\sum_{[u,v] \in \cP} \|\delta Y_{u,v}\|_{q,r,u}^p > \|Y\|_{p,q,r,[s,t + \eta]}^p - \epsilon.
\end{equation}
We now consider adding time $t$ into the partition. Let $t' \in \cP$ be the last time in $\cP$ such that $t' < t$, and let $t''$ be the first time in $\cP$ such that $t'' > t$. Then, using part~(ii) of Proposition~\ref{prop: Lqrs is Banach space}, we have that
\begin{equation*}
\|\delta Y_{t',t''}\|_{q,r,t'}^p = \|\delta Y_{t',t} + \delta Y_{t,t''}\|_{q,r,t'}^p \leq \big(\|\delta Y_{t',t}\|_{q,r,t'} + \|\delta Y_{t,t''}\|_{q,r,t}\big)^p \leq \|\delta Y_{t',t}\|_{q,r,t'}^p + C \epsilon,
\end{equation*}
where in the last inequality we used \eqref{eq: delta Y t,v q,r,t less epsilon}, as well as the bound $(a + b)^p \leq a^p + C \epsilon$, which holds whenever $0 \leq a \leq \|Y\|_{p,q,r,[0,T]}$ and $0 \leq b \leq \epsilon$, and the constant $C$ depends only on $p$ and $\|Y\|_{p,q,r,[0,T]}$. This then allows us to estimate
\begin{equation}\label{eq: bound after inserting time t into partition}
\sum_{[u,v] \in \cP} \|\delta Y_{u,v}\|_{q,r,u}^p \leq \sum_{[u,v] \in \cP \cup \{t\}|_{[s,t]}} \|\delta Y_{u,v}\|_{q,r,u}^p + C \epsilon + \sum_{[u,v] \in \cP|_{[t'',t + \eta]}} \|\delta Y_{u,v}\|_{q,r,u}^p.
\end{equation}
Using the superadditivity of the control $w$, and the inequalities in \eqref{eq: approx partition for Y p,q,r,[s,t+delta]} and \eqref{eq: bound after inserting time t into partition}, we then have, for any $0 < h < \eta$, that
\begin{align*}
\|Y&\|_{p,q,r,[s,t + h]}^p \leq \|Y\|_{p,q,r,[s,t + \eta]}^p - \|Y\|_{p,q,r,[t + h,t + \eta]}^p\\
&< \sum_{[u,v] \in \cP} \|\delta Y_{u,v}\|_{q,r,u}^p + \epsilon - \|Y\|_{p,q,r,[t + h,t + \eta]}^p\\
&\leq \sum_{[u,v] \in \cP \cup \{t\}|_{[s,t]}} \|\delta Y_{u,v}\|_{q,r,u}^p + (C + 1) \epsilon + \sum_{[u,v] \in \cP|_{[t'',t + \eta]}} \|\delta Y_{u,v}\|_{q,r,u}^p - \|Y\|_{p,q,r,[t + h,t + \eta]}^p\\
&\leq \|Y\|_{p,q,r,[s,t]}^p + (C + 1) \epsilon + \|Y\|_{p,q,r,[t'',t + \eta]}^p - \|Y\|_{p,q,r,[t + h,t + \eta]}^p.
\end{align*}
We thus infer that $\lim_{h \searrow 0} \|Y\|_{p,q,r,[s,t + h]}^p \leq \|Y\|_{p,q,r,[s,t]}^p + (C + 1) \epsilon$, and by taking $\epsilon \to 0$ we obtain $w(s,t+) \leq w(s,t)$, with the reverse inequality being clear.

We now turn to showing right-continuity in the first argument. Let us fix $0 \leq s < t \leq T$. It is straightforward to see that
\begin{equation*}
\|Y\|_{p,q,r,(s,t]} := \lim_{u \searrow s} \|Y\|_{p,q,r,[u,t]} = \bigg(\sup_{\cP \subset (s,t]} \sum_{[u,v] \in \cP} \|\delta Y_{u,v}\|_{q,r,u}^p\bigg)^{\hspace{-2pt}\frac{1}{p}},
\end{equation*}
where the supremum is taken over all possible finite partitions $\cP = \{t_0 < t_1 < \cdots < t_N\}$ with $t_0 > s$ and $t_N = t$.

Let $\{s = u_0 < u_1 < \cdots < u_N = t\}$ be a partition of the interval $[s,t]$, and let $v$ be another point such that $u_0 < v < u_1$. Using part~(ii) of Proposition~\ref{prop: Lqrs is Banach space}, we have that
\begin{equation*}
\bigg(\sum_{i=0}^{N-1} \|\delta Y_{u_i,u_{i+1}}\|_{q,r,u_i}^p\bigg)^{\hspace{-2pt}\frac{1}{p}} \leq \bigg(\big(\|\delta Y_{u_0,v}\|_{q,r,u_0} + \|\delta Y_{v,u_1}\|_{q,r,v}\big)^p + \sum_{i=1}^{N-1} \|\delta Y_{u_i,u_{i+1}}\|_{q,r,u_i}^p\bigg)^{\hspace{-2pt}\frac{1}{p}}.
\end{equation*}
It then follows from the Minkowski inequality for the sequence space $l^p$ that
\begin{align*}
\bigg(\sum_{i=0}^{N-1} \|\delta Y_{u_i,u_{i+1}}\|_{q,r,u_i}^p\bigg)^{\hspace{-2pt}\frac{1}{p}} &\leq \|\delta Y_{u_0,v}\|_{q,r,u_0} + \bigg(\|\delta Y_{v,u_1}\|_{q,r,v}^p + \sum_{i=1}^{N-1} \|\delta Y_{u_i,u_{i+1}}\|_{q,r,u_i}^p\bigg)^{\hspace{-2pt}\frac{1}{p}}\\
&\leq \|\delta Y_{u_0,v}\|_{q,r,u_0} + \|Y\|_{p,q,r,(s,t]}.
\end{align*}
Since $\|\delta Y_{u_0,v}\|_{q,r,u_0} \to 0$ as $v \searrow u_0$ (by Lemma~\ref{lemma: V^pL^q implies L^q-regulated} and the fact that $Y$ is right-continuous in probability), we deduce that
\begin{equation*}
\bigg(\sum_{i=0}^{N-1} \|\delta Y_{u_i,u_{i+1}}\|_{q,r,u_i}^p\bigg)^{\hspace{-2pt}\frac{1}{p}} \leq \|Y\|_{p,q,r,(s,t]},
\end{equation*}
and thus $w(s,t) = \|Y\|_{p,q,r,[s,t]}^p \leq \|Y\|_{p,q,r,(s,t]}^p = w(s+,t)$, with the reverse inequality being clear.
\end{proof}

\subsection{Martingales in mixed moment spaces}\label{section: martingales in mixed moment spaces}

In the setting of \cite{FrizHocquetLe2021}, the regularity of processes is considered with respect to H\"older spaces $C^\alpha L^{q,r}$ for $\alpha \in (0,1]$, with corresponding norm given by
\[ Y \mapsto \sup_{0 \leq s < t \leq T} \frac{\|\delta Y_{s,t}\|_{q,r,s}}{|t - s|^\alpha}. \]
In Example~\ref{examples: BM in mixed moment spaces} we showed that any square integrable L\'evy process is an element of $V^2 L^{2,\infty}$, and even $C^{\frac{1}{2}} L^{2,\infty}$. However, as we will see below, the space $V^p L^{q,r}$ includes elements beyond L\'evy processes and those in $C^{\frac{1}{p}} L^{q,r}$.

\begin{lemma}\label{lemma: quadratic variation is all that matters for p,q,r}
Let $p \in [2,\infty)$, $q \in [2,\infty)$ and $r \in [q,\infty]$. The maps $M \mapsto \|M\|_{p,q,r,[0,T]}$ and $M \mapsto \|[M]\|_{\frac{p}{2},\frac{q}{2},\frac{r}{2},[0,T]}^{\frac{1}{2}}$ are equivalent seminorms on the corresponding space of c\`adl\`ag martingales.

In particular, if $M$ is a c\`adl\`ag martingale with $M_0 \in L^q$, then $M \in V^p L^{q,r}$ if and only if its quadratic variation satisfies $[M] \in V^{\frac{p}{2}} L^{\frac{q}{2},\frac{r}{2}}$.
\end{lemma}

\begin{proof}
The conditional BDG inequality implies that
\[ \|\delta M_{s,t}\|_{q,r,s} = \big\| \E_s \big[|\delta M_{s,t}|^q\big]^{\frac{1}{q}} \big\|_{L^r} \lesssim \Big\| \E_s \Big[ \big| \delta [M]_{s,t} \big|^{\frac{q}{2}} \Big]^{\frac{1}{q}} \Big\|_{L^r} = \big\| \delta [M]_{s,t} \big\|_{\frac{q}{2},\frac{r}{2},s}^{\frac{1}{2}} \]
for all $(s,t) \in \Delta_{[0,T]}$, so that $\|M\|_{p,q,r,[0,T]} \lesssim \|[M]\|_{\frac{p}{2},\frac{q}{2},\frac{r}{2},[0,T]}^{\frac{1}{2}}$. On the other hand, combining the conditional BDG inequality with the conditional Doob $L^p$ inequality, we have that
\[ \big\| \delta [M]_{s,t} \big\|_{\frac{q}{2},\frac{r}{2},s}^{\frac{1}{2}} \lesssim \Big\| \E_s \Big[ \sup_{u \in [s,t]} |\delta M_{s,u}|^q \Big]^{\frac{1}{q}} \Big\|_{L^r} \lesssim \big\| \E_s \big[ |\delta M_{s,t}|^q \big]^{\frac{1}{q}} \big\|_{L^r} = \| \delta M_{s,t} \|_{q,r,s} \]
for all $(s,t) \in \Delta_{[0,T]}$, so that $\|[M]\|_{\frac{p}{2},\frac{q}{2},\frac{r}{2},[0,T]}^{\frac{1}{2}} \lesssim \|M\|_{p,q,r,[0,T]}$.
\end{proof}

\begin{example}\label{example: ito integral in mixed moment spaces}
It follows easily from Lemma~\ref{lemma: quadratic variation is all that matters for p,q,r} that, given any bounded predictable process $H$ and any c\`adl\`ag martingale $M \in V^p L^{q,r}$, we also have that $\int_0^\cdot H_s \dd M_s \in V^p L^{q,r}$.
\end{example}

\begin{remark}
It was shown in \cite[Lemma~4.4]{LiangTang2025} that if $M = (M_t)_{t \in [0,T]}$ is a continuous martingale with $M_T \in L^p$ for some $p \geq 2$, then $M \in V^p L^p$. This is also true for c\`adl\`ag martingales $M$. Indeed, for any partition $\cP$ of $[0,T]$, we have that
\begin{equation*}
\sum_{[s,t] \in \cP} \big\| \delta [M]_{s,t} \big\|^{\frac{p}{2}}_{L^{\frac{p}{2}}} = \bigg\| \sum_{[s,t] \in \cP} \big| \delta [M]_{s,t} \big|^{\frac{p}{2}} \bigg\|_{L^1} \leq \bigg\| \bigg( \sum_{[s,t] \in \cP} \delta [M]_{s,t} \bigg)^{\hspace{-2pt}\frac{p}{2}} \bigg\|_{L^1} = \big\| [M]_T \big\|_{L^{\frac{p}{2}}}^{\frac{p}{2}},
\end{equation*}
so that $[M] \in V^{\frac{p}{2}} L^{\frac{p}{2}}$, and hence $M \in V^p L^p$ by Lemma~\ref{lemma: quadratic variation is all that matters for p,q,r}.
\end{remark}

In the next result we give a characterization of all martingales in $V^p L^{q,r}$ in terms of their characteristics.\footnote{See, e.g., \cite[II.~Definition~2.6]{JacodShiryaev2006} for a definition and detailed exposition of semimartingale characteristics.}

\begin{proposition}\label{proposition: a criterium for martingales in mixed moment spaces}
Let $p \in [2,\infty)$, $q \in [2,\infty)$ and $r \in [q,\infty]$, and let $M$ be an $\R^d$-valued c\`adl\`ag martingale. Then $M \in V^p L^{q,r}$ if and only if the characteristics $(B,C,\nu)$ of $M$ satisfy $\tr(C) + \int_0^{\cdot} \int_{\R^d} |u|^2 \, \nu(\d s,\d u) \in V^{\frac{p}{2}} L^{\frac{q}{2},\frac{r}{2}}$ and $\int_0^{\cdot} \int_{\R^d} |u|^q \, \nu(\d s,\d u) \in V^{\frac{p}{q}} L^{1,\frac{r}{q}}$.
\end{proposition}

\begin{proof}
By the conditional version (\cite[Lemma~A.2]{AllanPieperTeichmann2025}) of the generalized BDG inequality of \cite{HernandezHernandezJackal2022}, for any $(s,t) \in \Delta_{[0,T]}$, we have that
\begin{equation*}
\| \delta M_{s,t} \|_{q,r,s} \lesssim \Big\|\E_s \Big[ |\delta \langle M \rangle_{s,t}|^{\frac{q}{2}} \vee \delta A^{(\frac{q}{2})}_{s,t} \Big]^{\frac{1}{q}} \Big\|_{L^r} = \Big\|\E_s \Big[ \big|\delta \langle M^c \rangle_{s,t} + \delta A^{(1)}_{s,t} \big|^{\frac{q}{2}} \vee \delta A^{(\frac{q}{2})}_{s,t} \Big]^{\frac{1}{q}}\Big\|_{L^r},
\end{equation*}
where $\langle M \rangle$ denotes the predictable quadratic variation of $M$, and $M^c$ denotes the continuous martingale part of $M$, and, in the notation of \cite{HernandezHernandezJackal2022}, we write
\begin{equation*}
A^{(\ell)} := \Pi^\ast_p \bigg( \sum_{s \leq \cdot} |\Delta M_s|^{2\ell} \bigg),
\end{equation*}
where $\Pi^\ast_p$ denotes the dual predictable projection. By, e.g., \cite[Theorem~13.2.21]{CohenElliott2015}, writing $\mu$ for the jump measure associated with $M$, we have that
\begin{equation*}
A^{(\ell)} = \Pi^\ast_p \bigg( \sum_{s \leq \cdot} |\Delta M_s|^{2\ell} \bigg) = \Pi^\ast_p \bigg( \int_0^\cdot \int_{\R^d} |u|^{2\ell} \, \mu(\d s,\d u) \bigg) = \int_0^{\cdot} \int_{\R^d} |u|^{2\ell} \, \nu(\d s,\d u).
\end{equation*}
Thus, we find that
\begin{equation*}
\| \delta M_{s,t} \|_{q,r,s} \lesssim \bigg\| |\delta \langle M^c \rangle_{s,t}| + \int_s^t \int_{\R^d} |u|^2 \, \nu(\d v,\d u) \bigg\|^{\frac{1}{2}}_{\frac{q}{2},\frac{r}{2},s} + \bigg\| \int_s^t \int_{\R^d} |u|^q \, \nu(\d v,\d u) \bigg\|_{1,\frac{r}{q},s}^{\frac{1}{q}}.
\end{equation*}

Since we may reverse the inequality in the (generalized) BDG inequality, we may also reverse the inequality above, up to a change of the multiplicative constant. It then just remains to note that, by allowing the multiplicative constant to depend on the dimension $d$, these inequalities are also valid when we replace $|\delta \langle M^c \rangle_{s,t}|$ by $\tr(\delta C_{s,t})$.
\end{proof}

\begin{example}
For some $q \geq 2$, let $L$ be an $\R^d$-valued L\'evy process such that $\E[|L_1|^q] < \infty$, so that in particular its L\'evy measure $\nu$ satisfies $\int_{\R^d} |u|^q \, \nu(\d u) < \infty$. We infer from Proposition~\ref{proposition: a criterium for martingales in mixed moment spaces} that $L \in V^p L^{q,\infty}$ for every $p \geq q$, but not in general for $p < q$.
\end{example}

The following lemma is a generalization of \cite[Proposition~5.14]{FrizVictoir2010} to the mixed moment setting. It implies that c\`adl\`ag processes which almost surely do not jump at deterministic times may be viewed as continuous paths taking values in mixed moment spaces. Since the proof is essentially the same as in the deterministic case, it is omitted here for brevity.

\begin{lemma}
Let $p, q \in [1,\infty)$ and $r \in [q,\infty]$, and suppose that $Y$ is a c\`adl\`ag stochastic process such that $\Delta Y_t = 0$ almost surely for every (deterministic) time $t \in (0,T]$. Then $Y \in V^p L^{q,r}$ if and only if there exists a deterministic continuous non-decreasing function $\phi \colon [0,T] \to [0,1]$ such that $Y \circ \phi^{-1} \in C^{\frac{1}{p}} L^{q,r}$.
\end{lemma}

Of course, if the process $Y$ is adapted to $(\cF_t)_{t \in [0,T]}$, then the time-changed process $Y \circ \phi^{-1}$ is in general only adapted to the time-changed filtration $(\cF_{\phi^{-1}(t)})_{t \in [0,1]}$. Nevertheless, the previous lemma suggests that the processes in $V^p L^{q,r}$ which truly lie outside the scope of the H\"older space $C^{\frac{1}{p}} L^{q,r}$ are those which jump with positive probability at deterministic times. Such processes (in particular those with jumps at predictable times) are of particular interest in credit default risk modelling, in which one typically wishes to incorporate fixed or announced dates of shocks to the market; see, e.g., \cite{BandiniCalviaColaneri2022}, \cite{FontanaSchmidt2018} or \cite{JiaoLi2015}.

\begin{example}
Let $W = (W_t)_{t \in [0,T]}$ be a Brownian motion, let $\lambda \colon [0,T] \to [0,T]$ be a non-decreasing c\`adl\`ag function, and let $M_t = W_{\lambda(t)}$ for each $t \in [0,T]$. Then the process $M = (M_t)_{t \in [0,T]}$ is an $(\cF_{\lambda(t)})_{t \in [0,T]}$-adapted martingale with $[M] = \lambda$. In particular, we have by Lemma~\ref{lemma: quadratic variation is all that matters for p,q,r} that $M \in V^2 L^{q,\infty}$ for every $q \in [2,\infty)$. However, if the time-change $\lambda$ is discontinuous, then $M \notin C^{\frac{1}{p}} L^{q,\infty}$.
\end{example}

\begin{example}
Let $N$ be a counting process with compensator $A$. (For instance, $N$ could be an extended Poisson process, in the sense of \cite[I.~Definition~3.26]{JacodShiryaev2006}, with potentially discontinuous intensity $t \mapsto A_t = \E [N_t]$.) Considering the martingale $M = N - A$, if $A \in V^1 L^{2,\infty}$, then we have by Proposition~\ref{proposition: a criterium for martingales in mixed moment spaces} that in particular $M \in V^2 L^{2,\infty}$, so that also $N \in V^2 L^{2,\infty}$. However, since $\E [\Delta N_t \,|\, \cF_{t-} ] = \Delta A_t$ for every $t \in (0,T]$, we note that if $A$ has a jump at time $t$ (with positive probability), then $N$ will also have a jump at time $t$, which means that in general $N \notin C^{\frac{1}{2}} L^{2,\infty}$.
\end{example}

\section{A stochastic sewing lemma for discontinuous controls}\label{section: stochastic sewing lemma}

We are now ready to state our main technical result, namely, a new stochastic sewing lemma. Although more of a curiosity than the focus of this work, we then show how the It\^o stochastic integral against a c\`adl\`ag martingale in $V^p L^{q,\infty}$ can be constructed using stochastic sewing, providing a novel perspective on this central object of stochastic analysis.

\subsection{The stochastic sewing lemma}

Given a two-parameter process $\Xi = (\Xi_{s,t})_{(s,t) \in \Delta_{[0,T]}}$ taking values in some Banach space, we write $\delta \Xi_{s,u,t} := \Xi_{s,t} - \Xi_{s,u} - \Xi_{u,t}$ whenever $s \leq u \leq t$. Recall that we say that a control $w$ is \emph{continuous from the inside} if $w(s+,t) = w(s,t) = w(s,t-)$ for all $(s,t) \in \Delta_{[0,T]}$.\footnote{Recall that we define $w(t+,t) = 0$ and $w(t,t-) = 0$ for all $t \in [0,T]$.} We also recall from \cite[Definition~1.1]{FrizZhang2018} the following distinct notions of convergence.

\begin{definition}
Writing $\cP$ for a partition of the interval $[0,T]$, we say that the Riemann sum $\sum_{[u,v] \in \cP} \Xi_{u,v}$ converges to a limit $\cI$ in the sense of
\begin{itemize}
\item Mesh Riemann--Stieltjes (MRS) if for every $\epsilon > 0$ there exists a $\delta > 0$ such that, for any partition $\cP$ of $[0,T]$ with mesh size $|\cP| < \delta$, we have that $\|\sum_{[u,v] \in \cP} \Xi_{u,v} - \cI\| < \epsilon$,
\item Refinement Riemann--Stieltjes (RRS) if for any $\epsilon > 0$ there exists a partition $\cP^{\epsilon}$ of $[0,T]$ such that, for any refinement $\cP \supseteq \cP^{\epsilon}$, we have that $\|\sum_{[u,v] \in \cP} \Xi_{u,v} - \cI\| < \epsilon$.
\end{itemize}
\end{definition}

\begin{theorem}\label{theorem: mild stochastic sewing lemma}
Let $q \in [2,\infty)$, $r \in [q,\infty]$, and let $\Xi = (\Xi_{s,t})_{(s,t) \in \Delta_{[0,T]}}$ be a two-parameter process, taking values in a finite-dimensional Banach space, such that $\Xi_{s,t}$ is $\cF_t$-measurable for every $(s,t) \in \Delta_{[0,T]}$. Let $(w_{1,i}, w_{2,i})_{i = 1, \ldots, N}$ and $(\bw_{1,j}, \bw_{2,j})_{j = 1, \ldots, M}$ be controls, and suppose that
\begin{equation}\label{eq: stoch sewing E_s Lqr bound}
\big\| \E_s [\delta \Xi_{s,u,t}] \big\|_{L^r} \leq \sum_{i=1}^N w_{1,i}(s,u)^{\alpha_{1,i}} w_{2,i}(u,t)^{\alpha_{2,i}}
\end{equation}
and
\begin{equation}\label{eq: stoch sewing qrs bound}
\|\delta \Xi_{s,u,t}\|_{q,r,s} \leq \sum_{j=1}^M \bw_{1,j}(s,u)^{\beta_{1,j}} \bw_{2,j}(u,t)^{\beta_{2,j}}
\end{equation}
for all $0 \leq s \leq u \leq t \leq T$, where $\alpha_{1,i}, \alpha_{2,i}, \beta_{1,j}, \beta_{2,j} > 0$ are such that $\alpha_{1,i} + \alpha_{2,i} > 1$ for every $i = 1, \ldots, N$, and $\beta_{1,j} + \beta_{2,j} > \frac{1}{2}$ for every $j = 1, \ldots, M$.

Then there exists an adapted process $\cI = (\cI_t)_{t \in [0,T]}$, with $\cI_0 = 0$, which is unique up to modification, such that
\begin{equation}\label{eq: L^qr bound full generality sewing}
\big\| \E_s [\delta \cI_{s,t} - \Xi_{s,t}] \big\|_{L^r} \leq C \sum_{i=1}^N w_{1,i}(s,t-)^{\alpha_{1,i}} w_{2,i}(s+,t)^{\alpha_{2,i}}
\end{equation}
and
\begin{equation}\label{eq: qrs bound full generality sewing}
\|\delta \cI_{s,t} - \Xi_{s,t}\|_{q,r,s} \leq C \bigg(\sum_{i=1}^N w_{1,i}(s,t-)^{\alpha_{1,i}} w_{2,i}(s+,t)^{\alpha_{2,i}} + \sum_{j=1}^M \bw_{1,j}(s,t-)^{\beta_{1,j}} \bw_{2,j}(s+,t)^{\beta_{2,j}}\bigg).
\end{equation}
for all $(s,t) \in \Delta_{[0,T]}$, where the constant $C$ depends only on $q, N, M$, $\min_{i=1,\ldots,N} (\alpha_{1,i} + \alpha_{2,i})$ and $\min_{j=1,\ldots,M} (\beta_{1,j} + \beta_{2,j})$. Moreover, for each $(s,t) \in \Delta_{[0,T]}$, we have that
\begin{equation}\label{eq: convergence in stochastic sewing lemma}
\bigg\|\delta \cI_{s,t} - \sum_{[u,v] \in \cP} \Xi_{u,v}\bigg\|_{q,r,s} \, \longrightarrow \, 0
\end{equation}
in the sense of RRS convergence, over partitions $\cP$ of the interval $[s,t]$.

Furthermore, if for every $i$ and $j$, either $w_{1,i}$ is left-continuous in its second argument or $w_{2,i}$ is right-continuous in its first argument, and similarly for $\bw_{1,j}$ and $\bw_{2,j}$, then the convergence in \eqref{eq: convergence in stochastic sewing lemma} also holds in the sense of MRS convergence.
\end{theorem}

The proof of Theorem~\ref{theorem: mild stochastic sewing lemma} will be presented in Section~\ref{section: proof of sewing lemma}.

\begin{remark}\label{remark: stochastic sewing can be generalised}
In Theorem~\ref{theorem: mild stochastic sewing lemma} we could have been slightly more general. Namely, in place of the bounds in \eqref{eq: stoch sewing E_s Lqr bound} and \eqref{eq: stoch sewing qrs bound}, the theorem would still be true under bounds of the form
\begin{equation*}
\begin{split}
\big\| \E_s [\delta \Xi_{s,u,t}] \big\|_{L^r} &\leq \sum_{i=1}^N \sigma_i(s,u)^{\delta_{1,i}} w_{1,i}(s,t)^{\alpha_{1,i} - \delta_{1,i}} w_{2,i}(u,t)^{\alpha_{2,i}},\\
\|\delta \Xi_{s,u,t}\|_{q,r,s} &\leq \sum_{j=1}^M \bar{\sigma}_j(s,u)^{\delta_{2,j}} \bw_{1,j}(s,t)^{\beta_{1,j} - \delta_{2,j}} \bw_{2,j}(u,t)^{\beta_{2,j}},
\end{split}
\end{equation*}
where $(\sigma_i)_{i = 1, \ldots, N}$ and $(\bar{\sigma}_j)_{j = 1, \ldots, M}$ are mild controls\footnote{This is a weaker notion of control which does not need to be superadditive; for details, see \cite[Section~1.1]{FrizZhang2018}.}, $(w_{1,i}, w_{2,i})_{i=1,\ldots,N}$ and $(\bw_{1,j}, \bw_{2,j})_{j=1,\ldots,M}$ are controls, and the exponents $\alpha_{1,i}, \alpha_{2,i}, \beta_{1,j}, \beta_{2,j} > 0$ and $0 < \delta_{1,i} \leq \alpha_{1,i}$, $0 < \delta_{2,j} \leq \beta_{1,j}$ are such that $\alpha_{1,i} + \alpha_{2,i} - \delta_{1,i} > 1$ for every $i = 1, \ldots, N$, and $\beta_{1,j} + \beta_{2,j} - \delta_{2,j} > \frac{1}{2}$ for every $j = 1, \ldots, M$. In particular, under these assumptions, one may choose $\delta_{1,i} = \alpha_{1,i}$ whenever $\alpha_{2,i} > 1$ and $\delta_{2,j} = \beta_{1,j}$ whenever $\beta_{2,j} > \frac{1}{2}$, thus extending the mild sewing lemma of \cite[Theorem~1.7]{FrizZhang2018} to the stochastic setting.
\end{remark}

\subsection{A uniform estimate}

So far we have only considered processes which are right-continuous in probability. The following theorem provides sample path regularity, as well as a uniform estimate, for the process $\cI$ in Theorem~\ref{theorem: mild stochastic sewing lemma}. We defer the proof of Theorem~\ref{theorem: sample path regularity sewing} to Appendix~\ref{section: proof of uniform estimate}.

\begin{theorem}\label{theorem: sample path regularity sewing}
Let $q \in [2,\infty)$, $r \in [q,\infty]$, and let $\Xi = (\Xi_{s,t})_{(s,t) \in \Delta_{[0,T]}}$ be a two-parameter process, taking values in a finite-dimensional Banach space, such that $\Xi_{s,t}$ is $\cF_t$-measurable for every $(s,t) \in \Delta_{[0,T]}$. Suppose that the filtration $(\cF_t)_{t \in [0,T]}$ is complete, and that, for every $s \in [0,T)$, the process $t \mapsto \Xi_{s,t}$ has almost surely c\`adl\`ag sample paths.

Let $(w_{1,i},w_{2,i})_{i = 1, \ldots, N}$, $(\bw_{1,j},\bw_{2,j})_{j = 1, \ldots, M}$ and $(\hw_{1,k},\hw_{2,k})_{k = 1, \ldots, K}$ be controls, and suppose that $\Xi$ satisfies the bounds in \eqref{eq: stoch sewing E_s Lqr bound} and \eqref{eq: stoch sewing qrs bound}, so that the hypotheses of Theorem~\ref{theorem: mild stochastic sewing lemma} are satisfied. Assume further that
\begin{equation}\label{eq: uniform q,r,s bound for uniform convergence}
\Big\| \sup_{v \in [u,t]} |\delta \Xi_{s,u,v}| \Big\|_{q,r,s} \leq \sum_{k=1}^K \hw_{1,k}(s,u)^{\gamma_{1,k}} \hw_{2,k}(u,t)^{\gamma_{2,k}}
\end{equation}
holds for all $0 \leq s \leq u \leq t \leq T$, where $\gamma_{1,k}, \gamma_{2,k} > 0$ satisfy $\gamma_{1,k} + \gamma_{2,k} > \frac{1}{q}$ for every $k$.

Then the process $\cI$ given in Theorem~\ref{theorem: mild stochastic sewing lemma} has a version $\widetilde{\cI}$ with c\`adl\`ag sample paths. Moreover, there exists an $\eta > 0$ and a nested sequence of partitions $(\cP^h)_{h \in \N_0}$, with $\cP^0 = \{0,T\}$ and with vanishing mesh size as $h \to \infty$, such that, for every $h \in \N_0$,
\begin{equation}\label{eq: uniform estimate in full generality sewing}
\begin{split}
\Big\|\sup_{t \in [0,T]} \big|&\widetilde{\cI}_t - \Xi^{\cP^h}_t \big|\Big\|_{q,r,0} \leq C 2^{-\eta h} \bigg(\sum_{i=1}^N w_{1,i}(0,T-)^{\alpha_{1,i}} w_{2,i}(0+,T)^{\alpha_{2,i}}\\
&+ \sum_{j=1}^M \bw_{1,j}(0,T-)^{\beta_{1,j}} \bw_{2,j}(0+,T)^{\beta_{2,j}} + \sum_{k=1}^K \hw_{1,k}(0,T-)^{\gamma_{1,k}} \hw_{2,k}(0+,T)^{\gamma_{2,k}}\bigg),
\end{split}
\end{equation}
where the constants $\eta$ and $C$ depend only on $q, N, M, K$ and the exponents $(\alpha_{1,i}, \alpha_{2,i})_{i=1,\ldots,N}$, $(\beta_{1,j}, \beta_{2,j})_{j=1,\ldots,M}$ and $(\gamma_{1,k}, \gamma_{2,k})_{k=1,\ldots,K}$, and where, for a given partition $\cP$, we write
\begin{equation}\label{eq: defn Xi^cP}
\Xi^{\cP}_t := \sum_{[u,v] \in \cP} \Xi_{u \wedge t,v \wedge t}.
\end{equation}
\end{theorem}

For the remainder of the paper, we shall assume that all the processes we consider take values in finite-dimensional Banach spaces.

\subsection{It\^o integration}

\begin{lemma}\label{lemma: existence of integrals against martingales}
Let $p \in [1,4)$, $q \in [2,\infty)$ and $r \in [q,\infty]$. Suppose that $Y \in V^p L^{q,r}$ is adapted, and that $M \in V^p L^{q,\infty}$ is a martingale. Then there exists a unique (up to modification) $L^q$-integrable martingale $(\int_0^t Y_u \dd M_u)_{t \in [0,T]}$, started at zero, which satisfies
\begin{equation}\label{eq: int against martingale estimate}
\bigg\|\int_s^t Y_u \dd M_u - Y_s \delta M_{s,t}\bigg\|_{q,r,s} \leq C \|Y\|_{p,q,r,[s,t)} \|M\|_{p,q,\infty,[s,t]}
\end{equation}
for all $(s,t) \in \Delta_{[0,T]}$, where the constant $C$ depends only on $p$ and $q$. Moreover, for every $(s,t) \in \Delta_{[0,T]}$, we have that
\begin{equation}\label{eq: Ito integral convergence}
\lim_{|\cP| \to 0} \bigg\|\int_s^t Y_u \dd M_u - \sum_{[u,v] \in \cP} Y_u \delta M_{u,v}\bigg\|_{q,r,s} = 0,
\end{equation}
where the limit holds along partitions $\cP$ of the interval $[s,t]$ as the mesh size tends to zero.
\end{lemma}

\begin{proof}
Letting $\Xi_{s,t} = Y_s \delta M_{s,t}$, it is easy to see that $\delta \Xi_{s,u,t} = -\delta Y_{s,u} \delta M_{u,t}$. We thus have that $\E_s [\delta \Xi_{s,u,t}] = -\E_s [\delta Y_{s,u} \E_u [\delta M_{u,t}]] = 0$, and that
\begin{align*}
\|\delta \Xi_{s,u,t}\|_{q,r,s} &= \big\| \E_s \big[|\delta Y_{s,u}|^q \E_u \big[ |\delta M_{u,t}|^q \big] \big]^{\frac{1}{q}} \big\|_{L^r} \leq \|\delta Y_{s,u}\|_{q,r,s} \|\delta M_{u,t}\|_{q,\infty,u}\\
&\leq \|Y\|_{p,q,r,[s,u]} \|M\|_{p,q,\infty,[u,t]} = w_1(s,u)^{\frac{1}{p}} w_2(u,t)^{\frac{1}{p}},
\end{align*}
where $w_1(s,u) := \|Y\|_{p,q,r,[s,u]}^p$ and $w_2(u,t) := \|M\|_{p,q,\infty,[u,t]}^p$ are controls, which are both right-continuous by Lemma~\ref{lemma: right-continuity of controls}.

Since $p < 4$, we have that $\frac{1}{p} + \frac{1}{p} > \frac{1}{2}$. It then follows from Theorem~\ref{theorem: mild stochastic sewing lemma} that there exists a unique adapted process $\cI =: \int_0^\cdot Y_u \dd M_u$ which satisfies
\begin{equation*}
\E_s \bigg[\int_s^t Y_u \dd M_u - Y_s \delta M_{s,t}\bigg] = \E_s [\delta \cI_{s,t} - \Xi_{s,t}] = 0,
\end{equation*}
which implies that $\int_0^\cdot Y_u \dd M_u$ is a martingale, as well as the bound in \eqref{eq: int against martingale estimate}, and which additionally satisfies the convergence in \eqref{eq: Ito integral convergence}.

Since $\|Y_s \delta M_{s,t}\|_{L^q} = \E [|Y_s|^q \E_s [|\delta M_{s,t}|^q]]^{\frac{1}{q}} \leq \|Y_s\|_{L^q} \|\delta M_{s,t}\|_{q,\infty,s} < \infty$, it follows from \eqref{eq: int against martingale estimate} that $\int_s^t Y_u \dd M_u \in L^q$ for all $(s,t) \in \Delta_{[0,T]}$, so that the process $\int_0^\cdot Y_u \dd M_u$ is $L^q$-integrable.
\end{proof}

\begin{lemma}
Recall the setting of Lemma~\ref{lemma: existence of integrals against martingales}, and assume in addition that the filtration $(\cF_t)_{t \in [0,T]}$ satisfies the usual conditions, and that the sample paths of $Y$ are almost surely c\`adl\`ag. Then the integral $\int_0^\cdot Y_u \dd M_u$ defined in Lemma~\ref{lemma: existence of integrals against martingales} is a version of the classical It\^o integral of $Y_-$ with respect to $M$.
\end{lemma}

\begin{proof}
This is an immediate consequence of the convergence in \eqref{eq: Ito integral convergence} and, e.g., \cite[Chapter II, Theorem 21]{Protter2005}, by the almost sure uniqueness of limits in probability.
\end{proof}

We saw above that the It\^o integral against a suitably integrable martingale can be constructed using stochastic sewing. In the next lemma, we see that we can also use Theorem~\ref{theorem: sample path regularity sewing} to obtain the uniform approximation of It\^o integrals by Riemann sums.

\begin{lemma}
Let $p \in [1,4)$, $q \in [2,\infty)$ and $r \in [q,\infty]$, and assume that the filtration $(\cF_t)_{t \in [0,T]}$ is complete. Suppose that $Y \in V^p L^{q,r}$ is adapted, and that $M \in V^p L^{q,\infty}$ is a c\`adl\`ag martingale. Then there exists a nested sequence of partitions $(\cP^n)_{n \in \N}$ such that
\begin{align*}
\lim_{n \to \infty} \bigg\| \sup_{t \in [0,T]} \bigg| \int_0^t Y_s \dd M_s - \sum_{[u,v] \in \cP^n} Y_u \delta M_{u \wedge t,v \wedge t} \bigg| \bigg\|_{q,r,0} = 0.
\end{align*}
\end{lemma}

\begin{proof}
Letting $\Xi_{s,t} = Y_s \delta M_{s,t}$, we have that $\delta \Xi_{s,u,t} = -\delta Y_{s,u} \delta M_{u,t}$. Using the conditional Doob $L^p$ inequality, we then have that
\begin{equation*}
\begin{split}
&\Big\| \sup_{v \in [u,t]} |\delta \Xi_{s,u,v}| \Big\|_{q,r,s} = \Big\| \E_s \Big[ |\delta Y_{s,u}|^q \sup_{v \in [u,t]} |\delta M_{u,v}|^q \Big]^{\frac{1}{q}} \Big\|_{L^r}\\
&= \Big\| \E_s \Big[ |\delta Y_{s,u}|^q \E_u \Big[ \sup_{v \in [u,t]} |\delta M_{u,v}|^q \Big] \Big]^{\frac{1}{q}} \Big\|_{L^r} \lesssim \Big\| \E_s \Big[ |\delta Y_{s,u}|^q \E_u \big[ |\delta M_{u,t}|^q \big] \Big]^{\frac{1}{q}} \Big\|_{L^r}\\
&\leq \|\delta Y_{s,u}\|_{q,r,s} \|\delta M_{u,t}\|_{q,\infty,u} \leq \|Y\|_{p,q,r,[s,u]} \|M\|_{p,q,\infty,[u,t]},
\end{split}
\end{equation*}
and the result then follows by an application of Theorem~\ref{theorem: sample path regularity sewing}.
\end{proof}

\section{Rough stochastic integration}\label{section: rough stochastic integration}

In a similar spirit to \cite{FrizHocquetLe2021}, in this section we proceed to define the notion of stochastic controlled paths in our c\`adl\`ag setting. We establish the corresponding notion of rough stochastic integration, along with stability of stochastic controlled paths under rough stochastic integration and under the application of sufficiently regular functions.

\subsection{Stochastic controlled paths}

We consider pairs $\bX = (X,\X)$, consisting of a c\`adl\`ag path $X \colon [0,T] \to \R^d$ and a c\`adl\`ag two-parameter function $\X \colon \Delta_{[0,T]} \to \R^{d \times d}$, such that $\|X\|_{p,[0,T]} < \infty$ and $\|\X\|_{\frac{p}{2},[0,T]} < \infty$. For such pairs, and any $(s,t) \in \Delta_{[0,T]}$, we use the natural seminorm
\begin{equation*}
\|\bX\|_{p,[s,t]} := \|X\|_{p,[s,t]} + \|\X\|_{\frac{p}{2},[s,t]},
\end{equation*}
which induces the pseudometric
\begin{equation*}
(\bX,\tbX) \, \mapsto \, \|\bX - \tbX\|_{p,[s,t]} = \|X - \tX\|_{p,[s,t]} + \|\X - \tbbX\|_{\frac{p}{2},[s,t]}
\end{equation*}
for pairs $\bX = (X,\X)$ and $\tbX = (\tX,\tbbX)$.

For a given $p \in [2,3)$, a \emph{c\`adl\`ag rough path} is such a pair $\bX = (X,\X)$, such that $\|\bX\|_{p,[0,T]} < \infty$, and such that Chen's relation
\begin{equation*}
\X_{s,t} = \X_{s,u} + \X_{u,t} + \delta X_{s,u} \otimes \delta X_{u,t}
\end{equation*}
holds for all $0 \leq s \leq u \leq t \leq T$.

We write $\sV^p = \sV^p([0,T];\R^d)$ for the (complete metric) space of c\`adl\`ag rough paths.

\smallskip

Given a two-parameter process $A = (A_{s,t})_{(s,t) \in \Delta_{[0,T]}}$, we write $\E_{\edot} A$ for the two-parameter process given by
\begin{equation*}
(\E_{\edot} A)_{s,t} := \E_s [A_{s,t}]
\end{equation*}
for every $(s,t) \in \Delta_{[0,T]}$.

\smallskip

In the following, we consider the norm $\|\cdot\|_{\frac{p}{2},r,[0,T]}$. Strictly speaking, in Definition~\ref{defn: V^pL^q and V^pL^q,r spaces} we only defined the norm $\|\cdot\|_{p,q,[s,t]}$ for $q < \infty$, but we can extend this definition to include $q = \infty$ in the natural way, by setting $\|F\|_{p,\infty,[s,t]} := (\sup_{\cP \subset [s,t]} \sum_{[u,v] \in \cP} \|F_{u,v}\|_{L^\infty}^p)^{\frac{1}{p}}$.

\begin{definition}
Let $p \in [2,3)$, $q \in [2,\infty)$ and $r \in [q,\infty]$, and let $X \in V^p$. We call a pair of processes $(Y,Y')$ a \emph{stochastic controlled path} (relative to $X$), if
\begin{enumerate}[(i)]
\item $Y$ and $Y'$ are both adapted,
\item $Y \in V^p L^{q,r}$, 
\item $Y' \in V^p L^{q,r}$ and $\sup_{s \in [0,T]} \|Y'_s\|_{L^{r}} < \infty$, and
\item $\|\E_{\edot} R^Y\|_{\frac{p}{2},r,[0,T]} < \infty$, where the two-parameter process $R^Y = (R^Y_{s,t})_{(s,t) \in \Delta_{[0,T]}}$ is defined by
\begin{equation}\label{eq: defn R^Y_s,t}
\delta Y_{s,t} = Y'_{s} \delta X_{s,t} + R^Y_{s,t}
\end{equation}
for every $(s,t) \in \Delta_{[0,T]}$.
\end{enumerate}
We write $\cV^{p,q,r}_X$ for the space of stochastic controlled paths relative to $X$. For brevity, in the case when $r = q$, we will simply write $\cV^{p,q}_X = \cV^{p,q,q}_X$.
\end{definition}

\begin{remark}\label{remark: right continuous control for remainder}
Let $p \in [2,3)$, $q \in [2,\infty)$ and $r \in [q,\infty]$, and let $X \in V^p$ and $(Y,Y') \in \cV^{p,q,r}_X$. It follows from \eqref{eq: defn R^Y_s,t}, part~(iii) of Proposition~\ref{prop: Lqrs is Banach space}, and Lemma~\ref{lemma: V^pL^q implies L^q-regulated} that, for any $s \in [0,T)$,
\begin{equation*}
\lim_{t \searrow s} \big\|\E_s [R^Y_{s,t}]\big\|_{L^{r}} \leq \lim_{t \searrow s} \|\delta Y_{s,t}\|_{q,r,s} + \sup_{u \in [0,T]} \|Y'_u\|_{L^{r}} \lim_{t \searrow s} |\delta X_{s,t}| = 0,
\end{equation*}
and similarly that $\lim_{t \searrow s} \|R^Y_{s,t}\|_{q,r,s} = 0$. It is then straightforward to adapt the proof of Lemma~\ref{lemma: right-continuity of controls} to show that $w_1(s,t) := \|\E_{\edot} R^Y\|_{\frac{p}{2},r,[s,t]}^{\frac{p}{2}}$ and $w_2(s,t) := \|R^Y\|_{p,q,r,[s,t]}^p$ define right-continuous controls $w_1, w_2$.
\end{remark}

\subsection{Rough stochastic integration}

\begin{lemma}\label{lemma: existence of rough stochastic integrals}
Let $p \in [2,3)$, $q \in [2,\infty)$ and $r \in [q,\infty]$. Let $\bX = (X,\X) \in \sV^p$ be a c\`adl\`ag rough path, and let $(Y,Y') \in \cV^{p,q,r}_X$ be a stochastic controlled path. Then there exists an $L^q$-integrable adapted process $(\int_0^t Y_u \dd \bX_u)_{t \in [0,T]}$, such that, for every $(s,t) \in \Delta_{[0,T]}$,
\begin{equation*}
\lim_{|\cP| \to 0} \bigg\|\int_s^t Y_u \dd \bX_u - \sum_{[u,v] \in \cP} \big(Y_u \delta X_{u,v} + Y'_u \X_{u,v}\big)\bigg\|_{q,r,s} = 0,
\end{equation*}
where the limit holds along partitions $\cP$ of the interval $[s,t]$ as the mesh size tends to zero. Moreover, we have that
\begin{equation}\label{eq: Lqr- estimate for rough stochastic integral}
\bigg\|\E_s \bigg[\int_s^t Y_u \dd \bX_u - Y_s \delta X_{s,t} - Y'_s \X_{s,t}\bigg]\bigg\|_{L^{r}} \leq C\Big(\|\E_{\edot} R^Y\|_{\frac{p}{2},r,[s,t)} \|X\|_{p,[s,t]} + \|Y'\|_{p,q,r,[s,t)} \|\X\|_{\frac{p}{2},[s,t]}\Big)
\end{equation}
and
\begin{equation}\label{eq: p,q,r,s- estimate for rough stochastic integral}
\begin{split}
&\bigg\|\int_s^t Y_u \dd \bX_u - Y_s \delta X_{s,t} - Y'_s \X_{s,t}\bigg\|_{q,r,s}\\
&\leq C \Big(\|Y\|_{p,q,r,[s,t)} + \|Y'\|_{p,q,r,[s,t)} + \|\E_{\edot} R^Y\|_{\frac{p}{2},r,[s,t)} + \|X\|_{p,[s,t)} \sup_{u \in [0,T)} \|Y'_u\|_{L^{r}}\Big) \|\bX\|_{p,[s,t]}
\end{split}
\end{equation}
for every $(s,t) \in \Delta_{[0,T]}$, where the constant $C$ depends only on $p$ and $q$.
\end{lemma}

We call the process $\int_0^\cdot Y_u \dd \bX_u$ constructed in Lemma~\ref{lemma: existence of rough stochastic integrals} the \emph{rough stochastic integral} of $(Y,Y')$ with respect to $\bX$.

\begin{proof}
Letting $\Xi_{s,t} = Y_s \delta X_{s,t} + Y'_s \X_{s,t}$, a short calculation using Chen's relation shows that
\begin{align*}
\delta \Xi_{s,u,t} = -R^Y_{s,u} \delta X_{u,t} - \delta Y'_{s,u} \X_{u,t}
\end{align*}
for any $s \leq u \leq t$. Recalling \eqref{eq: L^qr norm of E_s bounded by q,r,s}, then we have that
\begin{align*}
\big\|\E_s [\delta \Xi_{s,u,t}]\big\|_{L^{r}} &\leq \big\|\E_s[R^Y_{s,u}]\big\|_{L^{r}} |\delta X_{u,t}| + \|\delta Y'_{s,u}\|_{q,r,s} |\X_{u,t}|\\
&\leq \|\E_{\edot} R^Y\|_{\frac{p}{2},r,[s,u]} \|X\|_{p,[u,t]} + \|Y'\|_{p,q,r,[s,u]} \|\X\|_{\frac{p}{2},[u,t]}\\
&=: w_{1,1}(s,u)^{\frac{2}{p}} w_{2,1}(u,t)^{\frac{1}{p}} + w_{1,2}(s,u)^{\frac{1}{p}} w_{2,2}(u,t)^{\frac{2}{p}},
\end{align*}
where $w_{1,1}, w_{2,1}, w_{1,2}$ and $w_{2,2}$ are controls, which are all right-continuous by \cite[Lemma~7.1]{FrizZhang2018}, Lemma~\ref{lemma: right-continuity of controls} and Remark~\ref{remark: right continuous control for remainder}. We see from \eqref{eq: defn R^Y_s,t} that
\begin{equation*}
\|R^Y\|_{p,q,r,[s,u]} \leq \|Y\|_{p,q,r,[s,u]} + \sup_{v \in [s,u)} \|Y'_v\|_{L^{r}} \|X\|_{p,[s,u]},
\end{equation*}
and we then also have that
\begin{align*}
\|&\delta \Xi_{s,u,t}\|_{q,r,s} \leq \|R^Y_{s,u}\|_{q,r,s} |\delta X_{u,t}| + \|\delta Y'_{s,u}\|_{q,r,s} |\X_{u,t}|\\
&\leq \|R^Y\|_{p,q,r,[s,u]} \|X\|_{p,[u,t]} + \|Y'\|_{p,q,r,[s,u]} \|\X\|_{\frac{p}{2},[u,t]}\\
&\leq \|Y\|_{p,q,r,[s,u]} \|X\|_{p,[u,t]} + \sup_{v \in [0,T)} \|Y'_v\|_{L^{r}} \|X\|_{p,[s,u]} \|X\|_{p,[u,t]} + \|Y'\|_{p,q,r,[s,u]} \|\X\|_{\frac{p}{2},[u,t]}\\
&=: \bw_{1,1}(s,u)^{\frac{1}{p}} \bw_{2,1}(u,t)^{\frac{1}{p}} + \bw_{1,2}(s,u)^{\frac{1}{p}} \bw_{2,2}(u,t)^{\frac{1}{p}} + \bw_{1,3}(s,u)^{\frac{1}{p}} \bw_{2,3}(u,t)^{\frac{2}{p}},
\end{align*}
where $\bw_{1,1}, \bw_{2,1}, \bw_{1,2}, \bw_{2,2}, \bw_{1,3}$ and $\bw_{2,3}$ are right-continuous controls. The result then follows from Theorem~\ref{theorem: mild stochastic sewing lemma}.
\end{proof}

\begin{remark}
For our purposes, it is sufficient in Lemma~\ref{lemma: existence of rough stochastic integrals} and in what follows to consider c\`adl\`ag rough paths and stochastic controlled paths which are c\`adl\`ag in $L^q$. However, we note that the stochastic sewing lemma in Theorem~\ref{theorem: mild stochastic sewing lemma} recovers the full scope of generality of the sewing results obtained in \cite{FrizZhang2018}. In particular, no continuity assumptions are required on the controls (except to obtain MRS convergence).

Consequently, Theorem~\ref{theorem: mild stochastic sewing lemma} can in principle be used to define rough stochastic integration against rough paths $\bX$ that are not necessarily c\`adl\`ag, as well as for integrands that are merely $L^q$-regulated. Indeed, the discussion in \cite{FrizZhang2018} concerning such more general classes of integrators and integrands, as well as the invariance of (rough) integrals under left- and right-point modifications of the integrands, carries over to the present stochastic framework. While we do not pursue such generality here, we note that the pure jump sewing lemma in \cite[Theorem~2.11]{FrizZhang2018} admits a straightforward stochastic extension; see Lemma~\ref{lemma: pure jump stochastic sewing lemma}.
\end{remark}

\begin{remark}
An alternative (though, in light of the stochastic sewing lemma established in Theorem~\ref{theorem: mild stochastic sewing lemma}, unnecessarily involved) approach to rough stochastic integration was presented in \cite[Theorem~4.9]{LiangTang2025}. There, the authors employed a Doob--Meyer-type decomposition for stochastic controlled paths (first introduced in \cite[Section~3.1]{FrizHocquetLe2021}), and defined the rough stochastic integral as the sum of an $L^q$-valued rough integral (constructed via a standard sewing lemma) and the integral of a martingale against a rough path. In this setting, only a single (continuous) control appears in the stochastic sewing lemma.

Although such a Doob--Meyer decomposition is not required for the construction of rough stochastic integrals in the present framework, it remains an interesting open question whether an analogous decomposition holds in the general c\`adl\`ag setting.
\end{remark}

\begin{lemma}\label{lemma: rough stochastic integral cadlag sample paths}
Let $p \in [2,3)$, $q \in [2,\infty)$ and $r \in [q,\infty]$. Let $\bX = (X,\X) \in \sV^p$ be a c\`adl\`ag rough path, let $(Y,Y') \in \cV^{p,q,r}_X$ be a stochastic controlled path, and suppose that the filtration $(\cF_t)_{t \in [0,T]}$ is complete. Then the rough stochastic integral $\int_0^\cdot Y_u \dd \bX_u$ has a version with c\`adl\`ag sample paths. Moreover, we have that
\begin{equation}\label{eq: uniform bound for rough stochastic integral}
\begin{split}
&\bigg\|\sup_{v \in [s,t]} \bigg|\int_s^v Y_u \dd \bX_u - \big(Y_s \delta X_{s,v} + Y'_s \X_{s,v}\big)\bigg|\bigg\|_{q,r,s}\\
&\leq C \Big(\|Y\|_{p,q,r,[s,t)} + \|Y'\|_{p,q,r,[s,t)} + \|\E_{\edot} R^Y\|_{\frac{p}{2},r,[s,t)} + \|X\|_{p,[s,t)} \sup_{u \in [0,T)} \|Y'_u\|_{L^{r}}\Big) \|\bX\|_{p,[s,t]}
\end{split}
\end{equation}
for every $(s,t) \in \Delta_{[0,T]}$, where the constant $C$ depends only on $p$ and $q$.
\end{lemma}

\begin{proof}
Letting $\Xi_{s,t} = Y_s \delta X_{s,t} + Y'_s \X_{s,t}$, we have, for any $s \leq u \leq t$ and any $v \in [u,t]$, that $\delta \Xi_{s,u,v} = -R^Y_{s,u} \delta X_{u,v} - \delta Y'_{s,u} \X_{u,v}$, and hence that
\begin{align*}
\Big\|\sup_{v \in [u,t]} |\delta \Xi_{s,u,v}|\Big\|_{q,r,s} &\leq \|R^Y_{s,u}\|_{q,r,s} \sup_{v \in [u,t]} |\delta X_{u,v}| + \|Y'_{s,u}\|_{q,r,s} \sup_{v \in [u,t]} |\X_{u,v}|\\
&\leq \|R^Y\|_{p,q,r,[s,u]} \|X\|_{p,[u,t]} + \|Y'\|_{p,q,r,[s,u]} \|\X\|_{\frac{p}{2},[u,t]}\\
&=: \hw_{1,1}(s,u)^{\frac{1}{p}} \hw_{2,1}(u,t)^{\frac{1}{p}} + \hw_{1,2}(s,u)^{\frac{1}{p}} \hw_{2,2}(u,t)^{\frac{2}{p}},
\end{align*}
where $\hw_{1,1}, \hw_{2,1}, \hw_{1,2}$ and $\hw_{2,2}$ are right-continuous controls. We then have from Theorem~\ref{theorem: sample path regularity sewing} that $\int_0^\cdot Y_u \dd \bX_u$ has a version with c\`adl\`ag sample paths. Restricting to an arbitrary interval $[s,t]$, and taking the trivial partition $\cP^0 = \{s,t\}$, the estimate in \eqref{eq: uniform estimate in full generality sewing} implies the one in \eqref{eq: uniform bound for rough stochastic integral}.
\end{proof}

Given a process $Y$, we will write $Y_{t-} := \lim_{s \nearrow t} Y_s$ and $\Delta Y_t := \lim_{s \nearrow t} \delta Y_{s,t} = Y_t - Y_{t-}$, understood as limits in $L^q$. Of course, when the process considered has c\`adl\`ag sample paths, these objects then also exist as almost sure limits.

In the following, we will also denote $\Delta \X_t := \lim_{s \nearrow t} \X_{s,t}$.

\begin{lemma}\label{lemma: jump structure of rough stochastic integrals}
Recall the setting of Lemma~\ref{lemma: rough stochastic integral cadlag sample paths}. The rough stochastic integral $Z := \int_0^\cdot Y_u \dd \bX_u$ admits the canonical jump structure. That is, almost surely,
\begin{equation}\label{eq: jumps of rough stochastic integrals}
\Delta Z_t = Y_{t-} \Delta X_t + Y'_{t-} \Delta \X_t
\end{equation}
holds for every $t \in (0,T]$.
\end{lemma}

\begin{proof}
Since $Y_s, Y'_s \in L^q$, and the map $t \mapsto \|\bX\|_{p,[s,t]}$ is right-continuous, it follows from \eqref{eq: p,q,r,s- estimate for rough stochastic integral} that $\lim_{t \searrow s} \|\delta Z_{s,t}\|_{L^q} = 0$ for every $s \in [0,T)$, so that $Z$ is right-continuous in $L^q$.

It is straightforward to see that $\|X\|_{p,[s,t)}, \|Y\|_{p,q,r,[s,t)}, \|Y'\|_{p,q,r,[s,t)}, \|\E_{\edot} R^Y\|_{\frac{p}{2},r,[s,t)} \to 0$ as $s \nearrow t$, and it then also follows from \eqref{eq: p,q,r,s- estimate for rough stochastic integral} that
\begin{equation*}
\lim_{s \nearrow t} \big\| \delta Z_{s,t} - Y_s \delta X_{s,t} - Y'_s \X_{s,t} \big\|_{L^q} = 0.
\end{equation*}
We have from Lemma~\ref{lemma: V^pL^q is a Banach space} that the limits $Y_{t-}$ and $Y'_{t-}$ exist in $L^q$, and we thus deduce that the relation in \eqref{eq: jumps of rough stochastic integrals} holds almost surely for every fixed $t \in (0,T]$, and in particular that the left-limit $Z_{t-} = Z_t - \Delta Z_t$ exists in $L^q$.

Recalling the proof of Lemma~\ref{lemma: rough stochastic integral cadlag sample paths} above, we further have from Theorem~\ref{theorem: sample path regularity sewing} that
\begin{equation}\label{eq: uniform convergence for rough stochastic integral}
\Big\| \sup_{t \in [0,T]} \big| Z_t - \Xi^{\cP^h}_t \big| \Big\|_{q,r,0} \to 0
\end{equation}
as $h \to \infty$, for some sequence of partitions $(\cP^h)_{h \in \N}$ with vanishing mesh size, where
\begin{equation*}
\Xi^{\cP^h}_t := \sum_{[u,v] \in \cP^h} Y_{u \wedge t} \delta X_{u \wedge t,v \wedge t} + Y'_{u \wedge t} \X_{u \wedge t,v \wedge t},
\end{equation*}
which is clearly c\`adl\`ag, and only jumps at the jump times of $X$ and $\X$. The convergence in \eqref{eq: uniform convergence for rough stochastic integral} also implies that
\begin{equation*}
\sup_{t \in [0,T]} \big| \Delta Z_t - \Delta \Xi^{\cP^h}_t \big| \to 0
\end{equation*}
in $L^{q,r}_0$ as $h \to \infty$, and hence also almost surely along a subsequence. Thus, almost surely,
\[ \{t \in (0,T] \, : \, |\Delta Z_t| > 0\} \subset \{t \in (0,T] \, : \, |\Delta X_t| \vee |\Delta \X_t| > 0\} =: J_{\bX}, \]
where, crucially, the latter set of jump times $J_{\bX}$ of $\bX = (X,\X)$ is deterministic and countable. Since \eqref{eq: jumps of rough stochastic integrals} holds almost surely for each $t \in (0,T]$, it also holds simultaneously for every $t \in J_{\bX}$ almost surely, and the result follows.
\end{proof}

\begin{lemma}\label{lemma:  rough stochastic integrals as controlled path}
Let $p \in [2,3)$, $q \in [2,\infty)$, $r \in [q,\infty]$, $\bX = (X,\X) \in \sV^p$ be a c\`adl\`ag rough path, and let $(Y,Y') \in \cV^{p,q,r}_X$ be a stochastic controlled path such that $\sup_{u \in [0,T]} \|Y_u\|_{L^{r}} < \infty$. Then the pair
\begin{equation*}
(Z,Z') := \bigg(\int_0^\cdot Y_u \dd \bX_u, Y\bigg) \in \cV^{p,q,r}_X
\end{equation*}
is itself a stochastic controlled path. Moreover, we have the bounds
\begin{equation}\label{eq: bound for Z pqr}
\begin{split}
\|Z\|_{p,q,r,[0,T]} \leq C \Big(\|&Y\|_{p,q,r,[0,T)} + \|Y'\|_{p,q,r,[0,T)} + \|\E_{\edot} R^Y\|_{\frac{p}{2},r,[0,T)}\\
&+ \sup_{u \in [0,T)} \|Y_u\|_{L^{r}} + (1 + \|X\|_{p,[0,T)}) \sup_{u \in [0,T)} \|Y'_u\|_{L^{r}}\Big) \|\bX\|_{p,[0,T]}
\end{split}
\end{equation}
and
\begin{equation}\label{eq: bound for E R^Z}
\|\E_{\edot} R^Z\|_{\frac{p}{2},r,[0,T]} \leq C \Big(\|\E_{\edot} R^Y\|_{\frac{p}{2},r,[0,T)} \|X\|_{p,[0,T]} + \big(\|Y'\|_{p,q,r,[0,T)} + \sup_{u \in [0,T)} \|Y'_u\|_{L^{r}}\big) \|\X\|_{\frac{p}{2},[0,T]}\Big),
\end{equation}
where the constant $C$ depends only on $p$ and $q$.
\end{lemma}

\begin{proof}
Noting that $R^Z_{s,t} = (\int_s^t Y_u \dd \bX_u - Y_s \delta X_{s,t} - Y'_s \X_{s,t}) + Y'_s \X_{s,t}$, it follows from \eqref{eq: Lqr- estimate for rough stochastic integral} that
\begin{equation*}
\big\|\E_s R^Z_{s,t}\big\|_{L^{r}} \lesssim \|\E_{\edot} R^Y\|_{\frac{p}{2},r,[s,t)} \|X\|_{p,[s,t]} + \|Y'\|_{p,q,r,[s,t)} \|\X\|_{\frac{p}{2},[s,t]} + \|Y'_s\|_{L^{r}} \|\X\|_{\frac{p}{2},[s,t]}.
\end{equation*}
and we similarly have from \eqref{eq: p,q,r,s- estimate for rough stochastic integral} that
\begin{align*}
\|\delta Z_{s,t}\|_{q,r,s} \lesssim \Big(\|&Y\|_{p,q,r,[s,t)} + \|Y'\|_{p,q,r,[s,t)} + \|\E_{\edot} R^Y\|_{\frac{p}{2},r,[s,t)} + \|X\|_{p,[s,t)} \sup_{u \in [0,T)} \|Y'_u\|_{L^{r}}\Big) \|\bX\|_{p,[s,t]}\\
&+ \|Y_s\|_{L^{r}} \|X\|_{p,[s,t]} + \|Y'_s\|_{L^{r}} \|\X\|_{\frac{p}{2},[s,t]},
\end{align*}
from which we deduce the desired bounds.

We have that $Z' = Y \in V^p L^{q,r}$, and it follows from the bounds derived above that $(Z,Z') \in \cV^{p,q,r}_X$, noting that $Z$ is right-continuous in probability by Lemma~\ref{lemma: jump structure of rough stochastic integrals}.
\end{proof}

\begin{lemma}
Let $p \in [2,3)$, $q \in [2,\infty)$, $r \in [q,\infty]$, $\bX = (X,\X) \in \sV^p$, $\tbX =(\tX,\tbbX) \in \sV^p$, $(Y,Y') \in \cV^{p,q,r}_X$ and $(\tY,\tY') \in \cV^{p,q,r}_{\tX}$, such that $\sup_{u \in [0,T]} \|Y_u\|_{L^{r}} < \infty$ and $\sup_{u \in [0,T]} \|\tY_u\|_{L^{r}} < \infty$. Let $(Z,Z') = (\int_0^\cdot Y_u \dd \bX_u,Y)$ and $(\tZ,\tZ') = (\int_0^\cdot \tY_u \dd \tbX_u,\tY)$. Then we have the stability estimates
\begin{equation}\label{eq: p,q,r,s bound on Y-tY}
\begin{split}
\|Y - \tY\|_{p,q,r,[0,T]} &\leq \|R^Y - R^{\tY}\|_{p,q,r,[0,T]} + \sup_{u \in [0,T)} \|Y'_u - \tY'_u\|_{L^{r}} \|X\|_{p,[0,T]}\\
&\quad + \sup_{u \in [0,T)} \|\tY'_u\|_{L^{r}} \|X - \tX\|_{p,[0,T]},
\end{split}
\end{equation}
\begin{equation}\label{eq: stability remainder of rough stochstic integrals}
\begin{split}
\|\E_{\edot} (R^Z - R^{\tZ})\|_{\frac{p}{2},r,[0,T]} \leq C \Big(&\big(\|\E_{\edot} (R^Y - R^{\tY})\|_{\frac{p}{2},r,[0,T)} + \|Y' - \tY'\|_{p,q,r,[0,T)}\big) \|\bX\|_{p,[0,T]}\\
&+ \big(\|\E_{\edot} R^{\tY}\|_{\frac{p}{2},r,[0,T)} + \|\tY'\|_{p,q,r,[0,T)}\big) \|\bX - \tbX\|_{p,[0,T]}\\
&+ \sup_{u \in [0,T)} \|Y'_u - \tY'_u\|_{L^{r}} \|\X\|_{\frac{p}{2},[0,T]} + \sup_{u \in [0,T)} \|\tY'_u\|_{L^{r}} \|\X - \tbbX\|_{\frac{p}{2},[0,T]}\Big),
\end{split}
\end{equation}
and
\begin{equation}\label{eq: stability for rough stochastic integrals}
\begin{split}
\|Z - \tZ&\|_{p,q,r,[0,T]}\\
\leq C \Big(&\Big(\|Y - \tY\|_{p,q,r,[0,T)} + \|Y' - \tY'\|_{p,q,r,[0,T)} + \sup_{u \in [0,T)} \|Y'_u - \tY'_u\|_{L^{r}} (1 + \|X\|_{p,[0,T)})\\
&\quad + \|\E_{\edot} (R^Y - R^{\tY})\|_{\frac{p}{2},r,[0,T)} + \sup_{u \in [0,T)} \|Y_u - \tY_u\|_{L^{r}}\Big) \|\bX\|_{p,[0,T]}\\
&+ \Big(\|\tY\|_{p,q,r,[0,T)} + \|\tY'\|_{p,q,r,[0,T)} + \|\E_{\edot} R^{\tY}\|_{\frac{p}{2},r,[0,T)}\\
&\qquad + \sup_{u \in [0,T)} \|\tY'_u\|_{L^{r}} \big(1 + \|X\|_{p,[0,T]} + \|\tX\|_{p,[0,T]}\big) + \sup_{u \in [0,T)} \|\tY_u\|_{L^{r}}\Big) \|\bX - \tbX\|_{p,[0,T]}\Big),
\end{split}
\end{equation}
where the constant $C$ depends only on $p$ and $q$.
\end{lemma}

\begin{proof}
Noting that
\begin{equation*}
\|\delta Y_{s,t} - \delta \tY_{s,t}\|_{q,r,s} \leq \|R^Y_{s,t} - R^{\tY}_{s,t}\|_{q,r,s} + \|Y'_s - \tY'_s\|_{L^{r}} |\delta X_{s,t}| + \|\tY'_s\|_{L^{r}} |\delta X_{s,t} - \delta \tX_{s,t}|,
\end{equation*}
we immediately deduce the estimate in \eqref{eq: p,q,r,s bound on Y-tY}. Let
\begin{equation*}
\Xi_{s,t} := Y_s \delta X_{s,t} + Y'_s \X_{s,t} - \tY_s \delta \tX_{s,t} - \tY'_s \tbbX_{s,t}
\end{equation*}
for $(s,t) \in \Delta_{[0,T]}$. We then find that
\begin{equation*}
-\delta \Xi_{s,u,t} = (R^Y_{s,u} - R^{\tY}_{s,u}) \delta X_{u,t} + R^{\tY}_{s,u} (\delta X_{u,t} - \delta \tX_{u,t}) + (\delta Y'_{s,u} - \delta \tY'_{s,u}) \X_{u,t} + \delta \tY'_{s,u} (\X_{u,t} - \tbbX_{u,t}),
\end{equation*}
so that
\begin{align*}
\big\|\E_s [\delta \Xi_{s,u,t}]\big\|_{L^{r}} &\leq \|\E_{\edot} (R^Y - R^{\tY})\|_{\frac{p}{2},r,[s,u]} \|X\|_{p,[u,t]} + \|\E_{\edot} R^{\tY}\|_{\frac{p}{2},r,[s,u]} \|X - \tX\|_{p,[u,t]}\\
&\quad + \|Y' - \tY'\|_{p,q,r,[s,u]} \|\X\|_{\frac{p}{2},[u,t]} + \|\tY'\|_{p,q,r,[s,u]} \|\X - \tbbX\|_{\frac{p}{2},[u,t]},
\end{align*}
and
\begin{align*}
\|\delta \Xi_{s,u,t}\|_{q,r,s} &\leq \|R^Y - R^{\tY}\|_{p,q,r,[s,u]} \|X\|_{p,[u,t]} + \|R^{\tY}\|_{p,q,r,[s,u]} \|X - \tX\|_{p,[u,t]}\\
&\quad + \|Y' - \tY'\|_{p,q,r,[s,u]} \|\X\|_{\frac{p}{2},[u,t]} + \|\tY'\|_{p,q,r,[s,u]} \|\X - \tbbX\|_{\frac{p}{2},[u,t]}.
\end{align*}
Using the bounds $\|R^{\tY}\|_{p,q,r,[s,u]} \leq \|\tY\|_{p,q,r,[s,u]} + \sup_{v \in [s,u)} \|\tY'_v\|_{L^{r}} \|\tX\|_{p,[s,u]}$ and
\begin{align*}
\|R^Y - R^{\tY}\|_{p,q,r,[s,u]} &\leq \|Y - \tY\|_{p,q,r,[s,u]} + \sup_{v \in [s,u)} \|Y'_v - \tY'_v\|_{L^{r}} \|X\|_{p,[s,u]}\\
&\quad + \sup_{v \in [s,u)} \|\tY'_v\|_{L^{r}} \|X - \tX\|_{p,[s,u]},
\end{align*}
we obtain
\begin{align*}
&\|\delta \Xi_{s,u,t}\|_{q,r,s}\\
&\leq \Big(\|Y - \tY\|_{p,q,r,[s,u]} + \sup_{v \in [s,u)} \|Y'_v - \tY'_v\|_{L^{r}} \|X\|_{p,[s,u]} + \sup_{v \in [s,u)} \|\tY'_v\|_{L^{r}} \|X - \tX\|_{p,[s,u]}\Big) \|X\|_{p,[u,t]}\\
&\quad + \Big(\|\tY\|_{p,q,r,[s,u]} + \sup_{v \in [s,u)} \|\tY'_v\|_{L^{r}} \|\tX\|_{p,[s,u]}\Big) \|X - \tX\|_{p,[u,t]}\\
&\quad + \|Y' - \tY'\|_{p,q,r,[s,u]} \|\X\|_{\frac{p}{2},[u,t]} + \|\tY'\|_{p,q,r,[s,u]} \|\X - \tbbX\|_{\frac{p}{2},[u,t]}.
\end{align*}
Applying Theorem~\ref{theorem: mild stochastic sewing lemma}, we see that \eqref{eq: stability remainder of rough stochstic integrals} follows from \eqref{eq: L^qr bound full generality sewing}, and \eqref{eq: stability for rough stochastic integrals} follows from \eqref{eq: qrs bound full generality sewing}.
\end{proof}

\subsection{Stochastic controlled vector fields}

It is well known that, given a vector field $f \in C^2_b$ and a (deterministic) controlled path $(Y,Y')$, one has that $(f(Y), \D f(Y) Y')$ is also a controlled path. In the context of rough stochastic control problems (e.g., \cite{FrizLeZhang2024}) and rough stochastic McKean--Vlasov equations (e.g., \cite{FrizHocquetLe2025}), one naturally encounters time-dependent and random vector fields. In \cite{FrizHocquetLe2021} the authors introduced the notion of a \emph{stochastic controlled vector field} as the correct notion to leverage the tradeoff between the spatial and temporal regularity, such that a stochastic controlled path composed with such a vector field gives another stochastic controlled path, along with suitable Lipschitz estimates.

The main results (and in particular their proofs) regarding stochastic controlled vector fields can be transferred from the H\"older setting to our $p$-variation setting almost verbatim. For brevity, we will therefore omit the proofs of the results in this subsection and refer the reader to \cite{FrizHocquetLe2021} for full details.

\begin{definition}\label{def stoch controlled vector field}
Let $p \in [2,3)$, $q \in [2,\infty)$, $r \in [q,\infty]$, $\gamma \in \{2,3\}$, let $E, W, \bar{W}$ be finite-dimensional Banach spaces, and let $X \in V^p$ be an $E$-valued path. We call $(f,f')$ a \emph{stochastic controlled vector field} if the following hold.
\begin{enumerate}[(i)]
\item \begin{equation*}
(f,f') \colon \Omega \times [0,T] \to C_b^{\gamma}(W;\bar{W}) \times C_b^{\gamma - 1}\big(W;\cL(E;\bar{W})\big)
\end{equation*}
are strongly\footnote{This essentially just means that $f$ and $f'$ take values in a separable subspace; see \cite[Sections~2.1 \& 3.3]{FrizHocquetLe2021} for a discussion of the necessity of the strong measurability assumption here.} progressively measurable, and
\[ \sup_{s \in [0,T]} \big\| \|f_s\|_{C^\gamma_b} \big\|_{L^r} + \sup_{s \in [0,T]} \big\| \|f'_s\|_{C^{\gamma-1}_b} \big\|_{L^r} < \infty. \]
\item Setting
\begin{equation}\label{eq: defintion bracket of vector field}
\llbracket g \rrbracket_{p,q,r,[s,t]} := \Big\| \sup_{x \in W} |\delta g(x)| \Big\|_{p,q,r,[s,t]},
\end{equation}
we have that $\llbracket f \rrbracket_{p,q,r,[0,T]}$, $\llbracket f' \rrbracket_{p,q,r,[0,T]}$ and $\llbracket \D f \rrbracket_{p,q,r,[0,T]}$ are finite.
\item The map $\E_{\edot} R^f$, given by $(s,t) \mapsto \E_s R^f_{s,t} := \E_s [\delta f_{s,t}] - f'_s \delta X_{s,t}$, satisfies
\begin{equation*}
\llbracket \E_{\edot} R^f \rrbracket_{\frac{p}{2},r,[0,T]} := \bigg( \sup_{\cP \subset [0,T]} \sum_{[s,t] \in \cP} \Big\| \sup_{x \in W} \big| \E_s \big[ R_{s,t}^f(x) \big] \big| \Big\|^{\frac{p}{2}}_{L^r} \bigg)^{\hspace{-2pt}\frac{2}{p}} < \infty.
\end{equation*}
\end{enumerate}
We denote the space of stochastic controlled vector fields by $\cV^{p,q,r}_X C^\gamma_b$.
\end{definition}

In the following we will also use the notation $|g|_\infty := \sup_{x} |g(x)|$.

\begin{lemma}\label{lemma: (f(Y), f'(Y)Y') is a controlled path}
Let $p \in [2,3)$, $q \in [2,\infty)$, $r \in [2q,\infty]$, $X \in V^p$, $(Y,Y') \in \cV^{p,q,r}_X$ and $(f,f') \in \cV^{p,q,\infty}_X C^2_b$. Then $(f(Y),\D f(Y) Y' + f'(Y)) \in \cV^{p,q,\frac{r}{2}}_X$ and, for every $(s,t) \in \Delta_{[0,T]}$, we have the estimates
\begin{align}
\| \delta f(Y)_{s,t} \|_{q,\frac{r}{2},s} &\leq \llbracket f \rrbracket_{p,q,\infty,[s,t]} + \big\| |\D f_s|_\infty \big\|_{L^\infty} \|Y\|_{p,q,r,[s,t]},\label{eq: path p,q,r/2,s bound for (f(Y), Df(Y)Y')}\\
\big\| \E_s \big[ R^{f(Y)}_{s,t} \big] \big\|_{L^{\frac{r}{2}}} &\leq \big\| \|\D f_s\|_{C^1_b} \big\|_{L^\infty} \|Y\|_{p,q,r,[s,t]}^2 + \llbracket \D f \rrbracket_{p,q,\infty,[s,t]} \|Y\|_{p,q,r,[s,t]}\nonumber\\
&\quad + \big\| |\D f_s|_\infty \big\|_{L^\infty} \| \E_{\edot} R^Y \|_{\frac{p}{2},r,[s,t]} + \llbracket \E_{\edot} R^f \rrbracket_{\frac{p}{2},\infty,[s,t]},\label{eq: remainder bound for (f(Y), Df(Y)Y')}\\
\big\| \delta (\D f(Y) Y')_{s,t} \big\|_{q,\frac{r}{2},s} &\leq \big\| \|\D f_s\|_{C^1_b} \big\|_{L^\infty} \|Y'_s\|_{L^r} \|Y\|_{p,q,r,[s,t]} + \big\| |\D f_t|_\infty \big\|_\infty \|Y'\|_{p,q,r,[s,t]}\nonumber\\
&\quad + \|Y'_s\|_{L^r} \llbracket \D f \rrbracket_{p,q,\infty,[s,t]},\label{eq: derivative bound for (f(Y), Df(Y)Y')}
\end{align}
and
\begin{equation*}
\big\| \delta f'(Y)_{s,t} \big\|_{q,\frac{r}{2},s} \leq \big\| \|f'_s\|_{C^1_b} \big\|_{L^\infty} \|Y\|_{p,q,r,[s,t]} + \llbracket f' \rrbracket_{p,q,\infty,[s,t]}.
\end{equation*}
\end{lemma}

\begin{remark}
In \cite{FrizHocquetLe2021} the assumption on the stochastic controlled vector field can be weakened to $(f,f') \in \cV^{p,q,\infty} C^{\gamma-1}_b$ for $\gamma \in (2,3]$. As is typical for the $p$-variation setting, this is not possible here (even when $f$ is deterministic and time-homogeneous). More precisely, if $f \in C^{\gamma-1}_b$ then, in place of \eqref{eq: derivative bound for (f(Y), Df(Y)Y')}, we could only hope to have
\begin{equation*}
\big\| \delta (\D f(Y) Y')_{s,t} \big\|_{q,\frac{r}{2},s} \lesssim \|Y'_s\|_{L^r} \|Y\|^{\gamma-2}_{p,q,r,[s,t]} + \|Y'\|_{p,q,r,[s,t]}
\end{equation*}
where $\gamma - 2 < 1$, which is insufficient to obtain $\D f(Y) Y' \in V^p L^{q,\frac{r}{2}}$.
\end{remark}

As previously observed in an analogous setting in \cite{FrizHocquetLe2021}, the reduction of integrability from $r$ to $\frac{r}{2}$ in Lemma~\ref{lemma: (f(Y), f'(Y)Y') is a controlled path} poses a problem for the construction of solutions to RSDEs. In particular, if $r \neq \frac{r}{2}$ then one cannot hope to obtain an invariance property of the solution map due to this loss of integrability. We will therefore restrict to the case $r = \infty$ for the rest of this and the next section.

\smallskip

For $\gamma \in \{2,3\}$, and a stochastic controlled vector field $(f,f') \in \cV^{p,q,\infty}_X C^\gamma_b$, we write
\[ \|f,f'\|_{\gamma,\infty} := \sup_{s \in [0,T]} \big\| \|f_s\|_{C^\gamma_b} \big\|_{L^\infty} + \sup_{s \in [0,T]} \big\| \|f'_s\|_{C^{\gamma-1}_b} \big\|_{L^\infty}, \]
and
\[ \llbracket f, f' \rrbracket_{p,q,\infty,[0,T]} := \llbracket f \rrbracket_{p,q,\infty,[0,T]} + \llbracket f' \rrbracket_{p,q,\infty,[0,T]} + \llbracket \D f \rrbracket_{p,q,\infty,[0,T]} + \llbracket \E_{\edot} R^f \rrbracket_{\frac{p}{2},\infty,[0,T]}. \]
For the norm in \eqref{eq: defintion bracket of vector field}, we also adopt the shorthand $\llbracket \hspace{1pt} \cdot \hspace{1pt} \rrbracket_{p,q,[0,T]} := \llbracket \hspace{1pt} \cdot \hspace{1pt} \rrbracket_{p,q,q,[0,T]}$.

\begin{lemma}\label{lemma: f(Y)- f(tY)}
Let $p \in [2,3)$, $q \in [2,\infty)$, $X, \tX \in V^p$, $(Y,Y') \in \cV^{p,q,\infty}_X$ and $(\tY,\tY') \in \cV^{p,q,\infty}_{\tX}$, and let $(f,f') \in \cV^{p,q,\infty}_X C^3_b$ and $(\tilde{f},\tilde{f}') \in \cV^{p,q,\infty}_{\tX} C^2_b$ be stochastic controlled vector fields such that $(\D f,\D f') \in \cV^{p,q,\infty}_X C^2_b$, and suppose that $\|Y\|_{p,q,\infty,[0,T]}$, $\|\tY\|_{p,q,\infty,[0,T]}$, $\|Y'\|_{p,q,\infty,[0,T]}$, $\|\tY'\|_{p,q,\infty,[0,T]}$, $\|\E_{\edot} R^Y\|_{\frac{p}{2},\infty,[0,T]}$, $\|\E_{\edot} R^{\tY}\|_{\frac{p}{2},\infty,[0,T]}$, $\sup_{s \in [0,T]} \|Y'_s\|_{L^\infty}$ and $\sup_{s \in [0,T]} \|\tY'_s\|_{L^\infty}$ are all bounded by some constant $L > 0$. Then, writing
\[ (Z,Z') := \big( f(Y), \D f(Y) Y' + f'(Y) \big) \text{~~and~~} (\tZ,\tZ') := \big( \tilde{f}(\tY), \D \tilde{f}(\tY) \tY' + \tilde{f}'(\tY) \big), \]
we have the estimates
\begin{equation}\label{eq: p,L^q-bound for f(Y)- f(tY)}
\begin{split}
\|Z - \tZ\|_{p,q,[0,T]} \leq C \Big( &\big\| 1 \wedge |Y_0 - \tY_0| \big\|_{L^q} + \|Y - \tY\|_{p,q,[0,T]}\\
&+ \llbracket f - \tilde{f} \rrbracket_{p,q,[0,T]} + \sup_{s \in [0,T]} \big\| \|f_s - \tilde{f}_s\|_{C^1_b} \big\|_{L^q} \Big),
\end{split}
\end{equation}
\begin{equation*}
\begin{split}
\| Z' - \tZ' \|_{p,q,[0,T]} \leq C \Big( &\big\| 1 \wedge |Y_0 - \tY_0| \big\|_{L^q} + \|Y'_0 - \tY'_0\|_{L^q} + \|Y - \tY\|_{p,q,[0,T]}\\
&+ \|Y' - \tY'\|_{p,q,[0,T]} + \llbracket \D f - \D \tilde{f} \rrbracket_{p,q,[0,T]} + \llbracket f' - \tilde{f}' \rrbracket_{p,q,[0,T]}\\
&+ \sup_{s \in [0,T]} \big\| \|\D f_s - \D \tilde{f}_s\|_{C^1_b} \big\|_{L^q} + \sup_{s \in [0,T]} \big\| \|f'_s - \tilde{f}'_s\|_{C^1_b} \big\|_{L^q} \Big),
\end{split}
\end{equation*}
and
\begin{equation*}
\begin{split}
&\big\| \E_{\edot} \big( R^Z - R^{\tZ} \big) \big\|_{\frac{p}{2},q,[0,T]}\\
&\leq C \Big( \big\| 1 \wedge |Y_0 - \tY_0| \big\|_{L^q} + \|Y - \tY\|_{p,q,[0,T]} + \big\| \E_{\edot} (R^Y - R^{\tY}) \big\|_{\frac{p}{2},q,[0,T]}\\
&\qquad + \sup_{s \in [0,T]} \big\| \|f_s - \tilde{f}_s\|_{C^2_b} \big\|_{L^q} + \llbracket \D f - \D \tilde{f} \rrbracket_{p,q,[0,T]} + \big\llbracket \E_{\edot} (R^f - R^{\tilde{f}}) \big\rrbracket_{\frac{p}{2},q,[0,T]} \Big),
\end{split}
\end{equation*}
where the constant $C$ depends only on $L$, $\|f,f'\|_{3,\infty}$, $\|\tilde{f},\tilde{f}'\|_{2,\infty}$, $\llbracket f, f' \rrbracket_{p,q,\infty,[0,T]}$, $\llbracket \tilde{f}, \tilde{f}' \rrbracket_{p,q,\infty,[0,T]}$ and $\llbracket \D f, \D f' \rrbracket_{p,q,\infty,[0,T]}$.
\end{lemma}

\section{Rough stochastic differential equations}\label{section: rough stochastic differential equations}

Given a c\`adl\`ag rough path $\bX$ and a sufficiently integrable c\`adl\`ag martingale $M$, we consider the RSDE given by
\begin{equation}\label{eq: dynamical description of a solution to RSDE}
Y_t = y_0 + \int_0^t b_s(Y_s) \dd s + \int_0^t \sigma_s(Y_{s-}) \dd M_s + \int_0^t f_s(Y_s) \dd \bX_s
\end{equation}
for $t \in [0,T]$. Here, $b$ and $\sigma$ are random bounded Lipschitz functions in the sense of Definition~\ref{def.randomcoef} below, and $(f,f')$ is a stochastic controlled vector field in the sense of Definition~\ref{def stoch controlled vector field}. In particular, we interpret the second integral as a classical It\^o integral against $M$, and we interpret the last integral as the rough stochastic integral of the stochastic controlled path $(f(Y),\D f(Y) Y' + f'(Y))$ with respect to $\bX$, in the sense of Lemma~\ref{lemma: existence of rough stochastic integrals}.

\begin{definition}\label{def.randomcoef}
Let $W, \bar{W}$ be finite dimensional Banach spaces, and fix a Borel set $S \subset W$. Let $(t,\omega) \mapsto g_t(\omega,\cdot)$ be a strongly measurable stochastic process from $\Omega \times [0,T] \to C(S;\bar{W})$ (in the sense of a family of strongly measurable random variables). We say that the function $g$ is \emph{random bounded Lipschitz} if it is progressively measurable from $\Omega \times [0,T] \to C_b^1(S;\bar{W})$, and is uniformly bounded and uniformly Lipschitz, in the sense that (with a slight abuse of notation)
\begin{equation}\label{eq: random Lipschitz constant}
\|g\|_{C^1_b} := \sup_{t \in [0,T]} \esssup_{\omega \in \Omega} \, \sup_{x \in S} |g_t(\omega,x)| + \sup_{t \in [0,T]} \esssup_{\omega \in \Omega} \, \sup_{x \neq \tilde{x}} \frac{|g_t(\omega,x) - g_t(\omega,\tilde{x})|}{|x - \tilde{x}|} < \infty.
\end{equation}
We say that $g$ is \emph{predictable bounded Lipschitz} if it is random bounded Lipschitz and is also predictable.
\end{definition}

\begin{theorem}\label{theorem: existence and estimates for solutions to RSDEs}
Let $p \in [2,3)$ and $q \in [2,\infty)$, and suppose that the filtration $(\cF_t)_{t \in [0,T]}$ satisfies the usual conditions. Let $b$ be random bounded Lipschitz, let $\sigma$ be predictable bounded Lipschitz, and let $(f,f') \in \cV^{p,q,\infty}_X C^3_b$ be a stochastic controlled vector field such that $(\D f,\D f') \in \cV^{p,q,\infty}_X C^2_b$. Let $y_0 \in L^q$ be $\cF_0$-measurable, let $M \in V^p L^{q,\infty}$ be a c\`adl\`ag martingale, and let $\bX = (X,\X) \in \sV^p$ be a c\`adl\`ag rough path.

Then there exists a process $Y$, which is unique up to indistinguishability, such that $Y$ has almost surely c\`adl\`ag sample paths, $(Y,Y') \in \cV^{p,q,\infty}_X$ is a stochastic controlled path, where $Y' = f(Y)$, and such that, almost surely, \eqref{eq: dynamical description of a solution to RSDE} holds for every $t \in [0,T]$.

Moreover, if $y_0, \ty_0 \in L^q$ are $\cF_0$-measurable, $M, \tM \in V^p L^{q,\infty}$ are two c\`adl\`ag martingales and $\bX, \tbX \in \sV^p$ are two c\`adl\`ag rough paths, such that the norms $\|M\|_{p,q,\infty,[0,T]}$, $\|\tM\|_{p,q,\infty,[0,T]}$, $\|\bX\|_{p,[0,T]}$ and $\|\tbX\|_{p,[0,T]}$ are all bounded by some constant $L > 0$, and if $Y, \tY$ are the solutions to \eqref{eq: dynamical description of a solution to RSDE} corresponding to the data $(y_0,M,\bX)$ and $(\ty_0,\tM,\tbX)$ respectively, then we have that
\begin{equation}\label{eq: Lipschitz continuity of solution map}
\begin{split}
\|&Y - \tY\|_{p,q,[0,T]} + \|Y' - \tY'\|_{p,q,[0,T]} + \|\E_{\edot} (R^Y - R^{\tY})\|_{\frac{p}{2},q,[0,T]}\\
&\leq C \big( \|y_0 - \ty_0\|_{L^q} + \|M - \tM\|_{p,q,[0,T]} + \|\bX - \tbX\|_{p,[0,T]} \big),
\end{split}
\end{equation}
where the constant $C$ depends only on $p, q, T, L$, $\|b\|_{C^1_b}$, $\|\sigma\|_{C^1_b}$, $\|f,f'\|_{3,\infty}$, $\llbracket f, f' \rrbracket_{p,q,\infty,[0,T]}$ and $\llbracket \D f, \D f' \rrbracket_{p,q,\infty,[0,T]}$.
\end{theorem}

Due to its substantial length, below we will only present the proof in the case when the vector fields $b, \sigma \in C^1_b$ and $f \in C^3_b$ are deterministic and time-homogeneous. This leads to a significant simplification of the bounds in Section~\ref{section: rough stochastic integration}, which are used extensively. The interested reader can nonetheless convince themself that one can straightforwardly extend the proof to the general setting using the full generality of the estimates provided in Section~\ref{section: rough stochastic integration}.

\begin{proof}[Proof of Theorem~\ref{theorem: existence and estimates for solutions to RSDEs}]
The basic strategy of the proof will be to establish a contraction mapping on a subset of a suitable Banach space, in order to apply the Banach fixed-point theorem. A non-trivial step therein is the identification of the correct Banach space. To this end, we let $\mathbf{B}_T$ denote the space of stochastic controlled paths $(Y,Y') \in \cV^{p,q,\infty}_X$ such that $Y$ has almost surely c\`adl\`ag sample paths, $Y_0 = y_0$ and $Y'_0 = f(y_0)$, and such that
\begin{equation}\label{eq: bounds defining B_T}
\|Y'\|_{p,q,\infty,[0,T]} \vee \sup_{s \in [0,T]} \|Y'_s\|_{L^\infty} \leq \|f\|_{C^3_b} \quad \text{and} \quad \|Y\|_{p,q,\infty,[0,T]} \vee \|\E_{\edot} R^Y\|_{\frac{p}{2},\infty,[0,T]} \leq 1.
\end{equation}

Setting $Y_t = y_0 + f(y_0) \delta X_{0,t}$ and $Y'_t = f(y_0)$ for all $t \in [0,T]$, we see that, for sufficiently small $T > 0$, we have $(Y,Y') \in \mathbf{B}_T$, so that in particular $\mathbf{B}_T \neq \emptyset$.

For each $(Y,Y') \in \mathbf{B}_T$, we define
\begin{equation*}
\Phi(Y,Y') := \big(\phi(Y),\phi(Y)'\big) := \bigg(y_0 + \int_0^\cdot b(Y_u) \dd u + \int_0^\cdot \sigma(Y_{u-}) \dd M_u + \int_0^\cdot f(Y_u) \dd \bX_u, f(Y)\bigg),
\end{equation*}
where $\int_0^\cdot f(Y_u) \dd \bX_u$ is the rough stochastic integral of $(f(Y),\D f(Y) Y')$---which is a stochastic controlled path in $\cV^{p,q,\infty}_X$ by Lemma~\ref{lemma: (f(Y), f'(Y)Y') is a controlled path}---with respect to $\bX$.

For $\eta > 1$, we also define the metric
\begin{equation*}
\begin{split}
d^\eta\big((Y,Y')&,(\tY,\tY')\big) := \|Y' - \tY'\|_{p,q,[0,T]}\\
&+ \eta \Big(\|Y - \tY\|_{p,q,[0,T]} + \|\E_{\edot} (R^Y - R^{\tY})\|_{\frac{p}{2},q,[0,T]} + \Big\|1 \wedge \sup_{u \in [0,T]} |Y_u - \tY_u|\Big\|_{L^q}\Big).
\end{split}
\end{equation*}
It follows from a straightforward extension of Lemma~\ref{lemma: V^pL^q is a Banach space} that $(\mathbf{B}_T,d^\eta)$ is a complete metric space, where in particular we can verify that limit points satisfy the bounds in \eqref{eq: bounds defining B_T} using Lemma~\ref{lemma: closedness of set of bounded paths in V^pL^q}. We note that the final term in the definition of the metric $d^\eta$ is required to ensure that limits under this metric inherit c\`adl\`ag sample paths.

\emph{Step 1. (Invariance)}
Let $(Y,Y') \in \mathbf{B}_T$. Using the bound in \eqref{eq: path p,q,r/2,s bound for (f(Y), Df(Y)Y')}, we note that
\begin{equation*}
\|\phi(Y)'\|_{p,q,\infty,[0,T]} \vee \sup_{s \in [0,T]} \|\phi(Y)'_s\|_{L^\infty} \leq \|f\|_{C^3_b}.
\end{equation*}
Applying the bound in \eqref{eq: bound for E R^Z} from Lemma~\ref{lemma:  rough stochastic integrals as controlled path}, followed by those in \eqref{eq: remainder bound for (f(Y), Df(Y)Y')} and \eqref{eq: derivative bound for (f(Y), Df(Y)Y')}, we find that $\|\E_{\edot} R^{\int_0^\cdot f(Y_u) \dd \bX_u}\|_{\frac{p}{2},\infty,[0,T]} \lesssim \|\bX\|_{p,[0,T]}$, so that
\begin{equation*}
\big\| \E_{\edot} R^{\phi(Y)} \big\|_{\frac{p}{2},\infty,[0,T]} \leq \bigg\|\int_0^\cdot b(Y_u) \dd u\bigg\|_{\frac{p}{2},\infty,[0,T]} + \big\| \E_{\edot} R^{\int_0^\cdot f(Y_u) \dd \bX_u} \big\|_{\frac{p}{2},\infty,[0,T]} \lesssim T + \|\bX\|_{p,[0,T]}.
\end{equation*}

By the bound in \eqref{eq: bound for Z pqr} and those in Lemma~\ref{lemma: (f(Y), f'(Y)Y') is a controlled path}, we obtain $\|\int_0^\cdot f(Y_u) \dd \bX_u\|_{p,q,\infty,[0,T]} \lesssim (1 + \|X\|_{p,[0,T]}) \|\bX\|_{p,[0,T]}$, and we have from the (conditional) BDG inequality that
\begin{equation*}
\bigg\|\int_s^t \sigma(Y_{u-}) \dd M_u\bigg\|_{q,\infty,s} \lesssim \bigg\|\int_s^t \sigma(Y_{u-})^{\otimes 2} \dd [M]_u\bigg\|_{\frac{q}{2},\infty,s}^{\frac{1}{2}} \leq \|\sigma\|_{C^1_b} \|[M]\|_{\frac{p}{2},\frac{q}{2},\infty,[s,t]}^{\frac{1}{2}},
\end{equation*}
so that, recalling Lemma~\ref{lemma: quadratic variation is all that matters for p,q,r}, $\|\int_0^\cdot \sigma(Y_{u-}) \dd M_u\|_{p,q,\infty,[0,T]} \lesssim \|M\|_{p,q,\infty,[0,T]}$. Thus,
\begin{equation*}
\|\phi(Y)\|_{p,q,\infty,[0,T]} \vee \big\| \E_{\edot} R^{\phi(Y)} \big\|_{\frac{p}{2},\infty,[0,T]} \leq C_1 \big(T + \|M\|_{p,q,\infty,[0,T]} + \|\bX\|_{p,[0,T]}\big),
\end{equation*}
where the constant $C_1 > 0$ depends only on $p, q, \|b\|_{C^1_b}, \|\sigma\|_{C^1_b}, \|f\|_{C^3_b}$ and $\|X\|_{p,[0,T]}$.

By Lemma~\ref{lemma: right-continuity of controls}, there exists a $T_1 > 0$ sufficiently small that
\begin{equation*}
T_1 + \|M\|_{p,q,\infty,[0,T_1]} + \|\bX\|_{p,[0,T_1]} \leq \frac{1}{C_1}.
\end{equation*}
It is clear that $\phi(Y)_0 = y_0$ and $\phi(Y)'_0 = f(y_0)$, and, by Lemma~\ref{lemma: rough stochastic integral cadlag sample paths}, we may take $\phi(Y)$ to have c\`adl\`ag sample paths. It follows that $\Phi(Y,Y') \in \mathbf{B}_T$ for every $T \leq T_1$. That is, $\mathbf{B}_T$ is invariant under $\Phi$.

\emph{Step 2. (Contraction)}
Let $\eta > 1$ and $T \in (0,T_1]$. Let $(Y,Y'), (\tY,\tY') \in \mathbf{B}_T$. We have that
\begin{align*}
\|&\phi(Y) - \phi(\tY)\|_{p,q,[0,T]} \leq \bigg\|\int_0^\cdot \big(b(Y_u) - b(\tY_u)\big) \dd u\bigg\|_{p,q,[0,T]}\\
&+ \bigg\|\int_0^\cdot \big(\sigma(Y_{u-}) - \sigma(\tY_{u-})\big) \dd M_u\bigg\|_{p,q,[0,T]} + \bigg\|\int_0^\cdot \big(f(Y_u) - f(\tY_u)\big) \dd \bX_u\bigg\|_{p,q,[0,T]}\\
&=: I_1 + I_2 + I_3,
\end{align*}
where $I_1 \lesssim T \sup_{u \in [0,T]} \|Y_u - \tY_u\|_{L^q} \leq T \|Y - \tY\|_{p,q,[0,T]}$. Using the BDG inequality, and then applying Lemma~\ref{lemma: contraction for integral against quad. variation for Ito approach}, we have that\footnote{One may wonder why we do not bound the martingale term using the bound provided by the stochastic sewing lemma, namely \eqref{eq: int against martingale estimate}. However, this would require $\sigma \in C^2_b$ (similarly to Lemma~\ref{lemma: (f(Y), f'(Y)Y') is a controlled path}). Instead, we use the BDG inequality to switch to quadratic variation and apply Lemma~\ref{lemma: contraction for integral against quad. variation for Ito approach}, which only requires $\sigma \in C^1_b$.}
\begin{align*}
&\bigg\|\int_s^t \big(\sigma(Y_{u-}) - \sigma(\tY_{u-})\big) \dd M_u\bigg\|_{L^q} \lesssim \bigg\|\int_s^t \big(\sigma(Y_{u-}) - \sigma(\tY_{u-})\big)^{\otimes 2} \dd [M]_u\bigg\|_{L^{\frac{q}{2}}}^{\frac{1}{2}}\\
&\lesssim \|Y - \tY\|_{p,q,[0,T]} \|[M]\|_{\frac{p}{2},\frac{q}{2},\infty,[s,t]}^{\frac{1}{2}} \lesssim \|Y - \tY\|_{p,q,[0,T]} \|M\|_{p,q,\infty,[s,t]},
\end{align*}
so that $I_2 \lesssim \|Y - \tY\|_{p,q,[0,T]} \|M\|_{p,q,\infty,[0,T]}$.

Using the estimate in \eqref{eq: stability for rough stochastic integrals} with $r = q$ and $\tbX = \bX$, combined with the bounds in Lemma~\ref{lemma: f(Y)- f(tY)}, we find that
\begin{equation*}
I_3 \lesssim \big(\|Y - \tY\|_{p,q,[0,T]} + \|Y' - \tY'\|_{p,q,[0,T]} + \|\E_{\edot} (R^Y - R^{\tY})\|_{\frac{p}{2},q,[0,T]}\big) (1 + \|X\|_{p,[0,T]}) \|\bX\|_{p,[0,T]}.
\end{equation*}
We thus have that
\begin{equation}\label{eq: contraction  |phi(Y)-phi(tY)| bound}
\begin{split}
\|\phi(Y) - \phi(\tY)\|_{p,q,[0,T]} \lesssim \big(&\|Y - \tY\|_{p,q,[0,T]} + \|Y' - \tY'\|_{p,q,[0,T]} + \|\E_{\edot} (R^Y - R^{\tY})\|_{\frac{p}{2},q,[0,T]}\big)\\
&\quad \times (1 + \|X\|_{p,[0,T]}) \big(T + \|M\|_{p,q,\infty,[0,T]} + \|\bX\|_{p,[0,T]}\big).
\end{split}
\end{equation}

Similarly, applying the estimate in \eqref{eq: stability remainder of rough stochstic integrals} with $r = q$ and $\tbX = \bX$, combined with the bounds in Lemma~\ref{lemma: f(Y)- f(tY)}, we find that
\begin{align*}
\big\|&\E_{\edot} \big(R^{\int_0^\cdot f(Y_u) \dd \bX_u} - R^{\int_0^\cdot f(\tY_u) \dd \bX_u}\big)\big\|_{\frac{p}{2},q,[0,T]}\\
&\lesssim \big(\|Y - \tY\|_{p,q,[0,T]} + \|Y' - \tY'\|_{p,q,[0,T]} + \|\E_{\edot} (R^Y - R^{\tY})\|_{\frac{p}{2},q,[0,T]}\big) \|\bX\|_{p,[0,T]}.
\end{align*}
Since
\begin{equation*}
\E_s \big[ R^{\phi(Y)}_{s,t} - R^{\phi(\tY)}_{s,t} \big] = \E_s \bigg[\int_s^t \big(b(Y_u) - b(\tY_u)\big) \dd u\bigg] + \E_s \Big[R^{\int_0^\cdot f(Y_u) \dd \bX_u}_{s,t} - R^{\int_0^\cdot f(\tY_u) \dd \bX_u}_{s,t}\Big],
\end{equation*}
we obtain
\begin{equation}\label{eq: contraction remainder bound}
\begin{split}
\big\|&\E_{\edot} \big(R^{\phi(Y)} - R^{\phi(\tY)}\big)\big\|_{\frac{p}{2},q,[0,T]}\\
&\lesssim \big(\|Y - \tY\|_{p,q,[0,T]} + \|Y' - \tY'\|_{p,q,[0,T]} + \|\E_{\edot} (R^Y - R^{\tY})\|_{\frac{p}{2},q,[0,T]}\big) \big(T + \|\bX\|_{p,[0,T]}\big).
\end{split}
\end{equation}
It also follows from the bound in \eqref{eq: p,L^q-bound for f(Y)- f(tY)} that
\begin{equation}\label{eq: contraction derivative bound}
\|\phi(Y)' - \phi(\tY)'\|_{p,q,[0,T]} = \|f(Y) - f(\tY)\|_{p,q,[0,T]} \lesssim \|Y - \tY\|_{p,q,[0,T]}.
\end{equation}

It is straightforward to see that
\begin{equation*}
\bigg\|\sup_{t \in [0,T]} \bigg|\int_0^t \big(b(Y_u) - b(\tY_u)\big) \dd u\bigg|\bigg\|_{L^q} \lesssim T \Big\|1 \wedge \sup_{u \in [0,T]} |Y_u - \tY_u|\Big\|_{L^q}.
\end{equation*}
By the BDG inequality and Lemma~\ref{lemma: contraction for integral against quad. variation for Ito approach}, we have that
\begin{equation*}
\begin{split}
&\bigg\|\sup_{t \in [0,T]} \bigg|\int_0^t \big(\sigma(Y_{u-}) - \sigma(\tY_{u-})\big) \dd M_u\bigg|\bigg\|_{L^q} \lesssim \bigg\|\int_0^T \big(\sigma(Y_{u-}) - \sigma(\tY_{u-})\big)^{\otimes 2} \dd [M]_u\bigg\|_{L^\frac{q}{2}}^{\frac{1}{2}}\\
&\lesssim \|Y - \tY\|_{p,q,[0,T]} \|[M]\|_{\frac{p}{2},\frac{q}{2},\infty,[0,T]}^{\frac{1}{2}} \lesssim \|Y - \tY\|_{p,q,[0,T]} \|M\|_{p,q,\infty,[0,T]}.
\end{split}
\end{equation*}
Applying the estimate in \eqref{eq: uniform bound for rough stochastic integral} with $r = q$, combined with the bounds in Lemma~\ref{lemma: f(Y)- f(tY)}, we find that
\begin{equation*}
\begin{split}
&\bigg\|\sup_{t \in [0,T]} \bigg|\int_0^t \big(f(Y_u) - f(\tY_u)\big) \dd \bX_u\bigg|\bigg\|_{L^q}\\
&\lesssim \big(\|Y - \tY\|_{p,q,[0,T]} + \|Y' - \tY'\|_{p,q,[0,T]} + \|\E_{\edot} (R^Y - R^{\tY})\|_{\frac{p}{2},q,[0,T]}\big) (1 + \|X\|_{p,[0,T]}) \|\bX\|_{p,[0,T]}.
\end{split}
\end{equation*}

Combining the bounds in \eqref{eq: contraction  |phi(Y)-phi(tY)| bound}, \eqref{eq: contraction remainder bound} and \eqref{eq: contraction derivative bound} with the uniform bounds derived above, we have that there exists a constant $C_2 > \frac{1}{2}$, which depends only on $p, q, \|b\|_{C^1_b}, \|\sigma\|_{C^1_b}, \|f\|_{C^3_b}$ and $\|X\|_{p,[0,T]}$, such that
\begin{align*}
d^\eta&\big(\Phi(Y,Y'),\Phi(\tY,\tY')\big) \leq C_2 \Big(\|Y - \tY\|_{p,q,[0,T]}\\
&+ \eta \Big(\|Y - \tY\|_{p,q,[0,T]} + \|Y' - \tY'\|_{p,q,[0,T]} + \|\E_{\edot} (R^Y - R^{\tY})\|_{\frac{p}{2},q,[0,T]}\\
&\qquad \qquad + \Big\|1 \wedge \sup_{u \in [0,T]} |Y_u - \tY_u|\Big\|_{L^q}\Big) \times \big(T + \|M\|_{p,q,\infty,[0,T]} + \|\bX\|_{p,[0,T]}\big)\Big).
\end{align*}
Letting $\eta = 2 C_2 > 1$, we have that
\begin{align*}
d^{\eta}&\big(\Phi(Y,Y'),\Phi(\tY,\tY')\big) \leq \frac{\eta}{2} \|Y - \tY\|_{p,q,[0,T]}\\
&+ 2 C_2^2 \Big(\|Y - \tY\|_{p,q,[0,T]} + \|Y' - \tY'\|_{p,q,[0,T]} + \|\E_{\edot} (R^Y - R^{\tY})\|_{\frac{p}{2},q,[0,T]}\\
&\qquad \qquad + \Big\|1 \wedge \sup_{u \in [0,T]} |Y_u - \tY_u|\Big\|_{L^q}\Big) \times \big(T + \|M\|_{p,q,\infty,[0,T]} + \|\bX\|_{p,[0,T]}\big).
\end{align*}
By Lemma~\ref{lemma: right-continuity of controls}, there exists a $T_2 \in (0,T_1]$ sufficiently small such that
\begin{equation*}
T_2 + \|M\|_{p,q,\infty,[0,T_2]} + \|\bX\|_{p,[0,T_2]} \leq \frac{1}{4 C_2^2}.
\end{equation*}
Then, for any $T \leq T_2$, we obtain
\begin{align*}
&d^{\eta} \big(\Phi(Y,Y'),\Phi(\tY,\tY')\big)\\
&\leq \frac{\eta + 1}{2} \|Y - \tY\|_{p,q,[0,T]} + \frac{1}{2} \Big(\|Y' - \tY'\|_{p,q,[0,T]}\\
&\qquad + \|\E_{\edot} (R^Y - R^{\tY})\|_{\frac{p}{2},q,[0,T]} + \Big\|1 \wedge \sup_{u \in [0,T]} |Y_u - \tY_u|\Big\|_{L^q}\Big)\\
&\leq \frac{\eta + 1}{2 \eta} d^{\eta} \big((Y,Y'),(\tY,\tY')\big).
\end{align*}
Thus, by the Banach fixed point theorem, the map $\Phi$ has a unique fixed point, which is the unique solution to \eqref{eq: dynamical description of a solution to RSDE} over the time interval $[0,T]$, for $T \leq T_2$.

\emph{Step 3. (Local estimate)}
Let $(y_0,M,\bX)$ and $(\ty_0,\tM,\tbX)$ be data as in the statement of the theorem, and let $Y, \tY$ be the corresponding local solutions, which, by the previous step, exist on the time intervals $[0,T_2]$ and $[0,\widetilde{T}_2]$ respectively. Let $T \leq T_2 \wedge \widetilde{T}_2$.

Since $(Y,Y'), (\tY,\tY') \in \mathbf{B}_T$, we have that
\begin{equation*}
\|Y'\|_{p,q,\infty,[0,T]} \vee \|\tY'\|_{p,q,\infty,[0,T]} \vee \sup_{s \in [0,T]} \|Y'_s\|_{L^\infty} \vee \sup_{s \in [0,T]} \|\tY'_s\|_{L^\infty} \leq \|f\|_{C^3_b}
\end{equation*}
and
\begin{equation*}
\|Y\|_{p,q,\infty,[0,T]} \vee \|\tY\|_{p,q,\infty,[0,T]} \vee \|\E_{\edot} R^Y\|_{\frac{p}{2},\infty,[0,T]} \vee \|\E_{\edot} R^{\tY}\|_{\frac{p}{2},\infty,[0,T]} \leq 1.
\end{equation*}

We have from \eqref{eq: p,L^q-bound for f(Y)- f(tY)} that
\begin{equation*}
\|Y' - \tY'\|_{p,q,[0,T]} = \|f(Y) - f(\tY)\|_{p,q,[0,T]} \lesssim \|y_0 - \ty_0\|_{L^q} + \|Y - \tY\|_{p,q,[0,T]}.
\end{equation*}
Using the estimate in \eqref{eq: stability remainder of rough stochstic integrals} with $r = q$, combined with the bounds in Lemma~\ref{lemma: (f(Y), f'(Y)Y') is a controlled path} with $r = \infty$, and those in Lemma~\ref{lemma: f(Y)- f(tY)}, we find that
\begin{align*}
\|\E_{\edot} (R^Y - R^{\tY})\|_{\frac{p}{2},q,[0,T]} \lesssim \big(&\|y_0 - \ty_0\|_{L^q} + \|Y - \tY\|_{p,q,[0,T]} + \|Y' - \tY'\|_{p,q,[0,T]}\\
&+ \|\E_{\edot} (R^Y - R^{\tY})\|_{\frac{p}{2},q,[0,T]}\big) \big(T + \|\bX\|_{p,[0,T]}\big) + \|\bX - \tbX\|_{p,[0,T]}.
\end{align*}

It is straightforward to see that
\begin{equation*}
\bigg\|\int_0^\cdot b(Y_u) \dd u - \int_0^\cdot b(\tY_u) \dd u\bigg\|_{p,q,[0,T]} \lesssim T \big(\|y_0 - \ty_0\|_{L^q} + \|Y - \tY\|_{p,q,[0,T]}\big).
\end{equation*}
By the BDG inequality and Lemma~\ref{lemma: contraction for integral against quad. variation for Ito approach}, we have that
\begin{align*}
&\bigg\|\int_0^\cdot \sigma(Y_{u-}) \dd M_u - \int_0^\cdot \sigma(\tY_{u-}) \dd \tM_u\bigg\|_{p,q,[0,T]}\\
&\leq \bigg\|\int_0^\cdot \big(\sigma(Y_{u-}) - \sigma(\tY_{u-})\big) \dd M_u\bigg\|_{p,q,[0,T]} + \bigg\|\int_0^\cdot \sigma(\tY_{u-}) \dd (M - \tM)_u\bigg\|_{p,q,[0,T]}\\
&\lesssim \big(\|y_0 - \ty_0\|_{L^q} + \|Y - \tY\|_{p,q,[0,T]}\big) \|[M]\|_{\frac{p}{2},\frac{q}{2},\infty,[0,T]}^{\frac{1}{2}} + \big\|[M - \tM]\big\|_{\frac{p}{2},\frac{q}{2},[0,T]}^{\frac{1}{2}}\\
&\lesssim \big(\|y_0 - \ty_0\|_{L^q} + \|Y - \tY\|_{p,q,[0,T]}\big) \|M\|_{p,q,\infty,[0,T]} + \|M - \tM\|_{p,q,[0,T]}.
\end{align*}
Applying the estimate in \eqref{eq: stability for rough stochastic integrals} with $r = q$, combined with the bounds in Lemma~\ref{lemma: (f(Y), f'(Y)Y') is a controlled path} with $r = \infty$, and those in Lemma~\ref{lemma: f(Y)- f(tY)}, we obtain
\begin{align*}
&\bigg\|\int_0^\cdot f(Y_u) \dd \bX_u - \int_0^\cdot f(\tY_u) \dd \tbX_u\bigg\|_{p,q,[0,T]}\\
&\lesssim \big(\|y_0 - \ty_0\|_{L^q} + \|Y - \tY\|_{p,q,[0,T]} + \|Y' - \tY'\|_{p,q,[0,T]} + \|\E_{\edot} (R^Y - R^{\tY})\|_{\frac{p}{2},q,[0,T]}\big) \|\bX\|_{p,[0,T]}\\
&\quad + \|\bX - \tbX\|_{p,[s,t]}.
\end{align*}

Combining the previous estimates, we have that there exists a constant $C_3 > 0$, which depends only on $p, q, \|b\|_{C^1_b}, \|\sigma\|_{C^1_b}, \|f\|_{C^3_b}$ and $L$, such that
\begin{align*}
&\|Y - \tY\|_{p,q,[0,T]} + \|Y' - \tY'\|_{p,q,[0,T]} + \|\E_{\edot} (R^Y - R^{\tY})\|_{\frac{p}{2},q,[0,T]}\\
&\leq C_3 \Big(\big(\|y_0 - \ty_0\|_{L^q} + \|Y - \tY\|_{p,q,[0,T]} + \|Y' - \tY'\|_{p,q,[0,T]} + \|\E_{\edot} (R^Y - R^{\tY})\|_{\frac{p}{2},q,[0,T]}\big)\\
&\qquad \qquad \times \big(T + \|M\|_{p,q,\infty,[0,T]} + \|\bX\|_{p,[0,T]}\big)\\
&\qquad \quad + \|M - \tM\|_{p,q,[0,T]} + \|\bX - \tbX\|_{p,[0,T]}\Big).
\end{align*}
Letting $T_3 \in (0,T_2 \wedge \widetilde{T}_2]$ be sufficiently small such that
\begin{equation*}
T_3 + \|M\|_{p,q,\infty,[0,T_3]} + \|\bX\|_{p,[0,T_3]} \leq \frac{1}{2 C_3}
\end{equation*}
and rearranging, we deduce, for any $T \leq T_3$, the local estimate
\begin{equation}\label{eq: local estimate for RSDE}
\begin{split}
\|Y &- \tY\|_{p,q,[0,T]} + \|Y' - \tY'\|_{p,q,[0,T]} + \|\E_{\edot} (R^Y - R^{\tY})\|_{\frac{p}{2},q,[0,T]}\\
&\lesssim \big(\|y_0 - \ty_0\|_{L^q} + \|M - \tM\|_{p,q,[0,T]} + \|\bX - \tbX\|_{p,[0,T]}\big).
\end{split}
\end{equation}

\emph{Step 4. (Global solution)}
Let $w$ be the control given by
\begin{equation}\label{eq: control for global solution}
w(s,t) := (t - s) + \|M\|_{p,q,\infty,[s,t]}^p + \|X\|_{p,[s,t]}^p + \|\X\|_{\frac{p}{2},[s,t]}^{\frac{p}{2}}.
\end{equation}
The local solution found in Step~2 was obtained on the time interval $[0,T]$, for some $T$ chosen sufficiently small such that $w(0,T) \leq \epsilon$, where $\epsilon$ depends on the constants $C_1$ and $C_2$ above. However, since the estimates did not depend on the initial condition $y_0$, the same argument is valid starting from an arbitrary initial time $s$ and initial value $\xi_s$, giving in particular a unique solution on the interval $[s,t]$, for any time $t$ such that $w(s,t) \leq \epsilon$.

Now let $T > 0$ be arbitrary. By \cite[Lemma~1.5]{FrizZhang2018}, there exists a partition $\{0 = t_0 < t_1 < \cdots < t_n = T\}$ of the interval $[0,T]$ such that $w(t_{i-1},t_i-) < \epsilon$ for every $i = 1, 2, \ldots, n$. Let $\xi_0, \xi_1, \ldots, \xi_{n-1} \in L^q$, such that $\xi_i$ is $\cF_{t_i}$-measurable for each $i$.

For a given $i = 1, 2, \ldots, n$, we let
\begin{equation*}
\overline{M}_t := \begin{cases}
M_t &\text{ for } t \in [t_{i-1},t_i),\\
M_{t_i-} &\text{ for } t = t_i,
\end{cases}
\end{equation*}
noting that $\overline{M}$ is a c\`adl\`ag martingale on $[t_{i-1},t_i]$. We similarly let $\overline{X}_t = X_t$ and $\overline{\X}_{s,t} = \X_{s,t}$ whenever $(s,t) \in \Delta_{[t_{i-1},t_i]}$ with $t < t_i$, and let $\overline{X}_{t_i} = X_{t_i-}$ and $\overline{\X}_{s,t_i} = \X_{s,t_i-}$ for $s \in [t_{i-1},t_i)$, noting that $\overline{\bX} := (\overline{X},\overline{\X})$ is a c\`adl\`ag rough path on $[t_{i-1},t_i]$. Moreover, we have that
\begin{equation*}
(t_i - t_{i-1}) + \|\overline{M}\|_{p,q,\infty,[t_{i-1},t_i]}^p + \|\overline{X}\|_{p,[t_{i-1},t_i]}^p + \|\overline{\X}\|_{\frac{p}{2},[t_{i-1},t_i]}^{\frac{p}{2}} = w(t_{i-1},t_i-) < \epsilon.
\end{equation*}
Hence, we know that there exists a unique solution $\overline{Y}$ to the RSDE \eqref{eq: dynamical description of a solution to RSDE} with data $(\xi_{i-1},\overline{M},\overline{\bX})$ on the interval $[t_{i-1},t_i]$. In particular, by construction, we have that the sample paths of the solution $\overline{Y}$ are almost surely c\`adl\`ag, and that $\overline{Y}_{t_i} = \overline{Y}_{t_i-}$.

We then define a process $Y$ on $[t_{i-1},t_i]$ by reintroducing the jump at time $t_i$. That is, we set $Y_t := \overline{Y}_t$ for $t \in [t_{i-1},t_i)$, and
\begin{align*}
Y_{t_i} &:= \overline{Y}_{t_i} + \sigma(\overline{Y}_{t_i}) \Delta M_{t_i} + f(\overline{Y}_{t_i}) \Delta X_{t_i} + \D f(\overline{Y}_{t_i}) f(\overline{Y}_{t_i}) \Delta \X_{t_i}\\
&= Y_{t_i-} + \sigma(Y_{t_i-}) \Delta M_{t_i} + f(Y_{t_i-}) \Delta X_{t_i} + \D f(Y_{t_i-}) f(Y_{t_i-}) \Delta \X_{t_i}.
\end{align*}
We then obtain a global solution $Y$ by simply matching each initial value $\xi_i$ with the terminal value from the previous subinterval. That is, we let $\xi_0 = y_0$, and for each $i = 1, \ldots, n-1$ we let $\xi_i = Y_{t_i}$, where $Y$ is the process obtained above on the interval $[t_{i-1},t_i]$.

Finally, for each $i = 1, \ldots, n$, we note that
\begin{align*}
Y_{t_i} &= Y_{t_{i-1}} + \lim_{s \nearrow t_i} \bigg(\int_{t_{i-1}}^s b(Y_u) \dd u + \int_{t_{i-1}}^s \sigma(Y_{u-}) \dd M_u + \int_{t_{i-1}}^s f(Y_u) \dd \bX_u\bigg)\\
&\quad + \sigma(Y_{t_i-}) \Delta M_{t_i} + f(Y_{t_i-}) \Delta X_{t_i} + \D f(Y_{t_i-}) f(Y_{t_i-}) \Delta \X_{t_i}\\
&= Y_{t_{i-1}} + \int_{t_{i-1}}^{t_i} b(Y_u) \dd u + \int_{t_{i-1}}^{t_i} \sigma(Y_{u-}) \dd M_u + \int_{t_{i-1}}^{t_i} f(Y_u) \dd \bX_u,
\end{align*}
where we used the canonical jump structure, both of the It\^o integral (see, e.g., \cite[Lemma~12.1.10]{CohenElliott2015}) and of the rough stochastic integral, as established in Lemma~\ref{lemma: jump structure of rough stochastic integrals}.

We thus see that, on each interval $[t_{i-1},t_i]$, the process $Y$ is the solution of the RSDE with data $(Y_{t_{i-1}},M,\bX)$, and hence that $Y$ is indeed the solution of the RSDE with data $(y_0,M,\bX)$ on the entire interval $[0,T]$. We note that $Y$ has almost surely c\`adl\`ag sample paths\footnote{We could alternatively have defined a local solution on the half-open interval $[t_{i-1},t_i)$ and then included the jump at time $t_i$ to obtain a solution on $[t_{i-1},t_i]$, but it would then not be clear that the sample paths of the solution have left-limits almost surely at time $t_i$. This is why we defined a local solution $[t_{i-1},t_i]$ before including the jump, since this immediately guarantees sample path regularity on the entire interval up to and including time $t_i$.}, and that uniqueness of the local solutions implies that the global solution $Y$ is also unique.

\emph{Step 5. (Global estimate)}
Let $Y$ and $\tY$ be the global solutions---as obtained in the previous step---corresponding to the data $(y_0,M,\bX)$ and $(\ty_0,\tM,\tbX)$ respectively. Let $\tw$ be the control defined as in \eqref{eq: control for global solution} but with $M$ and $\bX$ replaced by $\tM$ and $\tbX$. We have from Step~3 above that there exists an $\epsilon > 0$, which depends only on $p, q, \|b\|_{C^1_b}, \|\sigma\|_{C^1_b}, \|f\|_{C^3_b}$ and $L$, such that the local estimate in \eqref{eq: local estimate for RSDE} holds on an interval $[s,t]$ whenever $w(s,t) \vee \tw(s,t) \leq \epsilon$.

By \cite[Lemma~1.5]{FrizZhang2018}, there exists a partition $\{0 = t_0 < t_1 < \cdots < t_n = T\}$ of the interval $[0,T]$ such that $w(t_{i-1},t_i-) \vee \tw(t_{i-1},t_i-) < \epsilon$ for every $i = 1, 2, \ldots, n$. Moreover, by the superadditivity of the controls $w, \tw$, this partition may be chosen such that the number $n$ of subintervals comprising the partition is bounded by a constant which depends only on $T, L$ and $\epsilon$.

For each $i = 1, 2, \ldots, n$, we have the local estimate
\begin{equation}\label{eq: local stability equation}
\begin{split}
&\|Y - \tY\|_{p,q,[t_{i-1},t_i)} + \|Y' - \tY'\|_{p,q,[t_{i-1},t_i)} + \|\E_{\edot} (R^Y - R^{\tY})\|_{\frac{p}{2},q,[t_{i-1},t_i)}\\
&\leq C \big(\|Y_{t_{i-1}} - \tY_{t_{i-1}}\|_{L^q} + \|M - \tM\|_{p,q,[t_{i-1},t_i)} + \|\bX - \tbX\|_{p,[t_{i-1},t_i)}\big),
\end{split}
\end{equation}
where the constant $C$ depends only on $p, q, \|b\|_{C^1_b}, \|\sigma\|_{C^1_b}, \|f\|_{C^3_b}$ and $L$.

Let $s \in [t_{i-1},t_i)$. By the BDG inequality and Lemma~\ref{lemma: contraction for integral against quad. variation for Ito approach}, we have that
\begin{align*}
\bigg\|&\int_s^{t_i} \sigma(Y_{u-}) \dd M_u - \int_s^{t_i} \sigma(\tY_{u-}) \dd \tM_u\bigg\|_{L^q}\\
&\leq \bigg\|\int_s^{t_i} \big(\sigma(Y_{u-}) - \sigma(\tY_{u-})\big) \dd M_u\bigg\|_{L^q} + \bigg\|\int_s^{t_i} \sigma(\tY_{u-}) \dd (M - \tM)_u\bigg\|_{L^q}\\
&\lesssim \big(\|Y_s - \tY_s\|_{L^q} + \|Y - \tY\|_{p,q,[s,t_i)}\big) \|M\|_{p,q,\infty,[s,t_i]} + \|M - \tM\|_{p,q,[s,t_i]}\\
&\lesssim \|Y_{t_{i-1}} - \tY_{t_{i-1}}\|_{L^q} + \|Y - \tY\|_{p,q,[t_{i-1},t_i)} + \|M - \tM\|_{p,q,[t_{i-1},t_i]},
\end{align*}
where we have bounded $\|M\|_{p,q,\infty,[s,t_i]} \leq L \lesssim 1$. Letting $s \nearrow t_i$, we obtain
\begin{align*}
&\big\|\sigma(Y_{t_i-}) \Delta M_{t_i} - \sigma(\tY_{t_i-}) \Delta \tM_{t_i}\big\|_{L^q}\\
&\lesssim \|Y_{t_{i-1}} - \tY_{t_{i-1}}\|_{L^q} + \|Y - \tY\|_{p,q,[t_{i-1},t_i)} + \|M - \tM\|_{p,q,[t_{i-1},t_i]}.
\end{align*}
We can also bound
\begin{align*}
\big\|f(Y_{t_i-}) \Delta X_{t_i} - f(\tY_{t_i-}) \Delta \tX_{t_i}\big\|_{L^q} &\lesssim \|Y_{t_{i-1}} - \tY_{t_{i-1}}\|_{L^q} + \|Y - \tY\|_{p,q,[t_{i-1},t_i)} + \|X - \tX\|_{p,[t_{i-1},t_i]}
\end{align*}
and
\begin{align*}
&\big\|\D f(Y_{t_i-}) f(Y_{t_i-}) \Delta \X_{t_i} - \D f(\tY_{t_i-}) f(\tY_{t_i-}) \Delta \tbbX_{t_i}\big\|_{L^q}\\
&\lesssim \|Y_{t_{i-1}} - \tY_{t_{i-1}}\|_{L^q} + \|Y - \tY\|_{p,q,[t_{i-1},t_i)} + \|\X - \tbbX\|_{\frac{p}{2},[t_{i-1},t_i]},
\end{align*}
where we similarly bounded $\|X\|_{p,[t_{i-1},t_i]} \lesssim 1$ and $\|\X\|_{\frac{p}{2},[t_{i-1},t_i]} \lesssim 1$.

Combining these bounds with the local estimate in \eqref{eq: local stability equation}, we have that
\begin{align*}
\|&Y - \tY\|_{p,q,[t_{i-1},t_i]} \lesssim \lim_{s \nearrow t_i} \big(\|Y - \tY\|_{p,q,[t_{i-1},s]} + \|Y - \tY\|_{p,q,[s,t_i]}\big)\\
&= \|Y - \tY\|_{p,q,[t_{i-1},t_i)} + \|\Delta Y_{t_i} - \Delta \tY_{t_i}\|_{L^q}\\
&\leq \|Y - \tY\|_{p,q,[t_{i-1},t_i)} + \big\|\sigma(Y_{t_i-}) \Delta M_{t_i} - \sigma(\tY_{t_i-}) \Delta \tM_{t_i}\big\|_{L^q}\\
&\quad + \big\|f(Y_{t_i-}) \Delta X_{t_i} - f(\tY_{t_i-}) \Delta \tX_{t_i}\big\|_{L^q} + \big\|\D f(Y_{t_i-}) f(Y_{t_i-}) \Delta \X_{t_i} - \D f(\tY_{t_i-}) f(\tY_{t_i-}) \Delta \tbbX_{t_i}\big\|_{L^q}\\
&\lesssim \|Y_{t_{i-1}} - \tY_{t_{i-1}}\|_{L^q} + \|M - \tM\|_{p,q,[t_{i-1},t_i]} + \|\bX - \tbX\|_{p,[t_{i-1},t_i]}.
\end{align*}
Since $\|Y_{t_{i-1}} - \tY_{t_{i-1}}\|_{L^q} \leq \|Y_{t_{i-2}} - \tY_{t_{i-2}}\|_{L^q} + \|Y - \tY\|_{p,q,[t_{i-2},t_{i-1}]}$ for each $i$, we can combine these estimates to obtain
\begin{equation*}
\|Y - \tY\|_{p,q,[t_{i-1},t_i]} \lesssim \|y_0 - \ty_0\|_{L^q} + \|M - \tM\|_{p,q,[0,t_i]} + \|\bX - \tbX\|_{p,[0,t_i]},
\end{equation*}
where the implicit multiplicative constant depends on $i$. Recalling that the number $n$ of subintervals in our partition depends only on $p, q, \|b\|_{C^1_b}, \|\sigma\|_{C^1_b}, \|f\|_{C^3_b}, T$ and $L$, we then have that
\begin{align}
\|Y - \tY\|_{p,q,[0,T]} &\leq n^{\frac{p-1}{p}} \bigg(\sum_{i=1}^n \|Y - \tY\|_{p,q,[t_{i-1},t_i]}^p\bigg)^{\hspace{-2pt}\frac{1}{p}}\label{eq:split p-var over subintervals}\\
&\lesssim \|y_0 - \ty_0\|_{L^q} + \|M - \tM\|_{p,q,[0,T]} + \|\bX - \tbX\|_{p,[0,T]}.\nonumber
\end{align}
It then also follows from the estimate in \eqref{eq: p,L^q-bound for f(Y)- f(tY)} that
\begin{equation*}
\|Y' - \tY'\|_{p,q,[0,T]} \lesssim \|y_0 - \ty_0\|_{L^q} + \|M - \tM\|_{p,q,[0,T]} + \|\bX - \tbX\|_{p,[0,T]}.
\end{equation*}

The difference between the remainder terms $\|\E_{\edot} (R^Y - R^{\tY})\|_{\frac{p}{2},q,[t_{i-1},t_i]}$ may be bounded similarly, except that in this case, since the increments $(s,t) \mapsto \E_s R^Y_{s,t}$ are not additive, but rather determined by the relation $R^Y_{s,t} = R^Y_{s,u} + R^Y_{u,t} + \delta Y'_{s,u} \delta X_{u,t}$ for $s \leq u \leq t$, splitting into subintervals as in \eqref{eq:split p-var over subintervals} introduces additional terms involving $\delta Y'$ and $\delta X$. However, these additional terms are straightforward to bound, and estimates over each of the subintervals may then be combined to obtain an estimate for $\|\E_{\edot} (R^Y - R^{\tY})\|_{\frac{p}{2},q,[0,T]}$---we omit the full details here for brevity---and we thus derive the global estimate in \eqref{eq: Lipschitz continuity of solution map}.
\end{proof}

\begin{remark}
In Theorem~\ref{theorem: existence and estimates for solutions to RSDEs} above we considered driving martingales which satisfy $M \in V^p L^{q,\infty}$. However, one can actually define solutions to the RSDE in \eqref{eq: dynamical description of a solution to RSDE} for martingales $M$ which only satisfy this condition locally.

More precisely, recall the setting of Theorem~\ref{theorem: existence and estimates for solutions to RSDEs}, but suppose now that $M$ is a c\`adl\`ag martingale such that $M^{\tau_k} \in V^p L^{q,\infty}$ for every $k \in \N$, where $(\tau_k)_{k \in \N}$ is an increasing sequence of stopping times such that $\P(\tau_k = T) \to 1$ as $k \to \infty$. For each $k \in \N$ we can also define the stopped rough path $\bX^{\tau_k}$. Strictly speaking, $\bX^{\tau_k}$ is now a random rough path, but the results of the previous sections may be easily generalized to handle such random rough paths (provided that they are uniformly bounded, as in the case of stopped rough paths). We thus have that, for each $k \in \N$, there exists a unique solution $Y^{(\tau_k)}$ to the RSDE
\begin{equation*}
Y^{(\tau_k)}_t = y_0 + \int_0^{t \wedge \tau_k} b(Y^{(\tau_k)}_s) \dd s + \int_0^t \sigma(Y^{(\tau_k)}_s) \dd M^{\tau_k}_s + \int_0^t f(Y^{(\tau_k)}_s) \dd \bX^{\tau_k}_s.
\end{equation*}
Moreover, these solutions for different $k \in \N$ are consistent, in the sense that $Y^{(\tau_m)}_{\cdot \wedge \tau_n} = Y^{(\tau_n)}$ whenever $n \leq m$. To see this, it is necessary to verify the (nontrivial but nonetheless true) fact that $\int_0^{t \wedge \tau} Y_s \dd \bX_s = \int_0^t Y_{s \wedge \tau} \dd \bX^{\tau}_s$ for any rough path $\bX$ and stochastic controlled path $Y$.

It follows that the process $Y$ given by $Y_t := \lim_{k \to \infty} Y^{(\tau_k)}_t$ is well-defined, and is the natural solution to the RSDE in \eqref{eq: dynamical description of a solution to RSDE} with data $(y_0,M,\bX)$. Moreover, for each $k \in \N$, the stopped solution $Y_{\cdot \wedge \tau_k}$ satisfies the local Lipschitz estimate in \eqref{eq: Lipschitz continuity of solution map}.
\end{remark}

\begin{remark}
Recalling Example~\ref{example: ito integral in mixed moment spaces}, let $Z \in V^p L^{q,\infty}$ be a c\`adl\`ag martingale, and let $H$ be predictable and locally bounded, with localizing sequence $(\tau_k)_{k \in \N}$. It is then straightforward to see that the process $M$, defined by $M_t := \int_0^t H_u \dd Z_u$, is a c\`adl\`ag martingale with $M^{\tau_k} \in V^p L^{q,\infty}$ for every $k \in \N$. By the previous remark, we can then define a solution to the RSDE in \eqref{eq: dynamical description of a solution to RSDE}, along with local Lipschitz estimates, with $M = \int_0^\cdot H_u \dd Z_u$ as the driving martingale.
\end{remark}

\section{Rough stochastic calculus}\label{section: rough stochastic calculus}

In this section we introduce some of the wider calculus of stochastic controlled paths. On the one hand, the following results---inspired by Section~5 of the first arXiv version of \cite{FrizHocquetLe2021}---are a bit more restrictive in terms of the permissible class of stochastic noise. In particular, the It\^o formula in Theorem~\ref{theorem: Ito formula for controlled paths} is only valid for stochastic controlled paths in $V^p L^{q,r}$ with $p \in [2,3)$ and $q \in [4,\infty)$, whereas L\'evy-type jumps require $q \leq p$, and are thus excluded.

On the other hand, we will nonetheless be able to obtain novel results on the structure of stochastic controlled paths with jumps---in particular Lemma~\ref{lemma: quadratic variation martingale with controlled path} below---as well as a generalization of the It\^o formula of \cite[Theorem~2.12]{FrizZhang2018} to the rough stochastic setting.

\begin{lemma}\label{lemma: existence of quad. cov. between controlled paths}
Let $p \in [2,3)$, $q \in [4,\infty)$ and $q_1, q_2 \in (\frac{q}{2},\infty)$ such that $\frac{1}{q_1} + \frac{1}{q_2} = \frac{2}{q}$. Let $\bX = (X,\X) \in \sV^p$ be a rough path, and let $(Y,Y') \in \cV^{p,q_1}_X$ and $(Z,Z') \in \cV^{p,q_2}_X$ be stochastic controlled paths. Then there exists an adapted process $[Y,Z] \in V^{\frac{p}{2}} L^{\frac{q}{2}}$ such that, for each $t \in [0,T]$,
\begin{equation*}
[Y,Z]_t = \lim_{|\cP| \to 0} \sum_{[u,v] \in \cP} \delta Y_{u,v} \otimes \delta Z_{u,v} - 2 (Y'_u \otimes Z'_u) \Sym(\X_{u,v})
\end{equation*}
as a limit in $L^{\frac{q}{2}}$, along partitions $\cP$ of the interval $[0,t]$ as the mesh size tends to zero.
\end{lemma}

\begin{proof}
Letting $\Xi_{s,t} = \delta Y_{s,t} \otimes \delta Z_{s,t} - 2 (Y'_s \otimes Z'_s) \Sym(\X_{s,t})$, we find that
\begin{align*}
\delta \Xi_{s,u,t} &= \delta Y_{s,u} \otimes \delta Z_{u,t} - (Y'_s \otimes Z'_s) (\delta X_{s,u} \otimes \delta X_{u,t}) + \delta Y_{u,t} \otimes \delta Z_{s,u} - (Y'_s \otimes Z'_s) (\delta X_{u,t} \otimes \delta X_{s,u})\\
&\quad + 2 (Y'_u \otimes Z'_u - Y'_s \otimes Z'_s) \Sym(\X_{u,t}).
\end{align*}
Since
\begin{align*}
&\delta Y_{s,u} \otimes \delta Z_{u,t} - (Y'_s \otimes Z'_s) (\delta X_{s,u} \otimes \delta X_{u,t}) = \delta Y_{s,u} \otimes \delta Z_{u,t} - (Y'_s \delta X_{s,u}) \otimes (Z'_s \delta X_{u,t})\\
&= \delta Y_{s,u} \otimes (Z'_u \delta X_{u,t} + R^Z_{u,t}) - (\delta Y_{s,u} - R^Y_{s,u}) \otimes (Z'_s \delta X_{u,t})\\
&= \delta Y_{s,u} \otimes (\delta Z'_{s,u} \delta X_{u,t}) + \delta Y_{s,u} \otimes R^Z_{u,t} + R^Y_{s,u} \otimes Z'_s \delta X_{u,t},
\end{align*}
by H\"older's inequality, we obtain
\begin{equation}\label{eq: delta Y_su delta Z_ut bound}
\begin{split}
\big\|\delta Y_{s,u} \otimes \delta Z_{u,t} - (&Y'_s \otimes Z'_s) (\delta X_{s,u} \otimes \delta X_{u,t})\big\|_{L^{\frac{q}{2}}} \leq \|Y\|_{p,q_1,[s,u]} \|Z'\|_{p,q_2,[s,u]} \|X\|_{p,[u,t]}\\
&+ \|Y\|_{p,q_1,[s,u]} \|R^Z\|_{p,q_2,[u,t]} + \|R^Y\|_{p,q_1,[s,u]} \sup_{v \in [0,T]} \|Z'_v\|_{L^{q_2}} \|X\|_{p,[u,t]}
\end{split}
\end{equation}
and
\begin{equation}\label{eq: Es delta Y_su delta Z_ut bound}
\begin{split}
\big\|\E_s \big[\delta Y_{s,u} \otimes \delta Z_{u,t} &- (Y'_s \otimes Z'_s) (\delta X_{s,u} \otimes \delta X_{u,t})\big]\big\|_{L^{\frac{q}{2}}} \leq \|Y\|_{p,q_1,[s,u]} \|Z'\|_{p,q_2,[s,u]} \|X\|_{p,[u,t]}\\
&+ \|Y\|_{p,q_1,[s,u]} \|\E_{\edot} R^Z\|_{\frac{p}{2},q_2,[u,t]} + \|\E_{\edot} R^Y\|_{\frac{p}{2},q_1,[s,u]} \sup_{v \in [0,T]} \|Z'_v\|_{L^{q_2}} \|X\|_{p,[u,t]}.
\end{split}
\end{equation}
Analogous bounds also hold for the quantity $\delta Y_{u,t} \otimes \delta Z_{s,u} - (Y'_s \otimes Z'_s) (\delta X_{u,t} \otimes \delta X_{s,u})$, and it is moreover clear that
\begin{equation}\label{eq: Y'u Z'u - Y's Z's bound}
\begin{split}
\big\|&(Y'_u \otimes Z'_u - Y'_s \otimes Z'_s) \Sym(\X_{u,t})\big\|_{L^{\frac{q}{2}}} = \big\|(Y'_u \otimes \delta Z'_{s,u} + \delta Y'_{s,u} \otimes Z'_s) \Sym(\X_{u,t})\big\|_{L^{\frac{q}{2}}}\\
&\leq \Big(\sup_{v \in [0,T]} \|Y'_v\|_{L^{q_1}} \|Z'\|_{p,q_2,[s,u]} + \sup_{v \in [0,T]} \|Z'_v\|_{L^{q_2}} \|Y'\|_{p,q_1,[s,u]}\Big) \|\X\|_{\frac{p}{2},[u,t]}.
\end{split}
\end{equation}
The result then follows by applying Theorem~\ref{theorem: mild stochastic sewing lemma} with $r = \frac{q}{2} \in [2,\infty)$. In particular, it follows from the estimates (which we omit here for brevity) provided by the stochastic sewing lemma that $[Y,Z] \in V^{\frac{p}{2}} L^{\frac{q}{2}}$.
\end{proof}

For $p \in [2,3)$ and $q \in [4,\infty)$, given a rough path $\bX = (X,\X) \in \sV^p$ and a stochastic controlled path $(Y,Y') \in \cV^{p,q}_X$, we write $[Y] := [Y,Y]$, with $[Y,Y]$ defined as in Lemma~\ref{lemma: existence of quad. cov. between controlled paths}.

If $M \in V^p L^q$ is a martingale, so that $(M,0) \in \cV^{p,q}_X$, then we have that
\begin{equation*}
[M]_t = \lim_{|\cP| \to 0} \sum_{[u,v] \in \cP} \delta M_{u,v} \otimes \delta M_{u,v},
\end{equation*}
so that $[M]$ coincides with the quadratic variation of $M$.

On the other hand, if $(Y,Y') = (X,1)$, then we see that $[Y]$ coincides with the bracket of the rough path $\bX$, which is defined as the path $[\bX] \in V^{\frac{p}{2}}$ with $[\bX]_0 = 0$ and increments given by
\begin{equation*}
\delta [\bX]_{s,t} = \delta X_{s,t} \otimes \delta X_{s,t} - 2 \Sym(\X_{s,t})
\end{equation*}
for every $(s,t) \in \Delta_{[0,T]}$; see, e.g., \cite[Section~2.4]{FrizZhang2018}.

\begin{lemma}\label{lemma: bracket of stochastic controlled path}
Let $p \in [2,3)$ and $q \in [4,\infty)$. Let $\bX = (X,\X) \in \sV^p$ be a rough path, and let $(Z,Z') \in \cV^{p,q}_X$ be a stochastic controlled path such that its remainder satisfies $\|R^Z\|_{\frac{p}{2},q,[0,T]} < \infty$. Then
\begin{equation*}
[Z]_t = \int_0^t (Z'_s \otimes Z'_s) \dd [\bX]_s
\end{equation*}
for every $t \in [0,T]$, where the right-hand side is defined as an $L^{\frac{q}{2}}$-valued Young integral.
\end{lemma}

In particular, if $(Y,Y') \in \cV^{p,q}_X$ is a stochastic controlled path, and if $(Z,Z') = (\int_0^\cdot Y_u \dd \bX_u, Y)$, which is itself a stochastic controlled path by Lemma~\ref{lemma:  rough stochastic integrals as controlled path}, then we see from the bound in \eqref{eq: p,q,r,s- estimate for rough stochastic integral} (with $r = q$) that $\|R^Z\|_{\frac{p}{2},q,[0,T]} < \infty$, so that
\begin{equation*}
[Z]_t = \int_0^t (Y_s \otimes Y_s) \dd [\bX]_s
\end{equation*}
for every $t \in [0,T]$. The proof of Lemma~\ref{lemma: bracket of stochastic controlled path} is similar to that of Proposition~5.6 in the first arXiv version of \cite{FrizHocquetLe2021}, and is therefore omitted for brevity.

\begin{lemma}\label{lemma: quadratic variation martingale with controlled path}
Let $p \in [2,3)$, $q \in [4,\infty)$ and $q_1, q_2 \in (\frac{q}{2},\infty)$ such that $\frac{1}{q_1} + \frac{1}{q_2} = \frac{2}{q}$, and let $\bX = (X,\X) \in \sV^p$ be a c\`adl\`ag rough path. Let $M \in V^p L^{q_1}$ be a c\`adl\`ag martingale, and let $(Z,Z') \in \cV^{p,q_2}_X$ be a stochastic controlled path such that $\|R^Z\|_{\frac{p}{2},q_2,[0,T]} < \infty$. Then, for each $t \in [0,T]$,
\begin{equation*}
[M,Z]_t = \sum_{0 < s \leq t} \Delta M_s \otimes \Delta Z_s,
\end{equation*}
where the sum is taken over all (deterministic) times $s$ such that $\|\Delta M_s \otimes \Delta Z_s\|_{L^{\frac{q}{2}}} \neq 0$.
\end{lemma}

\begin{proof}
For each $(s,t) \in \Delta_{[0,T]}$, we have that $\delta M_{s,t} \otimes \delta Z_{s,t} = R_{s,t} + \Xi_{s,t} + \Delta M_t \otimes \Delta Z_t$, where
\begin{equation*}
R_{s,t} := \delta M_{s,t-} \otimes \delta Z_{s,t} \qquad \text{and} \qquad \Xi_{s,t} := \Delta M_t \otimes \delta Z_{s,t-}.
\end{equation*}
It is straightforward to see that $\E_s [\Xi_{s,t}] = 0$. By H\"older's inequality, we also have that
\begin{equation*}
\|\Xi_{s,t}\|_{L^{\frac{q}{2}}} \leq \|\Delta M_t\|_{L^{q_1}} \|\delta Z_{s,t-}\|_{L^{q_2}} \leq \delta(t)^{\frac{1}{p}} w(s,t)^{\frac{1}{p}},
\end{equation*}
where $\delta(t) := \|\Delta M_t\|_{L^{q_1}}^p$ and $w(s,t) := \|Z\|_{p,q_2,[s,t)}^p$ is a control.

Since $\sum_{t \in [0,T]} \delta(t) \leq \|M\|_{p,q_1,[0,T]}^p < \infty$, it then follows from Lemma~\ref{lemma: pure jump stochastic sewing lemma} that
\begin{equation*}
\lim_{|\cP| \to 0} \bigg\|\sum_{[s,t] \in \cP} \Xi_{s,t}\bigg\|_{L^{\frac{q}{2}}} = 0.
\end{equation*}

Since $\delta Z_{s,t} = Z'_s \delta X_{s,t} + R^Z_{s,t}$, we have that
\begin{equation*}
\big\|\E_s [R_{s,t}]\big\|_{L^{\frac{q}{2}}} = \big\|\E_s [\delta M_{s,t-} \otimes R^Z_{s,t}]\big\|_{L^{\frac{q}{2}}} \leq \|\delta M_{s,t-}\|_{L^{q_1}} \|R^Z_{s,t}\|_{L^{q_2}} \leq w_1(s,t-)^{\frac{1}{p}} w_2(s,t)^{\frac{2}{p}},
\end{equation*}
where $w_1(s,t) := \|M\|_{p,q_1,[s,t]}^p$ and $w_2(s,t) := \|R^Z\|_{\frac{p}{2},q_2,[s,t]}^{\frac{p}{2}}$ are controls. We also have that
\begin{equation*}
\|R_{s,t}\|_{L^{\frac{q}{2}}} \leq \|\delta M_{s,t-}\|_{L^{q_1}} \|\delta Z_{s,t}\|_{L^{q_2}} = \bw_1(s,t-)^{\frac{1}{p}} \bw_2(s,t)^{\frac{1}{p}},
\end{equation*}
where $\bw_1(s,t) := \|M\|_{p,q_1,[s,t]}^p$ and $\bw_2(s,t) := \|Z\|_{p,q_2,[s,t]}^p$ are right-continuous controls. It then follows from Lemma~\ref{lemma: convergence to 0 of two param processes} that $\|\sum_{[s,t] \in \cP} R_{s,t}\|_{L^{\frac{q}{2}}} \to 0$ in the sense of RRS convergence.

Considering the stochastic controlled paths $(M,0)$ and $(Z,Z')$, it follows from Lemma~\ref{lemma: existence of quad. cov. between controlled paths}, and the convergence established above, that
\begin{equation*}
\sum_{[s,t] \in \cP} \Delta M_t \otimes \Delta Z_t \, \longrightarrow \, [M,Z]_T
\end{equation*}
in $L^{\frac{q}{2}}$ in the sense of RRS convergence. Given that $\sum_{[s,t] \in \cP} \Delta M_t \otimes \Delta Z_t$ converges in the RRS sense, it is straightforward to see that the limit must also be equal to $\sum_{0 < t \leq T} \Delta M_t \otimes \Delta Z_t$.
\end{proof}

\begin{remark}
Many relevant cases of Lemma~\ref{lemma: quadratic variation martingale with controlled path} require quite strong integrability of the stochastic controlled path $(Z,Z')$. In particular, with Proposition~\ref{proposition: a criterium for martingales in mixed moment spaces} in mind, we recall that for general c\`adl\`ag martingales to live in $V^p L^{q_1}$ typically requires $q_1 \leq p$. Since we consider $p \in [2,3)$, this means needing $q_1 < 3$, which requires at least $q_2 > 6$.
\end{remark}

\begin{lemma}\label{lemma: integration of controlled paths against each other}
Let $p \in [2,3)$, $q \in [4,\infty)$ and $q_1, q_2 \in (\frac{q}{2},\infty)$ such that $\frac{1}{q_1} + \frac{1}{q_2} = \frac{2}{q}$. Let $\bX = (X,\X) \in \sV^p$ be a rough path, and let $(Y,Y') \in \cV^{p,q_1}_X$ and $(Z,Z') \in \cV^{p,q_2}_X$ be stochastic controlled paths. Then there exists an adapted process $\int_0^\cdot Y_u \otimes \d Z_u \in V^p L^{\frac{q}{2}}$, such that, for every $(s,t) \in \Delta_{[0,T]}$,
\begin{equation*}
\int_s^t Y_u \otimes \d Z_u = \lim_{|\cP| \to 0} \sum_{[u,v] \in \cP} Y_u \otimes \delta Z_{u,v} + (Y'_u \otimes Z'_u) \X_{u,v}
\end{equation*}
as a limit in $L^{\frac{q}{2}}$ along partitions $\cP$ of the interval $[s,t]$ with vanishing mesh size.
\end{lemma}

\begin{proof}
Letting $\Xi_{s,t} = Y_s \otimes \delta Z_{s,t} + (Y'_s \otimes Z'_s) \X_{s,t}$, we have that
\begin{equation*}
-\delta \Xi_{s,u,t} = \delta Y_{s,u} \otimes \delta Z_{u,t} - (Y'_s \otimes Z'_s) (\delta X_{s,u} \otimes \delta X_{u,t}) + (Y'_u \otimes Z'_u - Y'_s \otimes Z'_s) \X_{u,t}.
\end{equation*}
Recalling the bounds in \eqref{eq: delta Y_su delta Z_ut bound}, \eqref{eq: Es delta Y_su delta Z_ut bound} and \eqref{eq: Y'u Z'u - Y's Z's bound}, the result follows by applying Theorem~\ref{theorem: mild stochastic sewing lemma} with $r = \frac{q}{2}$.
\end{proof}

Of course, the stochastic sewing lemma also provides corresponding estimates for the difference $\int_s^t Y_u \otimes \d Z_u - Y_s \otimes \delta Z_{s,t} - (Y'_s \otimes Z'_s) \X_{s,t}$, which we omit here for brevity.

\begin{lemma}
Let $p \in [2,3)$, $q \in [4,\infty)$ and $q_1, q_2 \in (\frac{q}{2},\infty)$ such that $\frac{1}{q_1} + \frac{1}{q_2} = \frac{2}{q}$. Let $\bX = (X,\X) \in \sV^p$ be a rough path, and let $(Y,Y') \in \cV^{p,q_1}_X$ and $(Z,Z') \in \cV^{p,q_2}_X$ be stochastic controlled paths. Then we have the following integration by parts formula
\begin{equation*}
Y_t \otimes Z_t = Y_0 \otimes Z_0 + \int_0^t Y_u \otimes \d Z_u + \int_0^t \dd Y_u \otimes Z_u + [Y,Z]_t
\end{equation*}
for every $t \in [0,T]$, where the integrals are defined in the sense of Lemma~\ref{lemma: integration of controlled paths against each other}.
\end{lemma}

\begin{proof}
Simply note that
\begin{align*}
Y_v \otimes Z_v - Y_u \otimes Z_u &= Y_u \otimes \delta Z_{u,v} + (Y'_u \otimes Z'_u) \X_{u,v} + \delta Y_{u,v} \otimes Z_u + (Y'_u \otimes Z'_u) (\X_{u,v})^\top\\
&\quad + \delta Y_{u,v} \otimes \delta Z_{u,v} - 2 (Y'_u \otimes Z'_u) \Sym(\X_{u,v}),
\end{align*}
and take $\lim_{|\cP| \to 0} \sum_{[u,v] \in \cP}$ on each side, using the convergence established in Lemmas~\ref{lemma: existence of quad. cov. between controlled paths} and \ref{lemma: integration of controlled paths against each other}.
\end{proof}

\begin{theorem}\label{theorem: Ito formula for controlled paths}
Let $p \in [2,3)$, $q \in [4,\infty)$, $r \in [2q,\infty]$ and $f \in C^3_b$. Let $\bX = (X,\X) \in \sV^p$ and $(Y,Y') \in \cV^{p,q,r}_X$. Then, for each $t \in [0,T]$, we have that
\begin{equation*}
\begin{split}
f(Y_t) &= f(Y_0) + \int_0^t \D f(Y_s) \dd Y_s + \frac{1}{2} \int_0^t \D^2 f(Y_s) \dd [Y]_s\\
&\quad + \sum_{0 < s \leq t} \Big(\Delta f(Y)_s - \D f(Y_{s-}) \Delta Y_s - \frac{1}{2} \D^2 f(Y_{s-}) (\Delta Y_s \otimes \Delta Y_s)\Big),
\end{split}
\end{equation*}
where the first integral is defined in the sense of Lemma~\ref{lemma: integration of controlled paths against each other}, the second integral is defined as an $L^{\frac{q}{2}}$-valued Young integral, and where the sum is taken over all (deterministic) times $s$ such that $\|\Delta Y_s\|_{L^q} \neq 0$.
\end{theorem}

\begin{proof}
Let $\cP = \{0 = t_0 < t_1 < \cdots < t_n = t\}$ be a partition of the interval $[0,t]$. We note that
\begin{equation*}
f(Y_t) - f(Y_0) = S_1(\cP) + S_2(\cP) + S_3(\cP) + S_4(\cP),
\end{equation*}
where
\begin{align*}
S_1(\cP) &= \sum_{i=0}^{n-1} \D f(Y_{t_i}) \delta Y_{t_i,t_{i+1}} + \frac{1}{2} \D^2 f(Y_{t_i}) (\delta Y_{t_i,t_{i+1}} \otimes \delta Y_{t_i,t_{i+1}}),\\
S_2(\cP) &= \sum_{i=0}^{n-1} \Delta f(Y)_{t_{i+1}} - \D f(Y_{t_{i+1}-}) \Delta Y_{t_{i+1}} - \frac{1}{2} \D^2 f(Y_{t_{i+1}-}) (\Delta Y_{t_{i+1}} \otimes \Delta Y_{t_{i+1}}),\\
S_3(\cP) &= \sum_{i=0}^{n-1} f(Y_{t_{i+1}-}) - f(Y_{t_i}) - \D f(Y_{t_i}) \delta Y_{t_i,t_{i+1}-} - \frac{1}{2} \D^2 f(Y_{t_i}) (\delta Y_{t_i,t_{i+1}-} \otimes \delta Y_{t_i,t_{i+1}-})\\
S_4(\cP) &= \sum_{i=0}^{n-1} \big(\D f(Y_{t_{i+1}-}) - \D f(Y_{t_i}) - \D^2 f(Y_{t_i}) \delta Y_{t_i,t_{i+1}-}\big) \Delta Y_{t_{i+1}}\\
&\qquad \quad + \frac{1}{2} \big(\D^2 f(Y_{t_{i+1}-}) - \D^2 f(Y_{t_i})\big) (\Delta Y_{t_{i+1}} \otimes \Delta Y_{t_{i+1}})
\end{align*}
and we used that, since the contraction of a symmetric tensor with an antisymmetric tensor is zero,
\begin{align*}
&\frac{1}{2} \big(\D^2 f(Y_{t_i}) (\Delta Y_{t_{i+1}} \otimes \delta Y_{t_i,t_{i+1}-}) + \D^2 f(Y_{t_i}) (\delta Y_{t_i,t_{i+1}-} \otimes \Delta Y_{t_{i+1}})\big)\\
&= \D^2 f(Y_{t_i}) \Sym(\delta Y_{t_i,t_{i+1}-} \otimes \Delta Y_{t_{i+1}}) = \D^2 f(Y_{t_i}) (\delta Y_{t_i,t_{i+1}-} \otimes \Delta Y_{t_{i+1}})\\
&= (\D^2 f(Y_{t_i}) \delta Y_{t_i,t_{i+1}-}) \Delta Y_{t_{i+1}}.
\end{align*}

Letting $\Gamma_{u,v} := \delta [Y]_{u,v} - \delta Y_{u,v} \otimes \delta Y_{u,v} + 2 (Y'_u \otimes Y'_u) \Sym(\X_{u,v})$, we have that
\begin{align*}
S_1(\cP) &= \sum_{i=0}^{n-1} \big(\D f(Y_{t_i}) \delta Y_{t_i,t_{i+1}} + \D^2 f(Y_{t_i}) (Y'_{t_i} \otimes Y'_{t_i}) \X_{t_i,t_{i+1}}\big)\\
&\quad + \frac{1}{2} \sum_{i=0}^{n-1} \D^2 f(Y_{t_i}) \delta[Y]_{t_i,t_{i+1}} - \frac{1}{2} \sum_{i=0}^{n-1} \D^2 f(Y_{t_i}) \Gamma_{t_i,t_{i+1}}.
\end{align*}
Applying Lemma~\ref{lemma: convergence to 0 of two param processes} with $R_{s,t} = \D^2 f(Y_s) \Gamma_{s,t}$, we have that
\begin{equation*}
\bigg\|\sum_{i=0}^{n-1} \D^2 f(Y_{t_i}) \Gamma_{t_i,t_{i+1}}\bigg\|_{L^{\frac{q}{2}}} \, \longrightarrow \, 0
\end{equation*}
in the sense of RRS convergence, where the required bounds come from the stochastic sewing lemma (Theorem~\ref{theorem: mild stochastic sewing lemma}), as in the proof of Lemma~\ref{lemma: existence of quad. cov. between controlled paths}. We then have that
\begin{equation*}
S_1(\cP) \, \longrightarrow \, \int_0^t \D f(Y_s) \dd Y_s + \frac{1}{2} \int_0^t \D^2 f(Y_s) \dd [Y]_s
\end{equation*}
in $L^{\frac{q}{2}}$ in the RRS sense. In particular, since $(Y,Y') \in \cV^{p,q,r}_X$, we know from Lemma~\ref{lemma: (f(Y), f'(Y)Y') is a controlled path} that $(\D f(Y),\D^2 f(Y) Y') \in \cV^{p,q,\frac{r}{2}}_X$, so that the first integral above exists in the sense of Lemma~\ref{lemma: integration of controlled paths against each other}. Moreover, noting that $\D^2 f(Y) \in V^p L^{\frac{q}{2}}$, and that $[Y] \in V^{\frac{p}{2}} L^{\frac{q}{2}}$ by Lemma~\ref{lemma: existence of quad. cov. between controlled paths}, we have that the second integral exists as an $L^{\frac{q}{2}}$-valued Young integral.

By \cite[Lemma~1.5]{FrizZhang2018}, for any $\epsilon > 0$, there exists a partition $\cP^\epsilon$ of the interval $[0,t]$ such that, for any refinement $\cP \supseteq \cP^\epsilon$ and any $[s,t] \in \cP$, we have that $\|Y\|_{p,q,r,[s,t)} < \epsilon$.

Since $f \in C^3_b$, we have from Taylor's formula that
\begin{equation*}
\Big|f(Y_t) - f(Y_s) - \D f(Y_s) \delta Y_{s,t} - \frac{1}{2} \D^2 f(Y_s) (\delta Y_{s,t} \otimes \delta Y_{s,t})\Big| \leq \frac{1}{2} \|f\|_{C^3_b} |\delta Y_{s,t}|^3,
\end{equation*}
and hence, for any refinement $\cP \supseteq \cP^\epsilon$,
\begin{equation*}
\|S_3(\cP)\|_{L^1} \lesssim \sum_{i=0}^{n-1} \|\delta Y_{t_i,t_{i+1}-}\|_{L^3}^3 \leq \sum_{i=0}^{n-1} \|Y\|_{p,q,r,[t_i,t_{i+1})}^3 \leq \epsilon^{3-p} \|Y\|_{p,q,r,[0,t]}^p,
\end{equation*}
so that $S_3(\cP) \to 0$ in $L^1$ in the sense of RRS convergence.

Similarly, since
\begin{equation*}
\big|\D f(Y_t) - \D f(Y_s) - \D^2 f(Y_s) \delta Y_{s,t}\big| \leq \|f\|_{C^3_b} |\delta Y_{s,t}|^2,
\end{equation*}
for any refinement $\cP \supseteq \cP^\epsilon$, using H\"older's inequality, we have that
\begin{align*}
\|S_4(\cP)\|_{L^1} &\lesssim \sum_{i=0}^{n-1} \|\delta Y_{t_i,t_{i+1}-}\|_{L^4}^2 \|\Delta Y_{t_{i+1}}\|_{L^2} + \sum_{i=0}^{n-1} \|\delta Y_{t_i,t_{i+1}-}\|_{L^2} \|\Delta Y_{t_{i+1}}\|_{L^4}^2\\
&\leq \sum_{i=0}^{n-1} \|Y\|_{p,q,r,[t_i,t_{i+1})}^2 \|Y\|_{p,q,r,[t_i,t_{i+1}]} + \sum_{i=0}^{n-1} \|Y\|_{p,q,r,[t_i,t_{i+1})} \|Y\|_{p,q,r,[t_i,t_{i+1}]}^2\\
&\leq 2 \epsilon^{3-p} \|\delta Y\|^p_{p,q,r,[0,t]},
\end{align*}
so that $S_4(\cP) \to 0$ in $L^1$ in the RRS sense.

Let $J = \{s \in (0,T] : \|\Delta Y_s\|_{L^1} > 0\}$. Again by Taylor's formula, we have that
\begin{equation*}
\sum_{s \in J} \Big\|\Delta f(Y)_s - \D f(Y_{s-}) \Delta Y_s - \frac{1}{2} \D^2 f(Y_{s-}) (\Delta Y_s \otimes \Delta Y_s)\Big\|_{L^1} \lesssim \sum_{s \in J} \|\Delta Y_s\|_{L^3}^3 \leq \|Y\|_{p,q,r,[0,t]}^3 < \infty.
\end{equation*}
Thus, for any $\epsilon > 0$, there exists a finite subset $J^\epsilon \subseteq J$ such that
\begin{equation*}
\bigg\|\sum_{s \in J \setminus J^\epsilon} \Delta f(Y)_s - \D f(Y_{s-}) \Delta Y_s - \frac{1}{2} \D^2 f(Y_{s-}) (\Delta Y_s \otimes \Delta Y_s)\bigg\|_{L^1} < \epsilon.
\end{equation*}
Then, for any partition $\cP$ which is a refinement of $J^\epsilon$, we have that
\begin{align*}
&\bigg\|S_2(\cP) - \sum_{s \in J} \Big(\Delta f(Y)_s - \D f(Y_{s-}) \Delta Y_s - \frac{1}{2} \D^2 f(Y_{s-}) (\Delta Y_s \otimes \Delta Y_s)\Big)\bigg\|_{L^1}\\
&\leq \bigg\|\sum_{s \in J \setminus J^\epsilon} \Delta f(Y)_s - \D f(Y_{s-}) \Delta Y_s - \frac{1}{2} \D^2 f(Y_{s-}) (\Delta Y_s \otimes \Delta Y_s)\bigg\|_{L^1} < \epsilon,
\end{align*}
which implies the desired RRS convergence of $S_2(\cP)$ in $L^1$.
\end{proof}

\section{Proof of the stochastic sewing lemma}\label{section: proof of sewing lemma}

This section is dedicated to the proof of Theorem~\ref{theorem: mild stochastic sewing lemma}. We largely follow the proof strategies of the mild and generalized sewing lemmas presented in \cite{FrizZhang2018}. However, in the present stochastic setting, the classical argument of Young used therein is not sufficient, as we need to rely on martingale cancellation estimates. As first observed in \cite{Le2020}, Young's argument must therefore be replaced by the allocation argument of Yaskov \cite{Yaskov2018}.

The allocation argument relies crucially on the existence of a sequence of refinements of partitions $(\cP^h)_{h \in \N}$ such that the growth of the number of partition points is dominated by the decay of the corresponding increment bounds. In the continuous setting, if
\[ \big\| \E_s [\delta \Xi_{s,u,t}] \big\|_{L^r} \sim |t-s|^{1 + \epsilon_1} \qquad \text{and} \qquad \|\delta \Xi_{s,u,t}\|_{q,r,s} \sim |t-s|^{\frac{1}{2} + \epsilon_2} \]
for some $\epsilon_1, \epsilon_2 > 0$, then the dyadic partitions $\cP^h = \{\frac{iT}{2^h}\}_{i=0}^{2^h}$ provide a suitable sequence. However, in the present discontinuous setting, we assume estimates of the form
\[ \big\| \E_s [\delta \Xi_{s,u,t}] \big\|_{L^r} \sim w(s,t)^{1 + \epsilon_1} \qquad \text{and} \qquad \|\delta \Xi_{s,u,t}\|_{q,r,s} \sim \bw(s,t)^{\frac{1}{2} + \epsilon_2}, \]
for \emph{distinct} discontinuous controls $w, \bw$. Thus, a dyadic construction with respect to a single control is not sufficient.

In \cite{Le2023}, L\^e introduced the notion of $w$-midpoints as an analogue of dyadic partitions for a single control which is continuous from the inside. While in principle one could define partitions using the sum of controls $w + \bw$ and proceed with L\^e's argument, this would only yield bounds for $\| \E_s [\delta \cI_{s,t} - \Xi_{s,t}] \|_{L^r}$ and $\| \delta \cI_{s,t} - \Xi_{s,t} \|_{q,r,s}$ in terms of this sum of controls, and, as also noted in \cite[Remark~4.10]{LiangTang2025}, such bounds are insufficient to obtain the required estimates for rough stochastic integrals.

To overcome this difficulty, in Definition~\ref{definition: alternating midpoints} we introduce the notion of \emph{alternating midpoints}. This provides a dyadic-type refinement scheme which simultaneously respects (an arbitrary number of) distinct controls.

A second technical obstacle is that midpoint constructions are only available for controls which are continuous from the inside, whereas the controls in our assumptions need not satisfy any continuity properties. To overcome this, in Lemma~\ref{lemma: alternating midpoints allocation} we establish a novel variant of the allocation argument.

The key observation is that, given a dyadic-type sequence of partitions $\{d^h_i\}_{i=0}^{2^h}$, $h \in \N$, we can allocate any given collection of times $t_0 < t_1 < \cdots < t_n$ into dyadic intervals such that $\sum_{i=0}^{n-1} \Xi_{t_i,t_{i+1}} - \Xi_{t_0,t_n}$ can be bounded in terms of the modified controls $w(\cdot+,\cdot-)$ and $\bw(\cdot+,\cdot-)$, which are continuous from the inside by construction. Moreover, we can isolate the midpoints at each level $h$, such that each midpoint $d^h_i$ only appears once in our estimates, which is crucial to avoid repeatedly counting the jump at $d^h_i$ at every subsequent level.

With these two new technical ingredients, the remainder of the proof proceeds as a hybrid of arguments from \cite{FrizZhang2018} and \cite{Le2023}.

\begin{lemma}\label{lemma: application of conditional BDG}
Let $q \in [2,\infty)$ and $r \in [q,\infty]$, and let $(\mathcal{G}_k)_{0 \leq k \leq N} = (\cF_{s_k})_{0 \leq k \leq N}$ for some $N \in \N$ and an increasing sequence of times $(s_k)_{0 \leq k \leq N}$. Suppose that $(z_k)_{1 \leq k \leq N}$ is a $(\mathcal{G}_k)$-adapted sequence of integrable random variables such that
\begin{equation*}
\E[z_k \,|\, \mathcal{G}_{k-1}] = 0
\end{equation*}
for every $k = 1, \ldots, N$. Then, for any $s \leq s_0$, we have that
\begin{equation*}
\bigg\|\sum_{k=1}^N z_k\bigg\|_{q,r,s} \leq C \bigg(\sum_{k=1}^N \|z_k\|^2_{q,r,s_{k-1}}\bigg)^{\hspace{-2pt}\frac{1}{2}},
\end{equation*}
where the constant $C$ depends only on $q$.
\end{lemma}

\begin{proof}
Since the process $(\sum_{k=1}^n z_k)_{0 \leq n \leq N}$ is a martingale with respect to the filtration $(\mathcal{G}_n)_{0 \leq n \leq N}$, and $\cF_s \subseteq \mathcal{G}_0$, it follows from the conditional BDG inequality (see Proposition~A.1 in the first arXiv version of \cite{FrizHocquetLe2021}) that
\begin{equation*}
\E\bigg[\bigg|\sum_{k=1}^N z_k\bigg|^q \,\bigg|\, \cF_s\bigg] \leq \E\bigg[\bigg(\sup_{n \leq N} \bigg|\sum_{k=1}^n z_k\bigg| \bigg)^{\hspace{-2pt}q} \,\bigg|\, \cF_s\bigg] \lesssim \E\bigg[\bigg(\sum_{k=1}^N |z_k|^2\bigg)^{\hspace{-2pt}\frac{q}{2}} \,\bigg|\, \cF_s\bigg].
\end{equation*}
We then have that
\begin{align*}
&\bigg\| \E \bigg[\bigg|\sum_{k=1}^N z_k\bigg|^q \,\bigg|\, \cF_s\bigg]^{\frac{1}{q}} \bigg\|_{L^r} \lesssim \bigg\|\E\bigg[\bigg(\sum_{k=1}^N |z_k|^2\bigg)^{\hspace{-2pt}\frac{q}{2}} \,\bigg|\, \cF_s\bigg]^{\frac{1}{q}}\bigg\|_{L^r} = \bigg\| \sum_{k=1}^N |z_k|^2 \bigg\|_{\frac{q}{2},\frac{r}{2},s}^{\frac{1}{2}}\\
&\leq \bigg( \sum_{k=1}^N \big\| |z_k|^2 \big\|_{\frac{q}{2},\frac{r}{2},s} \bigg)^{\hspace{-2pt}\frac{1}{2}} = \bigg( \sum_{k=1}^N \|z_k\|_{q,r,s}^2 \bigg)^{\hspace{-2pt}\frac{1}{2}} \leq \bigg( \sum_{k=1}^N \|z_k\|_{q,r,s_{k-1}}^2 \bigg)^{\hspace{-2pt}\frac{1}{2}},
\end{align*}
where in the last inequality we used part (ii) of Proposition~\ref{prop: Lqrs is Banach space}.
\end{proof}

\begin{remark}\label{remark: the conditional BDG including sup}
It is clear from the proof of Lemma~\ref{lemma: application of conditional BDG} that the following stronger statement also holds:
\begin{equation*}
\bigg\|\sup_{n \leq N} \bigg|\sum_{k=1}^n z_k\bigg| \bigg\|_{q,r,s} \leq C \bigg(\sum_{k=1}^N \|z_k\|_{q,r,s_{k-1}}^2\bigg)^{\hspace{-2pt}\frac{1}{2}}.
\end{equation*}
\end{remark}

\begin{lemma}\label{lemma: continuity from the inside of tw}
Let $w \colon \Delta_{[0,T]} \to [0,\infty)$ be a control. Then the function $\tw$, defined by $\tw(s,t) := w(s+,t-) := \lim_{v \nearrow t} \lim_{u \searrow s} w(u,v)$ for $(s,t) \in \Delta_{[0,T]}$, is a control which is continuous from the inside.
\end{lemma}

\begin{proof}
We let $s < \tau_1 < \tau_2 < u < \tau_3 < \tau_4 < t$. Since $w$ is non-negative and superadditive, we have that
\begin{equation*}
w(\tau_1,\tau_2) + w(\tau_3,\tau_4) \leq w(\tau_1,\tau_2) + w(\tau_2,\tau_3) + w(\tau_3,\tau_4) \leq w(\tau_1,\tau_4).
\end{equation*}
Letting $\tau_1 \searrow s$, $\tau_2 \nearrow u$, $\tau_3 \searrow u$ and $\tau_4 \nearrow t$, we obtain $\tw(s,u) + \tw(u,t) \leq \tw(s,t)$, so that $\tw$ is itself superadditive, and hence a control.

Let us write $\bw(s,t) := w(s+,t)$ (which is clearly right-continuous in $s$). We then have that $\tw(s,t) = \bw(s,t-)$, which is clearly left-continuous in $t$. It remains to show that, for any fixed $t$, the map $s \mapsto \tw(s,t) = \bw(s,t-)$ is still right-continuous. To this end, we assume for a contradiction that
\begin{equation*}
\tw(s+,t) < \tw(s,t) - \epsilon
\end{equation*}
for some $s < t$ and $\epsilon > 0$. Then, for any $u, v$ such that $s < u < v < t$, we have that
\begin{equation*}
\bw(u,v) + \bw(v,t-) \leq \bw(u,t-) = \tw(u,t) \leq \tw(s+,t) < \tw(s,t) - \epsilon = \bw(s,t-) - \epsilon,
\end{equation*}
and letting $u \searrow s$, we obtain
\begin{equation*}
\bw(s,v) + \tw(v,t) \leq \bw(s,t-) - \epsilon.
\end{equation*}
By choosing $v$ sufficiently close to $t$, we can ensure that $|\bw(s,v) - \bw(s,t-)| < \epsilon$, and the above then implies that $\tw(v,t) < 0$, which gives the required contradiction.
\end{proof}

For convenience, we recall the following useful observation. For a proof, see, e.g., \cite[Exercise~1.9]{FrizVictoir2010}.

\begin{lemma}\label{lemma: product of controls}
Let $w_1$ and $w_2$ be controls, and let $\alpha, \beta > 0$ such that $\alpha + \beta \geq 1$. Then $w_1^\alpha w_2^\beta$ is also a control.
\end{lemma}

We recall from \cite{Le2023} the following notion of a $w$-midpoint.

\begin{definition}
Let $w$ be a control. We define the \emph{$w$-midpoint} of an arbitrary interval $[s,t]$ as the number
$$d^w(s,t) := \inf \Big\{u \in [s,t] : w(s,u) \geq \frac{1}{2} w(s,t)\Big\}.$$
\end{definition}

As shown in \cite[p.~10]{Le2023}, if a control $w$ is continuous from the inside, then the $w$-midpoint $d^w = d^w(s,t)$ of an interval $[s,t]$ satisfies
\begin{equation}\label{eq: property of w-midpoints}
w(s,d^w) \leq \frac{1}{2} w(s,t) \qquad \text{and} \qquad w(d^w,t) \leq \frac{1}{2} w(s,t).
\end{equation}

For our purposes, we also introduce the following notion of alternating midpoints.

\begin{definition}\label{definition: alternating midpoints}
Let $w_1, \ldots, w_N$ be controls, for some $N \geq 2$. We define the \emph{$(w_1, \ldots, w_N)$-alternating midpoints} of an arbitrary interval $[s,t]$, denoted by $d^h_i(s,t)$ for each $h \in \N_0$ and $i = 0, 1, \ldots, 2^h$, as follows. Let $d^0_0(s,t) = s$ and $d^0_1(s,t) = t$. For $h \geq 1$, if $i$ is even, then we let $d^h_i(s,t)= d^{h-1}_{\frac{i}{2}}(s,t)$, and if $i$ is odd, then we let
\begin{equation*}
d^h_{i}(s,t) := d^{w_j}\Big(d^{h-1}_{\frac{i-1}{2}}(s,t), d^{h-1}_{\frac{i+1}{2}}(s,t)\Big) \qquad \text{if} \quad h = j \mod N,
\end{equation*}
that is, the $w_j$-midpoint of the interval $[d^{h-1}_{(i-1)/2}(s,t), d^{h-1}_{(i+1)/2}(s,t)]$.
\end{definition}

As in \cite{Le2023}, one can use $w$-midpoints to recursively split intervals into subintervals in a way which is consistent with the control $w$. The idea behind alternating midpoints is that we can alternate which control is used to perform the splitting at each level $h$.

\begin{lemma}\label{lemma: alternating-midpoints may replace dyadics}
Let $w_1, \ldots, w_N$ be controls which are all continuous from the inside, and let $d^h_i = d^h_i(s,t)$ for $h \in \N_0$ and $i = 0, 1, \ldots, 2^h$ be the $(w_1, \ldots, w_N)$-alternating midpoints of an interval $[s,t]$. Then, for every $h \in \N_0$, $i = 0, 1, \ldots, 2^h - 1$, and every $j = 1, 2, \ldots, N$, we have that
\begin{equation*}
w_j(d^h_i, d^h_{i+1}) \leq 2^{-\lfloor \frac{h}{N} \rfloor} w_j(s,t).
\end{equation*}
\end{lemma}

\begin{proof}
Let us fix an arbitrary $j = 1, 2, \ldots, N$. We need to verify that $w_j(d^h_i, d^h_{i+1}) \leq 2^{-m} w_j(s,t)$ whenever $mN \leq h < (m+1)N$, for every $m \in \N$. We will show this by induction on $m$, noting that the result is trivial when $m = 0$.

Let us assume that $w_j(d^h_i, d^h_{i+1}) \leq 2^{-(m-1)} w_j(s,t)$ whenever $(m-1)N \leq h < mN$, for some $m \geq 1$, and then take some $h$ such that $mN \leq h < (m+1)N$. There exists a unique $\ell$ such that $h - N < \ell \leq h$ and $\ell = j \mod N$. Then, for any given $i$, there is a unique $a$ such that $[d^h_i, d^h_{i+1}] \subseteq [d^{\ell}_a, d^{\ell}_{a+1}]$. Since one of the points $d^{\ell}_a, d^{\ell}_{a+1}$ is the $w_j$-midpoint of some interval $[d^{\ell-1}_b, d^{\ell-1}_{b+1}]$, we have from \eqref{eq: property of w-midpoints} that
\begin{equation*}
w_j(d^h_i, d^h_{i+1}) \leq w_j(d^{\ell}_a, d^{\ell}_{a+1}) \leq \frac{1}{2} w_j(d^{\ell-1}_b, d^{\ell-1}_{b+1}) \leq \frac{1}{2} w_j(d^{h-N}_c, d^{h-N}_{c+1})
\end{equation*}
for some $c$. We have from the inductive hypothesis that $w_j(d^{h-N}_c, d^{h-N}_{c+1}) \leq 2^{-(m-1)} w_j(s,t)$, and, substituting this into the above, we obtain $w_j(d^h_i, d^h_{i+1}) \leq 2^{-m} w_j(s,t)$.
\end{proof}

\begin{lemma}\label{lemma: summable bounds for the allocation w1,...,wN}
Let $w_1, \ldots, w_N$ be controls which are all continuous from the inside, and let $d^h_i = d^h_i(s,t)$ for $h \in \N_0$ and $i = 0, 1, \ldots, 2^h$ be the $(w_1, \ldots, w_N)$-alternating midpoints of an interval $[s,t]$. Let $\theta > 1$. Then, for every $h \in \N_0$ and every $j = 1, 2, \ldots, N$, we have that
\begin{equation*}
\sum_{i=0}^{2^h - 1} w_j(d^h_i, d^h_{i+1})^\theta \leq 2^{-(\theta - 1) \lfloor \frac{h}{N} \rfloor} w_j(s,t)^\theta.
\end{equation*}
\end{lemma}

\begin{proof}
Let us fix an arbitrary $j = 1, 2, \ldots, N$. We need to verify that $\sum_{i=0}^{2^h - 1} w_j(d^h_i, d^h_{i+1})^\theta \leq 2^{-(\theta - 1) m} w_j(s,t)^\theta$ whenever $mN \leq h < (m+1)N$, for every $m \in \N$. We will show this by induction on $m$. When $m = 0$, the result follows from the observation that $w_j^\theta$ is superadditive.

Let us assume that $\sum_{i=0}^{2^h - 1} w_j(d^h_i, d^h_{i+1})^\theta \leq 2^{-(\theta - 1) (m-1)} w_j(s,t)^\theta$ whenever $(m-1)N \leq h < mN$, for some $m \geq 1$, and then take some $h$ such that $mN \leq h < (m+1)N$. There exists a unique $\ell$ such that $h - N < \ell \leq h$ and $\ell = j \mod N$. By the superadditivity of $w_j^\theta$, we have that
\begin{equation*}
\sum_{i=0}^{2^h - 1} w_j(d^h_i, d^h_{i+1})^\theta \leq \sum_{k=0}^{2^\ell - 1} w_j(d^\ell_k, d^\ell_{k+1})^\theta = \sum_{n=0}^{2^{\ell - 1} - 1} \Big(w_j(d^\ell_{2n}, d^\ell_{2n+1})^\theta + w_j(d^\ell_{2n+1}, d^\ell_{2(n+1)})^\theta\Big).
\end{equation*}
Since $d^\ell_{2n+1}$ is the $w_j$-midpoint of the interval $[d^{\ell - 1}_n, d^{\ell - 1}_{n+1}]$, we then have from \eqref{eq: property of w-midpoints} that
\begin{equation*}
\sum_{i=0}^{2^h - 1} w_j(d^h_i, d^h_{i+1})^\theta \leq \sum_{n=0}^{2^{\ell - 1} - 1} 2 \bigg(\frac{1}{2} w_j(d^{\ell - 1}_n, d^{\ell - 1}_{n+1})\bigg)^{\hspace{-2pt}\theta} \leq 2^{-(\theta-1)} \sum_{i=0}^{2^{h-N} - 1} w_j(d^{h-N}_i, d^{h-N}_{i+1})^\theta.
\end{equation*}
We have from the inductive hypothesis that
$$\sum_{i=0}^{2^{h-N} - 1} w_j(d^{h-N}_i, d^{h-N}_{i+1})^\theta \leq 2^{-(\theta - 1) (m-1)} w_j(s,t)^\theta,$$
and, substituting this into the above, we deduce the result.
\end{proof}

The following lemma is a variant of the $w$-allocation result of \cite[Lemma~3.6]{Le2023}.

\begin{lemma}\label{lemma: alternating midpoints allocation}
Let $(\Xi_{s,t})_{(s,t) \in \Delta_{[0,T]}}$ be a two-parameter process such that $\Xi_{s,s} = 0$ for every $s \in [0,T]$. Let $w_1, \ldots, w_N$ be controls which are all continuous from the inside, and let $w(s,t) = t - s$ for $(s,t) \in \Delta_{[0,T]}$, which defines a continuous control $w$. Let $d^h_i = d^h_i(s,t)$ for $h \in \N_0$ and $i = 0, 1, \ldots, 2^h$, be the $(w_1, \ldots, w_N, w)$-alternating midpoints of an interval $[s,t]$. Finally, let $\{t_0 < t_1 < \cdots < t_n\}$ be a finite set of points in $[s,t]$. Then there exists a positive integer $h_0$, and random variables $R^h_i$, $J^h_i$ for each $h \in \N_0$ and $i = 0, 1, \ldots, 2^h - 1$, such that the following hold.
\begin{enumerate}
\item[(i)] $R^h_i = J^h_i = 0$ for every $h \geq h_0$ and every $i$.
\item[(ii)] For each $h \in \N_0$ and $i = 0, 1, \ldots, 2^h - 1$ such that $d^h_i < d^h_{i+1}$, there exist four (not necessarily distinct) points $s_1^{h,i} \leq s_2^{h,i} \leq s_3^{h,i} \leq s_4^{h,i}$, which all lie in the open interval $(d^h_i,d^h_{i+1})$, such that
\begin{equation}\label{eq: R^h,2_i decomp in w-allocation general stoch sewing, w1,...,wN}
R^h_i = -\delta \Xi_{s_1^{h,i},s_2^{h,i},s_3^{h,i}} - \delta \Xi_{s_1^{h,i},s_3^{h,i},s_4^{h,i}}.
\end{equation}
\item[(iii)] For every $h \in \N$ and every $i$ such that $d^h_i < d^h_{i+1}$, there exists an additional point $s_0^{h,i} \in \{d^h_i,s_1^{h,i}\}$, such that
\begin{equation}\label{eq: R^h,1_i decomp in w-allocation general stoch sewing, w1,...,wN}
J^h_i = -\delta \Xi_{s_0^{h,i},s_1^{h,i},s_4^{h,i}},
\end{equation}
and whenever $i$ is even we have that $s^{h,i}_0 = s^{h,i}_1$ so that in particular $J^h_i = 0$. Moreover, we have that
\begin{equation}\label{eq: defn of R^0,1}
J^0_0 = -\delta \Xi_{t_0,s^{0,0}_1,t_n} - \delta \Xi_{s^{0,0}_1,s^{0,0}_4,t_n}.
\end{equation}
\item[(iv)] We have the equality
\begin{equation}\label{eq: decomposition I(cP) general stoch sewing, w1,...wN}
\sum_{i=0}^{n-1} \Xi_{t_i,t_{i+1}} - \Xi_{s,t} = \sum_{h=0}^\infty \sum_{i=0}^{2^h - 1} (R^h_i + J^h_i).
\end{equation}
\end{enumerate}
\end{lemma}

\begin{proof}
For any collection of times $\cQ = \{u_i\}_{i=0}^K$, we write
\[ I(\cQ) := \sum_{i=0}^{K-1} \Xi_{u_i,u_{i+1}} - \Xi_{u_0,u_K}, \]
with the convention that $I(\cQ) = 0$ whenever $\# \cQ \leq 1$. Writing $\cP = \{t_i\}_{i=0}^n$, we define, for every $h \in \N$ and $i = 0, 1, \ldots, 2^h - 1$ such that $d^h_i < d^h_{i+1}$,
\begin{equation*}
\cP^h_i = \begin{cases}
\cP \cap (d^h_i,d^h_{i+1}), &\text{if $i$ is even,}\\
\cP \cap [d^h_i,d^h_{i+1}), &\text{if $i$ is odd.}
\end{cases}
\end{equation*}

For any $h \in \N$ and $i = 0, \ldots, 2^h-1$ we then define a random variable $I^h_i$ by letting $I^h_i = 0$ if $\# \cP^h_i \leq 2$, and otherwise letting
\[ I^h_i := I(\cP^h_i) - I(\cP^{h+1}_{2i}) - I(\cP^{h+1}_{2i+1}). \]
We also let $I^0_0 := I(\cP) - I(\cP^1_0) - I(\cP^1_1)$.

If either $\cP^{h+1}_{2i}$ or $\cP^{h+1}_{2i+1}$ is empty, then we set $R^h_i = 0$ and \eqref{eq: R^h,2_i decomp in w-allocation general stoch sewing, w1,...,wN} is satisfied with $s^{h,i}_j = d^h_i$ for $j = 1, 2, 3, 4$. If $\cP^{h+1}_{2i}$ and $\cP^{h+1}_{2i+1}$ are not empty, we define $s_0^{h,i} \leq s_1^{h,i} \leq s_2^{h,i} \leq s_3^{h,i} \leq s_4^{h,i}$ by
\begin{align*}
s_1^{h,i} = \min \cP_{2i}^{h+1}, ~~ s_2^{h,i} = \max \cP_{2i}^{h+1}, ~~ s_3^{h,i} = \min \cP_{2i+1}^{h+1}, ~~ s_4^{h,i} = \max \cP_{2i+1}^{h+1},
\end{align*}
and additionally $s_0^{h,i} = \min \cP_i^h \cap [d_i^h,d_{2i+1}^{h+1})$.

We note that, if $i$ is even, then necessarily $d^h_i \notin \cP^h_i$, so we simply have that $s^{h,i}_0 = s^{h,i}_1$. It is only when $i$ is odd that we might have that $d^h_i \in \cP^h_i$, in which case we would have that $s^{h,i}_0 = d^h_i < s^{h,i}_1$.

It is straightforward to see that, for every $h \in \N$ and every $i$,
\begin{equation*}
I^h_i = -\delta \Xi_{s_0^{h,i},s_1^{h,i},s_4^{h,i}} - \delta \Xi_{s_1^{h,i},s_2^{h,i},s_3^{h,i}} - \delta \Xi_{s_1^{h,i},s_3^{h,i},s_4^{h,i}}.
\end{equation*}
In particular, we see that $I^h_i = R^h_i + J^h_i$, with $R^h_i$ and $J^h_i$ defined as in \eqref{eq: R^h,2_i decomp in w-allocation general stoch sewing, w1,...,wN} and \eqref{eq: R^h,1_i decomp in w-allocation general stoch sewing, w1,...,wN}.

In the case when $h = 0$, we may have that $t_0 = d^0_0$ or $t_n = d^0_1$ (or both), so we first separate these points using the expression in \eqref{eq: defn of R^0,1}. More precisely, we note that $I^0_0 = R^0_0 + J^0_0$.

Thus, by recursively applying the definition of $I^h_i$, we find that
\begin{equation*}
I(\cP) = I(\cP^1_0) + I(\cP^1_1) + I^0_0 = \sum_{i=0}^{2^k-1} I(\cP^k_i) + \sum_{h=0}^{k-1} \sum_{i=0}^{2^h-1} I^h_i
\end{equation*}
for each $k \in \N$. By the inclusion of the continuous control $w$ in the alternating midpoints, we ensure that\footnote{That the mesh size of this sequence of partitions tends to zero is implicitly used in the proof of \cite[Lemma~3.6]{Le2023}, even though it does not actually hold when $w$-allocation is performed with a general control $w$. However, a modification of the proof shows that the stochastic sewing lemma of \cite{Le2023} is nonetheless true.} $\max_i |\cP^h_i| \to 0$ as $h \to \infty$, which guarantees that $\# \cP^h_i \leq 1$ and hence $I(\cP^h_i) = 0$ for all sufficiently large $h$. Thus, taking the limit as $k \to \infty$, we obtain
\begin{align*}
I(\cP) = \sum_{h=0}^\infty \sum_{i=0}^{2^h-1} I^h_i = \sum_{h=0}^\infty \sum_{i=0}^{2^h-1} (R^h_i + J^h_i).
\end{align*}
\end{proof}

\begin{lemma}\label{lemma: application joint w-allocation}
Let $q \in [2,\infty)$, $r \in [q,\infty]$, let $w_1$ and $w_2$ be controls, and let $\theta > 1$. Let $(\Xi_{s,t})_{(s,t) \in \Delta_{[0,T]}}$ be a two-parameter process such that $\Xi_{s,t}$ is $\cF_t$-measurable for every $(s,t) \in \Delta_{[0,T]}$, and such that $\Xi_{s,s} = 0$ for every $s \in [0,T]$. Suppose that there exists a partition $\cP = \{0 = s_0 < s_1 < \cdots < s_n = T\}$ of $[0,T]$ such that, for any refinement $\cP' \supseteq \cP$, any interval $[u_1,u_3] \in \cP'$ and any $u_2 \in (u_1,u_3)$, we have that
\begin{equation}\label{eq: bound E_u_1 delta Xi u_1 u_2 u_3 Lqr}
\big\|\E_{u_1} [\delta \Xi_{u_1,u_2,u_3}]\big\|_{L^{r}} \leq w_1(u_1,u_3)^\theta
\end{equation}
and
\begin{equation}\label{eq: bound delta Xi u_1 u_2 u_3 q r u_1}
\|\delta \Xi_{u_1,u_2,u_3}\|_{q,r,u_1} \leq w_2(u_1,u_3)^{\frac{\theta}{2}}.
\end{equation}
Then there exists a constant $C$, which depends only on $q$ and $\theta$, such that we have the following estimates.
\begin{enumerate}
\item[(i)] For any interval $[s_j,s_{j+1}] \in \cP$, and any finite collection of times $\{t_0 < t_1 < \cdots < t_m\} \subset [s_j,s_{j+1}]$, we have that
\begin{equation}\label{eq: uniform Lqr - bound on E_s[I(cP)] general stoch sewing}
\bigg\|\E_{s_j} \bigg[ \sum_{i=0}^{m-1} \Xi_{t_i,t_{i+1}} - \Xi_{t_0,t_m} \bigg]\bigg\|_{L^{r}} \leq C w_1(s_j,s_{j+1})^\theta
\end{equation}
and
\begin{equation}\label{eq: uniform q,r,s - bound on I(cP) general stoch sewing}
\bigg\| \sum_{i=0}^{m-1} \Xi_{t_i,t_{i+1}} - \Xi_{t_0,t_m} \bigg\|_{q,r,s_j} \leq C \big(w_1(s_j,s_{j+1})^\theta + w_2(s_j,s_{j+1})^{\frac{\theta}{2}}\big).
\end{equation}
\item[(ii)] For each $j = 0, 1, \ldots, n - 1$, let $\cP^j$ be a partition of the interval $[s_j,s_{j+1}]$. Then
\begin{equation}\label{eq: Lqr bound for sum of E_s[I(cP)] for disjoint intervals general stoch sewing}
\bigg\| \E_0 \bigg[ \sum_{j=0}^{n-1} \bigg( \sum_{[u,v] \in \cP^j} \Xi_{u,v} - \Xi_{s_j,s_{j+1}} \bigg) \bigg] \bigg\|_{L^{r}} \leq C w_1(0,T)^\theta
\end{equation}
and
\begin{equation}\label{eq: bound for sum of I(cP) for disjoint intervals general stoch sewing}
\bigg\| \sum_{j=0}^{n-1} \bigg( \sum_{[u,v] \in \cP^j} \Xi_{u,v} - \Xi_{s_j,s_{j+1}} \bigg) \bigg\|_{q,r,0} \leq C \big(w_1(0,T)^\theta + w_2(0,T)^{\frac{\theta}{2}}\big).
\end{equation}
\end{enumerate}
\end{lemma}

\begin{proof}
\emph{Part~(i):}
Let $\tw_1(s,t) := w_1(s+,t-)$ and $\tw_2(s,t) := w_2(s+,t-)$ for $(s,t) \in \Delta_{[0,T]}$, which we know from Lemma~\ref{lemma: continuity from the inside of tw} defines controls $\tw_1, \tw_2$ which are continuous from the inside.

Applying Lemma~\ref{lemma: alternating midpoints allocation} on the interval $[s_j,s_{j+1}]$ with the controls $\tw_1, \tw_2$ and the set of points $\{t_0 < t_1 < \cdots < t_m\}$, we have from \eqref{eq: decomposition I(cP) general stoch sewing, w1,...wN}, that
\begin{equation}\label{eq: allocating points without last point}
\sum_{i=0}^{m-1} \Xi_{t_i,t_{i+1}} - \Xi_{t_0,t_m} = \sum_{h=0}^\infty \sum_{i=0}^{2^h-1} (R^h_i + J^h_i),
\end{equation}
where, $R^h_i$ and $J^h_i$ are defined as in \eqref{eq: R^h,2_i decomp in w-allocation general stoch sewing, w1,...,wN}, \eqref{eq: R^h,1_i decomp in w-allocation general stoch sewing, w1,...,wN} and \eqref{eq: defn of R^0,1}.

We recall from Lemma~\ref{lemma: alternating midpoints allocation} that $J^h_i = -\delta \Xi_{s_0^{h,i},s_1^{h,i},s_4^{h,i}}$ for $h \in \N$, and that if $i$ is even then $s_0^{h,i} = s_1^{h,i}$ so that $J^h_i = 0$. We also recall that the points $s_0^{h,i}, s_1^{h,i}, s_2^{h,i}, s_3^{h,i}, s_4^{h,i}$ lie in the interval $[d^h_i,d^h_{i+1})$. In particular, for any $h \in \N$, if $i$ is odd, then we have that $[d^h_i, d^h_{i+1}) \subset (d^{h-1}_{\frac{i-1}{2}}, d^{h-1}_{\frac{i+1}{2}})$ so that, using the bounds in \eqref{eq: bound E_u_1 delta Xi u_1 u_2 u_3 Lqr} and \eqref{eq: bound delta Xi u_1 u_2 u_3 q r u_1}, it follows that\footnote{A technical point for the careful reader: strictly speaking it might be that $d^{h-1}_{\frac{i-1}{2}} = d^h_i$, so that we don't actually have $d^{h-1}_{\frac{i-1}{2}} < s^{h,i}_0$. However, since we included the control $w(s,t) = t - s$ in Lemma~\ref{lemma: alternating midpoints allocation}, we will certainly have that $[s^{h,i}_0, s^{h,i}_4] \subset (d^{h-3}_k, d^{h-3}_{k+1})$ for some $k$, which is enough for our purposes. Alternatively, we could instead use the controls $\tw_1(s,t) + \epsilon (t - s)$ and $\tw_2(s,t) + \epsilon (t - s)$, which would ensure that $d^{h-1}_{\frac{i-1}{2}} < d^h_i$, and then take $\epsilon \to 0$ at the end of the proof.}
\begin{equation}\label{eq: bounds on q,r,s and E L^r for R^h,1}
\big\| \E_{d^h_i} [J^h_i] \big\|_{L^r} \leq \tw_1 \big( d^{h-1}_{\frac{i-1}{2}}, d^{h-1}_{\frac{i+1}{2}} \big)^\theta \qquad \text{and} \qquad \|J^h_i\|_{q,r,d^h_i} \leq \tw_2 \big (d^{h-1}_{\frac{i-1}{2}}, d^{h-1}_{\frac{i+1}{2}} \big)^{\frac{\theta}{2}}.
\end{equation}
It is also immediate that $\|\E_{s_j} [J^0_0]\|_{L^r} \leq w_1(s_j,s_{j+1})^\theta$ and $\|J^0_0\|_{q,r,s_j} \leq w_2(s_j,s_{j+1})^{\frac{\theta}{2}}$.

Applying Lemma~\ref{lemma: application of conditional BDG} with the filtration $(\cG_i)_{0 \leq i \leq 2^h}$, where $\cG_i := \cF_{d^h_i}$, and the random variables $(z_i)_{1 \leq i \leq 2^h}$, where $z_{i+1} := J^h_i - \E_{d^h_i}[J^h_i]$, we deduce, for each $h \in \N$, that
\begin{equation*}
\bigg\| \sum_{i=0}^{2^h-1} J^h_i \bigg\|_{q,r,s_j} \lesssim \sum_{i=0}^{2^h-1} \big\| \E_{d^h_i} [J^h_i] \big\|_{L^r} + \bigg( \sum_{i=0}^{2^h-1} \|J^h_i\|_{q,r,d^h_i}^2 \bigg)^{\hspace{-2pt}\frac{1}{2}}.
\end{equation*}
Then, using \eqref{eq: bounds on q,r,s and E L^r for R^h,1} and applying Lemma~\ref{lemma: summable bounds for the allocation w1,...,wN} with the three alternating controls $\tw_1, \tw_2$ and $w$ (with $w$ as defined in Lemma~\ref{lemma: alternating midpoints allocation}), we then have that
\begin{align*}
\bigg\| \sum_{h=0}^\infty \sum_{i=0}^{2^h-1} J^h_i \bigg\|_{q,r,s_j} 
&\lesssim \sum_{h=0}^\infty 2^{-(\theta-1) \lfloor \frac{h-1}{3} \rfloor} w_1(s_j,s_{j+1})^\theta + \sum_{h=0}^\infty 2^{-\frac{1}{2} (\theta-1) \lfloor \frac{h-1}{3} \rfloor} w_2(s_j,s_{j+1})^{\frac{\theta}{2}}\\
&\lesssim w_1(s_j,s_{j+1})^\theta + w_2(s_j,s_{j+1})^{\frac{\theta}{2}}.
\end{align*}

Recalling that the points $s_1^{h,i}, s_2^{h,i}, s_3^{h,i}, s_4^{h,i}$ lie in the interval $(d^h_i,d^h_{i+1})$, we see that
\begin{equation*}
\big\| \E_{d^h_i} [R^h_i] \big\|_{L^r} \leq \tw_1(d^h_i,d^h_{i+1})^\theta \qquad \text{and} \qquad \|R^h_i\|_{q,r,d^h_i} \leq \tw_2(d^h_i,d^h_{i+1})^{\frac{\theta}{2}},
\end{equation*}
and we can then similarly bound
\begin{align*}
\bigg\| \sum_{h=0}^\infty \sum_{i=0}^{2^h-1} R^h_i \bigg\|_{q,r,s_j} 
&\lesssim \sum_{h=0}^\infty 2^{-(\theta-1) \lfloor \frac{h}{3} \rfloor} w_1(s_j,s_{j+1})^\theta + \sum_{h=0}^\infty 2^{-\frac{1}{2} (\theta-1) \lfloor \frac{h}{3} \rfloor} w_2(s_j,s_{j+1})^{\frac{\theta}{2}}\\
&\lesssim w_1(s_j,s_{j+1})^\theta + w_2(s_j,s_{j+1})^{\frac{\theta}{2}}.
\end{align*}
Recalling \eqref{eq: allocating points without last point}, we thus deduce the estimate in \eqref{eq: uniform q,r,s - bound on I(cP) general stoch sewing}. We similarly have that
\begin{align*}
&\bigg\| \E_{s_j} \bigg[ \sum_{i=0}^{m-1} \Xi_{t_i,t_{i+1}} - \Xi_{t_0,t_m} \bigg] \bigg\|_{L^r} \leq \sum_{h=0}^\infty \sum_{i=0}^{2^h-1} \big\| \E_{d^h_i} [R^h_i] \big\|_{L^r} + \sum_{h=0}^\infty \sum_{i=0}^{2^h-1} \big\| \E_{d^h_i} [J^h_i] \big\|_{L^r}\\
&\lesssim \sum_{h=0}^\infty 2^{-(\theta-1) \lfloor \frac{h}{3} \rfloor} w_1(s_j,s_{j+1})^\theta + \sum_{h=0}^\infty 2^{-(\theta-1) \lfloor \frac{h-1}{3} \rfloor} w_1(s_j,s_{j+1})^\theta \lesssim w_1(s_j,s_{j+1})^\theta,
\end{align*}
which gives the estimate in \eqref{eq: uniform Lqr - bound on E_s[I(cP)] general stoch sewing}.

\smallskip

\emph{Part~(ii):}
As before let us introduce the shorthand $I(\cP^j) := \sum_{[u,v] \in \cP^j} \Xi_{u,v} - \Xi_{s_j,s_{j+1}}$ for each $j = 0, 1, \ldots, n - 1$. Applying Lemma~\ref{lemma: application of conditional BDG} with the filtration $(\cG_j)_{0 \leq j \leq n}$, where $\cG_j := \cF_{s_j}$, and the random variables $(z_j)_{1 \leq j \leq n}$, where $z_{j+1} := I(\cP^j) - \E_{s_j}[I(\cP^j)]$, we obtain
\begin{align*}
&\bigg\| \sum_{j=0}^{n-1} I(\cP^j) \bigg\|_{q,r,0} \lesssim \sum_{j=0}^{n-1} \big\|\E_{s_j} [I(\cP^j)]\big\|_{L^{r}} + \bigg(\sum_{j=0}^{n-1} \|I(\cP^j)\|_{q,r,s_j}^2\bigg)^{\hspace{-2pt}\frac{1}{2}}\\
&= \sum_{j=0}^{n-1} \bigg\|\E_{s_j} \bigg[\sum_{[u,v] \in \cP^j} \Xi_{u,v} - \Xi_{s_j,s_{j+1}}\bigg]\bigg\|_{L^{r}} + \bigg(\sum_{j=0}^{n-1} \bigg\|\sum_{[u,v] \in \cP^j} \Xi_{u,v} - \Xi_{s_j,s_{j+1}}\bigg\|_{q,r,s_j}^2\bigg)^{\hspace{-2pt}\frac{1}{2}}.
\end{align*}
Applying the bounds in \eqref{eq: uniform Lqr - bound on E_s[I(cP)] general stoch sewing} and \eqref{eq: uniform q,r,s - bound on I(cP) general stoch sewing} for each $j$, we then have that
\begin{align*}
\bigg\|\sum_{j=0}^{n-1} I(\cP^j)\bigg\|_{q,r,0} &\lesssim \sum_{j=0}^{n-1} w_1(s_j,s_{j+1})^\theta + \bigg(\sum_{j=0}^{n-1} \big(w_1(s_j,s_{j+1})^{2\theta} + w_2(s_j,s_{j+1})^\theta\big)\bigg)^{\frac{1}{2}}\\
&\lesssim w_1(0,T)^\theta + w_2(0,T)^{\frac{\theta}{2}},
\end{align*}
which gives the estimate in \eqref{eq: bound for sum of I(cP) for disjoint intervals general stoch sewing}. Inspecting the bounds above, we also deduce the estimate in \eqref{eq: Lqr bound for sum of E_s[I(cP)] for disjoint intervals general stoch sewing}.
\end{proof}

\begin{proof}[Proof of Theorem~\ref{theorem: mild stochastic sewing lemma}]
We first note that, by the estimate in \eqref{eq: stoch sewing qrs bound}, we have that $\|\Xi_{s,s}\|_{q,r,s} = \|\delta \Xi_{s,s,s}\|_{q,r,s} = 0$, so that $\Xi_{s,s} = 0$ for every $s \in [0,T]$.
We fix a constant $\delta > 0$ such that $\delta < \min_{1 \leq i \leq N} (\alpha_{1,i} \wedge \alpha_{2,i})$, $\delta < \min_{1 \leq j \leq M} (\beta_{1,j} \wedge \beta_{2,j})$, and
\[ 2\delta < \min_{1 \leq i \leq N} (\alpha_{1,i} + \alpha_{2,i} - 1) \qquad \text{and} \qquad 2\delta < \min_{1 \leq j \leq M} \Big( \beta_{1,j} + \beta_{2,j} - \frac{1}{2} \Big). \]

By \cite[Lemma~1.5]{FrizZhang2018}, for any $\epsilon > 0$, there exists a partition $\{s_k\}_{k=0}^K$ of the interval $[0,T]$ such that
\begin{equation}\label{eq: all mild controls are summably bounded for full generality sewing}
\begin{split}
w_{1,i}(s_k+,s_{k+1}-) &\leq \epsilon^{\frac{1}{\delta}} w_{1,i}(0,T-),\\
\bw_{1,j}(s_k+,s_{k+1}-) &\leq \epsilon^{\frac{1}{\delta}} \bw_{1,j}(0,T-)
\end{split}
\end{equation}
for every $k = 0, 1, \ldots, K-1$, and every $i = 1, \ldots, N$ and $j = 1, \ldots, M$.
Since $w(s+,t) \to 0$ as $t \searrow s$ for any control $w$, for each $k = 0, 1, \ldots, K-1$, there exists a time $v_k \in (s_k,s_{k+1})$ such that, for every $i$ and $j$,
\begin{equation}\label{eq: all controls are summably bounded for full generality sewing}
\begin{split}
w_{2,i}(s_k+,v_k) &\leq \epsilon^{\frac{1}{\delta}} w_{2,i}(0+,T),\\
\bw_{2,j}(s_k+,v_k) &\leq \epsilon^{\frac{1}{\delta}} \bw_{2,j}(0+,T).
\end{split}
\end{equation}
We then define a partition $\cP^\epsilon$ of the interval $[0,T]$ by $\cP^\epsilon = \{t_n\}_{n=0}^{2K} := \{s_k\}_{k=0}^K \cup \{v_k\}_{k=0}^{K-1}$.

We now let $\theta > 1$ such that
\[ \theta < \min_{1 \leq i \leq N} (\alpha_{1,i} + \alpha_{2,i}) - 2\delta \qquad \text{and} \qquad \frac{\theta}{2} < \min_{1 \leq j \leq M} (\beta_{1,j} + \beta_{2,j}) - 2\delta, \]
and let
\begin{align*}
\rho_1(s,t) &= \sum_{i=1}^N w_{1,i}(0,T-)^{\frac{\delta}{\theta}} w_{2,i}(0+,T)^{\frac{\delta}{\theta}} w_{1,i}(s,t-)^{\frac{\alpha_{1,i} - \delta}{\theta}}w_{2,i}(s+,t)^{\frac{\alpha_{2,i} - \delta}{\theta}},\\
\rho_2(s,t) &= \sum_{j=1}^M \bw_{1,j}(0,T-)^{\frac{2\delta}{\theta}} \bw_{2,j}(0+,T)^{\frac{2\delta}{\theta}} \bw_{1,j}(s,t-)^{\frac{2(\beta_{1,j} - \delta)}{\theta}} \bw_{2,j}(s+,t)^{\frac{2(\beta_{2,j} - \delta)}{\theta}}
\end{align*}
for $(s,t) \in \Delta_{[0,T]}$, which, by Lemma~\ref{lemma: product of controls}, defines controls $\rho_1, \rho_2$.

Let us take an arbitrary refinement $\cP' \supseteq \cP^{\epsilon}$, any $[u_1,u_3] \in \cP'$ and any $u_2 \in (u_1,u_3)$. If $[u_1,u_3] \subset (s_k, s_{k+1}]$ for some $k = 0, 1, \ldots, K-1$, then, by the bounds in \eqref{eq: stoch sewing E_s Lqr bound} and \eqref{eq: stoch sewing qrs bound}, and using \eqref{eq: all mild controls are summably bounded for full generality sewing}, we have that
\begin{equation}\label{eq: new sewing bound 1}
\begin{split}
&\big\| \E_{u_1} [\delta \Xi_{u_1,u_2,u_3}] \big\|_{L^r} \leq \sum_{i=1}^N w_{1,i}(u_1,u_2)^{\alpha_{1,i}} w_{2,i}(u_2,u_3)^{\alpha_{2,i}}\\
&\leq \epsilon \sum_{i=1}^N w_{1,i}(0,T-)^\delta w_{2,i}(0+,T)^\delta w_{1,i}(u_1,u_2)^{\alpha_{1,i}-\delta} w_{2,i}(u_2,u_3)^{\alpha_{2,i}-\delta}\\
&\leq \epsilon \rho_1(u_1,u_3)^\theta,
\end{split}
\end{equation}
where in the last inequality we used the fact that the $\ell^p$ norm is non-increasing in $p$. Similarly,
\begin{equation}\label{eq: new sewing bound 2}
\begin{split}
&\|\delta \Xi_{u_1,u_2,u_3}\|_{q,r,u_1} \leq \sum_{j=1}^M \bw_{1,j}(u_1,u_2)^{\beta_{1,j}} \bw_{2,j}(u_2,u_3)^{\beta_{2,j}}\\
&\leq \epsilon \sum_{j=1}^M \bw_{1,j}(0,T-)^\delta \bw_{2,j}(0+,T)^\delta \bw_{1,j}(u_1,u_2)^{\beta_{1,j}-\delta} \bw_{2,j}(u_2,u_3)^{\beta_{2,j}-\delta}\\
&\leq C \epsilon \rho_2(u_1,u_3)^{\frac{\theta}{2}},
\end{split}
\end{equation}
for some constant $C$ which depends only on $M$ and $\theta$.

If instead $u_1 = s_k$ for some $k$, then we must have that $u_3 \leq v_k$, and using \eqref{eq: all controls are summably bounded for full generality sewing} then gives
\[ w_{2,i}(u_2,u_3)^{\alpha_{2,i}} \leq w_{2,i}(s_k+,v_k)^\delta w_{2,i}(u_2,u_3)^{\alpha_{2,i} - \delta} \leq \epsilon w_{2,i}(0+,T)^\delta w_{2,i}(u_2,u_3)^{\alpha_{2,i} - \delta}, \]
and similarly for $\bw_{2,j}$. Since we also have that $w_{1,i}(u_1,u_2) \leq w_{1,i}(0,T-)$ and similarly for $\bw_{1,j}$, we see that the inequalities in \eqref{eq: new sewing bound 1} and \eqref{eq: new sewing bound 2} also hold in this case. It follows that the hypotheses of Lemma~\ref{lemma: application joint w-allocation} are satisfied with the controls $C^{\frac{1}{\theta}} \epsilon^{\frac{1}{\theta}} \rho_1$ and $C^{\frac{2}{\theta}} \epsilon^{\frac{2}{\theta}} \rho_2$, and the partition $\cP^{\epsilon}$.

For any refinement $\cP' \supseteq \cP^\epsilon = \{t_n\}_{n=0}^{2K}$, we have that
\begin{equation*}
\sum_{[u,v] \in \cP'} \Xi_{u,v} - \sum_{[s,t] \in \cP^{\epsilon}} \Xi_{s,t} = \sum_{n=0}^{2K-1} \bigg( \sum_{[u,v] \in \cP'|_{[t_n,t_{n+1}]}} \Xi_{u,v} - \Xi_{t_n,t_{n+1}} \bigg),
\end{equation*}
and the estimate in \eqref{eq: bound for sum of I(cP) for disjoint intervals general stoch sewing} then implies that
\begin{equation}\label{eq: estimate for the estimates in full generality sewing}
\begin{split}
&\bigg\| \sum_{[u,v] \in \cP'} \Xi_{u,v} - \sum_{[s,t] \in \cP^{\epsilon}} \Xi_{s,t} \bigg\|_{q,r,0} \lesssim \epsilon \rho_1(0,T)^\theta + \epsilon \rho_2(0,T)^{\frac{\theta}{2}}\\
&\lesssim \epsilon \sum_{i=1}^N w_{1,i}(0,T-)^{\alpha_{1,i}} w_{2,i}(0+,T)^{\alpha_{2,i}} + \epsilon \sum_{j=1}^M \bw_{1,j}(0,T-)^{\beta_{1,j}} \bw_{2,j}(0+,T)^{\beta_{2,j}},
\end{split}
\end{equation}
and the estimate in \eqref{eq: Lqr bound for sum of E_s[I(cP)] for disjoint intervals general stoch sewing} similarly implies that
\begin{equation}\label{eq: estimate for Lqr E_0 in sewing}
\bigg\| \E_0 \bigg[ \sum_{[u,v] \in \cP'} \Xi_{u,v} - \sum_{[s,t] \in \cP^{\epsilon}} \Xi_{s,t} \bigg] \bigg\|_{L^r} \lesssim \epsilon \sum_{i=1}^N w_{1,i}(0,T-)^{\alpha_{1,i}} w_{2,i}(0+,T)^{\alpha_{2,i}}.
\end{equation}

Thus, for each $n \in \N$, there exists an $\epsilon_n > 0$ sufficiently small such that, writing $\cP_n := \cP^{\epsilon_n}$, we have for any refinement $\cP_n' \supseteq \cP_n$, that
\begin{equation}\label{eq: refinement of stoch Riemann sum less 1/n}
\bigg\|\sum_{[u,v] \in \cP_n'} \Xi_{u,v} - \sum_{[s,t] \in \cP_n} \Xi_{s,t}\bigg\|_{q,r,0} < \frac{1}{n}.
\end{equation}
By taking successive refinements, we may also assume that the sequence of partitions $(\cP_n)_{n \in \N}$ is nested. Then, for every $m \geq n$, we have that
\begin{equation}\label{eq: Cauchy bound on stoch Riemann sums}
\bigg\|\sum_{[u,v] \in \cP_m} \Xi_{u,v} - \sum_{[u,v] \in \cP_n} \Xi_{u,v}\bigg\|_{q,r,0} < \frac{1}{n},
\end{equation}
so that the sequence $(\sum_{[u,v] \in \cP_n} \Xi_{u,v} - \Xi_{0,T})_{n \in \N}$ is Cauchy, and hence convergent, in $L^{q,r}_0$. We write\footnote{This slightly awkward definition is necessary, since in general we have neither $\Xi_{0,T} \in L^{q,r}_0$, nor $\delta \cI_{0,T} \in L^{q,r}_0$.}
\begin{equation*}
\delta \cI_{0,T} := \Xi_{0,T} + \lim_{n \to \infty} \bigg(\sum_{[u,v] \in \cP_n} \Xi_{u,v} - \Xi_{0,T}\bigg).
\end{equation*}

Let $\epsilon > 0$, and choose an $n \in \N$ such that $\frac{2}{n} < \epsilon$. Taking the limit as $m \to \infty$ in \eqref{eq: Cauchy bound on stoch Riemann sums}, we see that $\|\delta \cI_{0,T} - \sum_{[s,t] \in \cP_n} \Xi_{s,t}\|_{q,r,0} \leq \frac{1}{n}$. It then follows from this and \eqref{eq: refinement of stoch Riemann sum less 1/n} that $\|\sum_{[u,v] \in \cP_n'} \Xi_{u,v} - \delta \cI_{0,T}\|_{q,r,0} < \epsilon$ for any refinement $\cP_n' \supseteq \cP_n$. This establishes RRS convergence of $\sum_{[u,v] \in \cP} \Xi_{u,v} - \Xi_{0,T}$ to $\delta \cI_{0,T} - \Xi_{0,T}$. We recall that limits in the RRS sense are unique, and in particular are independent of the choice of approximating partitions.

We presented this argument on the interval $[0,T]$, but it is equally valid on any given interval $[s,t] \subseteq [0,T]$, giving the existence of a corresponding random variable $\delta \cI_{s,t}$, along with the RRS convergence in \eqref{eq: convergence in stochastic sewing lemma}. It is clear from this convergence that $\delta \cI_{s,u} + \delta \cI_{u,t} = \delta \cI_{s,t}$ whenever $s \leq u \leq t$, so that $(\delta \cI_{s,t})_{(s,t) \in \Delta_{[0,T]}}$ are indeed the increments of a well-defined process $\cI = (\cI_t)_{t \in [0,T]}$, given by $\cI_t = \delta \cI_{0,t}$ for $t \in [0,T]$. Note, that the $L_s^{q,r}$-convergence implies that $\delta \cI_{s,t}$ is $\cF_t$-measurable for every $(s,t) \in \Delta_{[0,T]}$, so that the process $\cI$ is adapted.

The estimates in \eqref{eq: estimate for the estimates in full generality sewing} and \eqref{eq: estimate for Lqr E_0 in sewing} were derived on the interval $[0,T]$, but are equally valid on any given interval $[s,t] \subseteq [0,T]$. In particular, taking $\epsilon = 1$ and noting that in this case we may simply take $\cP^1 = \{s,t\}$, we have that
\begin{equation}\label{eq: bound on [s,t] for sewing estimate 1}
\bigg\|\sum_{[u,v] \in \cP} \Xi_{u,v} - \Xi_{s,t}\bigg\|_{q,r,s} \lesssim  \sum_{i=1}^{N} w_{1,i}(s,t-)^{\alpha_{1,i}} w_{2,i}(s+,t)^{\alpha_{2,i}} +  \sum_{j=1}^M\bw_{1,j}(s,t-)^{\beta_{1,j}} \bw_{2,j}(s+,t)^{\beta_{2,j}}
\end{equation}
and
\begin{equation}\label{eq: bound on [s,t] for sewing estimate 2}
\bigg\|\E_s \bigg[\sum_{[u,v] \in \cP} \Xi_{u,v} - \Xi_{s,t}\bigg]\bigg\|_{L^{r}} \lesssim \sum_{i=1}^{N} w_{1,i}(s,t-)^{\alpha_{1,i}} w_{2,i}(s+,t)^{\alpha_{2,i}}
\end{equation}
for every partition $\cP$ of the interval $[s,t]$.

Taking the RRS limit in \eqref{eq: bound on [s,t] for sewing estimate 1} gives the estimate in \eqref{eq: qrs bound full generality sewing}. Since, by part~(iii) of Proposition~\ref{prop: Lqrs is Banach space}, $\|\E_s [\sum_{[u,v] \in \cP} \Xi_{u,v} - \delta \cI_{s,t}]\|_{L^{r}} \leq \|\sum_{[u,v] \in \cP} \Xi_{u,v} - \delta \cI_{s,t}\|_{q,r,s} \to 0$, taking the RRS limit in \eqref{eq: bound on [s,t] for sewing estimate 2} gives the estimate in \eqref{eq: L^qr bound full generality sewing}.

Next we check that the process $\cI$ is unique. To this end, suppose that $(\cI_t)_{t \in [0,T]}$ and $(\widetilde{\cI}_t)_{t \in [0,T]}$ were two adapted processes, with $\cI_0 = \widetilde{\cI}_0 = 0$, which both satisfy the bounds in \eqref{eq: L^qr bound full generality sewing} and \eqref{eq: qrs bound full generality sewing}. It is then straightforward to see that
\begin{equation*}
\big\|\E_s [\delta \cI_{s,t} - \delta \widetilde{\cI}_{s,t}]\big\|_{L^{r}} \lesssim  \sum_{i=1}^{N} w_{1,i}(s,t-)^{\alpha_{1,i}} w_{2,i}(s+,t)^{\alpha_{2,i}}
\end{equation*}
and
\begin{equation*}
\|\delta \cI_{s,t} - \delta \widetilde{\cI}_{s,t}\|_{q,r,s} \lesssim \sum_{i=1}^{N} w_{1,i}(s,t-)^{\alpha_{1,i}} w_{2,i}(s+,t)^{\alpha_{2,i}} +  \sum_{j=1}^M\bw_{1,j}(s,t-)^{\beta_{1,j}} \bw_{2,j}(s+,t)^{\beta_{2,j}}.
\end{equation*}
We can then apply Lemma~\ref{lemma: convergence to 0 of two param processes} with $R_{s,t} = \delta \cI_{s,t} - \delta \widetilde{\cI}_{s,t}$, which implies that $\|\cI_t - \widetilde{\cI}_t\|_{L^q} = 0$ for each $t \in [0,T]$, and hence that $\cI = \widetilde{\cI}$ up to modification.

It remains to show that, under the additional continuity assumptions on the controls, the convergence in \eqref{eq: convergence in stochastic sewing lemma} also holds in the MRS sense, i.e., as the mesh size of the partition tends to zero.

Let $\epsilon > 0$. By the RRS convergence in \eqref{eq: convergence in stochastic sewing lemma}, we know that there exists a partition $\overline{\cP}^\epsilon$ of the interval $[0,T]$ such that, for any refinement $\overline{\cP} \supseteq \overline{\cP}^\epsilon$, we have that
\begin{equation*}
\bigg\| \sum_{[s,t] \in \overline{\cP}} \Xi_{s,t} - \delta \cI_{0,T} \bigg\|_{q,r,0} < \epsilon.
\end{equation*}
By \cite[Lemma~1.5]{FrizZhang2018}, combined with the additional assumptions on the controls, we may assume that the partition $\overline{\cP}^\epsilon$ is also chosen such that, for any $[s,t] \in \overline{\cP}^\epsilon$, and any $v \in (s,t)$,
\begin{equation*}
w_{1,i}(v,t) < \epsilon \qquad \text{or} \qquad w_{2,i}(s,v) < \epsilon
\end{equation*}
for every $i = 1, \ldots, N$, and
\begin{equation*}
\bw_{1,j}(v,t) < \epsilon \qquad \text{or} \qquad \bw_{2,j}(s,v) < \epsilon
\end{equation*}
for every $j = 1, \ldots, M$.

Let $\cP$ be any partition with $|\cP| < \min_{[s,t] \in \overline{\cP}^\epsilon} |t - s|$. Then, for any interval $[u,v] \in \cP$, there exists at most one point $\tau \in \overline{\cP}^\epsilon$ such that $\tau \in [u,v]$, and by the above we know that if such a point exists, then $w_{1,i}(u,\tau) < \epsilon$ or $w_{2,i}(\tau,v) < \epsilon$, and $\bw_{1,j}(u,\tau) < \epsilon$ or $\bw_{2,j}(\tau,v) < \epsilon$ for all $i, j$.

Let $\widetilde{\cP} = \cP \cup \overline{\cP}^\epsilon$ and, for each $\tau \in \overline{\cP}^\epsilon$, let $[u_\tau,v_\tau]$ be an interval in $\cP$ such that $\tau \in [u_\tau,v_\tau]$ (which is unique, unless $\tau \in \cP \cap \overline{\cP}^\epsilon$). We then have that
\begin{equation}\label{eq: bound to show MRS convergence}
\begin{split}
\bigg\|\sum_{[s,t] \in \cP} \Xi_{s,t} - \delta \cI_{0,T}\bigg\|_{q,r,0} &\leq \bigg\|\sum_{[s,t] \in \widetilde{\cP}} \Xi_{s,t} - \delta \cI_{0,T}\bigg\|_{q,r,0} + \bigg\|\sum_{[s,t] \in \widetilde{\cP}} \Xi_{s,t} - \sum_{[s,t] \in \cP} \Xi_{s,t}\bigg\|_{q,r,0}\\
&< \epsilon + \bigg\|\sum_{\tau \in \overline{\cP}^\epsilon} \delta \Xi_{u_\tau,\tau,v_\tau}\bigg\|_{q,r,0}.
\end{split}
\end{equation}

We fix a $\gamma > 0$ such that $\gamma < \min_{1 \leq i \leq N} \big(\alpha_{1,i} \wedge \alpha_{2,i})$, $\gamma < \min_{1 \leq j \leq M} 2(\beta_{1,j} \wedge \beta_{2,j})$,
\begin{equation*}
2\gamma < \min_{1 \leq i \leq N} (\alpha_{1,i} + \alpha_{2,i} - 1) \qquad \text{and} \qquad \gamma < \min_{1 \leq j \leq M} \Big( \beta_{1,j} + \beta_{2,j} - \frac{1}{2} \Big),
\end{equation*}
and define controls $\rho_3, \rho_4$ by
\begin{align*}
\rho_3(s,t) &:= \sum_{i=1}^N w_{1,i}(s,t)^{\alpha_{1,i}-\gamma} w_{2,i}(s,t)^{\alpha_{2,i} - \gamma},\\
\rho_4(s,t) &:= \sum_{j=1}^M \bw_{1,j}(s,t)^{2\beta_{1,j}-\gamma} \bw_{2,j}(s,t)^{2\beta_{2,j} - \gamma}.
\end{align*}

Writing $\overline{\cP}^\epsilon = \{0 = \tau_0 < \tau_1 < \cdots < \tau_m = T\}$, we note that $v_{\tau_{k-1}} \leq u_{\tau_k}$ for each $k = 1, \ldots, m$. We may therefore apply Lemma~\ref{lemma: application of conditional BDG} with the filtration $(\cG_k)_{0 \leq k < m}$, where $\cG_k = \cF_{v_{\tau_k}}$, and the random variables $(z_k)_{1 \leq k < m}$, where $z_k = \delta \Xi_{u_{\tau_k},\tau_k,v_{\tau_k}} - \E_{u_{\tau_k}} [\delta \Xi_{u_{\tau_k},\tau_k,v_{\tau_k}}]$, to obtain
\begin{align*}
&\bigg\|\sum_{\tau \in \overline{\cP}^\epsilon} \delta \Xi_{u_\tau,\tau,v_\tau}\bigg\|_{q,r,0} \lesssim \sum_{k=1}^{m-1} \big\|\E_{u_{\tau_k}} [\delta \Xi_{u_{\tau_k},\tau_k,v_{\tau_k}}]\big\|_{L^{r}} + \bigg(\sum_{k=1}^{m-1} \|\delta \Xi_{u_{\tau_k},\tau_k,v_{\tau_k}}\|_{q,r,u_{\tau_k}}^2\bigg)^{\hspace{-2pt}\frac{1}{2}}\\
&\lesssim \sum_{k=1}^{m-1} \sum_{i=1}^N w_{1,i}(u_{\tau_k},\tau_k)^{\alpha_{1,i}} w_{2,i}(\tau_k,v_{\tau_k})^{\alpha_{2,i}} + \bigg( \sum_{k=1}^{m-1} \sum_{j=1}^M \bw_{1,j}(u_{\tau_k},\tau_k)^{2\beta_{1,k}} \bw_{2,j}(\tau_k,v_{\tau_k})^{2\beta_{2,j}} \bigg)^{\hspace{-2pt}\frac{1}{2}}\\
&\lesssim \sum_{k=1}^{m-1} \rho_3(u_{\tau_k},v_{\tau_k}) \epsilon^\gamma + \bigg(\sum_{k=1}^{m-1} \rho_4(u_{\tau_k},v_{\tau_k})\bigg)^{\hspace{-2pt}\frac{1}{2}} \epsilon^{\frac{\gamma}{2}} \leq \rho_3(0,T) \epsilon^\gamma + \rho_4(0,T)^{\frac{1}{2}} \epsilon^{\frac{\gamma}{2}},
\end{align*}
where in the last line we used the fact that $w_{1,i}(u_{\tau_k},\tau_k) < \epsilon$ or $w_{2,i}(\tau_k,v_{\tau_k}) < \epsilon$, and $\bw_{1,j}(u_{\tau_k},\tau_k) < \epsilon$ or $\bw_{2,j}(\tau_k,v_{\tau_k}) < \epsilon$ for each $i$ and $j$. Combining this with the bound in \eqref{eq: bound to show MRS convergence}, we conclude that $\|\sum_{[u,v] \in \cP} \Xi_{u,v} - \delta \cI_{0,T}\|_{q,r,0} \to 0$ as $|\cP| \to 0$, i.e., in the sense of MRS convergence.
\end{proof}

\appendix

\section{Proof of Theorem~\ref{theorem: sample path regularity sewing}}\label{section: proof of uniform estimate}

\begin{lemma}\label{lemma: alternating midpoints bounds for sum of non continuous controls}
Let $(w_{1,i}, w_{2,i})_{i = 1, \ldots, N}$ be controls, let $w$ be the control given by $w(s,t) = t - s$ for $(s,t) \in \Delta_{[0,T]}$. For some fixed $s < t$, let $\cP^h = \{d^h_n\}_{n=0}^{2^h}$, $h_0 \in \N$, denote the $(w_{1,i}(\cdot+, \cdot-))_{i=1,\ldots,N}$-$(w_{2,i}(\cdot+, \cdot-))_{i=1,\ldots,N}$-$w$ alternating midpoints of the interval $[s,t]$, in the sense of Definition~\ref{definition: alternating midpoints}. Let $\alpha_{1,i}, \alpha_{2,i} > 0$ such that $\alpha_{1,i} + \alpha_{2,i} > 1$ for each $i = 1, \ldots, N$, and let $\eta > 0$ such that $\eta < \min_{1 \leq i \leq N} (\alpha_{1,i} \wedge \alpha_{2,i})$ and $2\eta < \min_{1 \leq i \leq N} (\alpha_{1,i} + \alpha_{2,i} - 1)$. Then
\begin{equation}\label{eq: summable bounds for the allocation w1,...,wN for sums of controls}
\sum_{n=0}^{2^h-1} \sum_{i=1}^N w_{1,i}(d^h_n, d^h_{n+1}-)^{\alpha_{1,i}} w_{2,i}(d^h_n+, d^h_{n+1})^{\alpha_{2,i}} \leq 2^{-\eta (\lfloor \frac{h}{H} \rfloor - 1)} \sum_{i=1}^N w_{1,i}(s,t-)^{\alpha_{1,i}} w_{2,i}(s+,t)^{\alpha_{2,i}}
\end{equation}
for every $h \in \N$, where $H = 2N + 1$.
\end{lemma}

\begin{proof}
We first note that if $h < 2H$ then the result holds immediately by superadditivity, so we may suppose that $h \geq 2H$.

By the inclusion of the control $w$ in the alternating midpoint construction, we note that, for every odd $n$, we have that $[d^h_n, d^h_{n+1}) \subset (d^{h-H}_{\ell(n)}, d^{h-H}_{\ell(n)+1})$ for some $\ell(n)$. Similarly, for every even $n$, we have that $(d^h_n, d^h_{n+1}] \subset (d^{h-H}_{\ell(n)}, d^{h-H}_{\ell(n)+1})$ for some $\ell(n)$. Using the bound in Lemma~\ref{lemma: alternating-midpoints may replace dyadics}, we then have, for every $k = 0, 1, \ldots, 2^{h-1}-1$, that
\[ w_{1,i}(d^h_{2k+1}, d^h_{2k+2}-) \leq w_{1,i}(d^{h-H}_{\ell(2k+1)}+, d^{h-H}_{\ell(2k+1)+1}-) \leq 2^{-(\lfloor \frac{h}{H} \rfloor - 1)} w_{1,i}(s,t-), \]
and
\[ w_{2,i}(d^h_{2k}+, d^h_{2k+1}) \leq w_{2,i}(d^{h-H}_{\ell(2k)}+, d^{h-H}_{\ell(2k)+1}-) \leq 2^{-(\lfloor \frac{h}{H} \rfloor - 1)} w_{2,i}(s+,t). \]
We also have the trivial bounds
\[ w_{1,i}(d^h_{2k}, d^h_{2k+1}-) \leq w_{1,i}(s,t-) \qquad \text{and} \qquad w_{2,i}(d^h_{2k+1}+, d^h_{2k+2}) \leq w_{2,i}(s+,t). \]
Using these bounds, we then have that
\begin{align*}
&\sum_{n=0}^{2^h-1} \sum_{i=1}^N w_{1,i}(d^h_n, d^h_{n+1}-)^{\alpha_{1,i}} w_{2,i}(d^h_n+, d^h_{n+1})^{\alpha_{2,i}}\\
&= \sum_{k=0}^{2^{h-1}-1} \sum_{i=1}^N w_{1,i}(d^h_{2k}, d^h_{2k+1}-)^{\alpha_{1,i}} w_{2,i}(d^h_{2k}+, d^h_{2k+1})^{\alpha_{2,i}}\\
&\quad + \sum_{k=0}^{2^{h-1}-1} \sum_{i=1}^N w_{1,i}(d^h_{2k+1}, d^h_{2k+2}-)^{\alpha_{1,i}} w_{2,i}(d^h_{2k+1}+, d^h_{2k+2})^{\alpha_{2,i}}\\
&\leq 2^{-\eta (\lfloor \frac{h}{H} \rfloor - 1)} \bigg( \sum_{k=0}^{2^{h-1}-1} \sum_{i=1}^N w_{1,i}(s,t-)^\eta w_{2,i}(s+,t)^\eta w_{1,i}(d^h_{2k}, d^h_{2k+1}-)^{\alpha_{1,i} - \eta} w_{2,i}(d^h_{2k}+, d^h_{2k+1})^{\alpha_{2,i} - \eta}\\
&\qquad \quad + \sum_{k=0}^{2^{h-1}-1} \sum_{i=1}^N w_{1,i}(s,t-)^\eta w_{2,i}(s+,t)^\eta w_{1,i}(d^h_{2k+1}, d^h_{2k+2}-)^{\alpha_{1,i} - \eta} w_{2,i}(d^h_{2k+1}+, d^h_{2k+2})^{\alpha_{2,i} - \eta} \bigg)\\
&\leq 2^{-\eta (\lfloor \frac{h}{H} \rfloor - 1)} \sum_{i=1}^N w_{1,i}(s,t-)^{\alpha_{1,i}} w_{2,i}(s+,t)^{\alpha_{2,i}},
\end{align*}
which is precisely the desired bound.
\end{proof}

\begin{remark}
We note that one could also use the bound in Lemma~\ref{lemma: alternating midpoints bounds for sum of non continuous controls} to prove Theorem~\ref{theorem: mild stochastic sewing lemma}, without the need for the variant of the allocation procedure presented in Lemma~\ref{lemma: alternating midpoints allocation}. However, the proof strategy adopted in Section~\ref{section: proof of sewing lemma} is slightly more general, since Lemma~\ref{lemma: application joint w-allocation} is valid for general controls, not only those of the form $\sum_{i=1}^N w_{1,i}(s,u)^{\alpha_{1,i}} w_{2,i}(u,t)^{\alpha_{2,i}}$. In particular, this allows for further generalizations of the stochastic sewing lemma, as discussed in Remark~\ref{remark: stochastic sewing can be generalised}.
\end{remark}

\begin{lemma}\label{lemma: unifrom approximation in sewing on the partition}
Recall the setting of Theorem~\ref{theorem: mild stochastic sewing lemma}, and let $w$ be the control given by $w(s,t) = t - s$ for $(s,t) \in \Delta_{[0,T]}$. For some $s < t$, let $\cP^h = \{d^h_n\}_{n=0}^{2^h}$, $h \in \N_0$, be the $(w_{1,i}(\cdot+,\cdot-), w_{2,i}(\cdot+,\cdot-))_{i=1,\ldots,N}$-$(\bw_{1,j}(\cdot+,\cdot-), \bw_{2,j}(\cdot+,\cdot-))_{j=1,\ldots,M}$-$w$-alternating midpoints of the interval $[s,t]$, in the sense of Definition~\ref{definition: alternating midpoints}. Then there exists an $\eta > 0$ such that
\begin{equation}\label{eq: uniform bound sup_t cI- Xi^Ph more specific}
\begin{split}
&\Big\| \sup_{v \in \cP^h} \big|\delta \cI_{s,v} - \delta \Xi^{\cP^h}_{s,v}\big| \Big\|_{q,r,s}\\
&\leq C 2^{-\lfloor \frac{h}{H} \rfloor \eta} \bigg( \sum_{i=1}^N w_{1,i}(s,t-)^{\alpha_{1,i}} w_{2,i}(s+,t)^{\alpha_{2,i}} + \sum_{j=1}^M \bw_{1,j}(s,t-)^{\beta_{1,j}} \bw_{2,j}(s+,t)^{\beta_{2,j}} \bigg)
\end{split}
\end{equation}
for every $h \in \N_0$, where $H = 2N + 2M + 1$, and $\Xi^{\cP^h}$ is defined as in \eqref{eq: defn Xi^cP}, and where the constants $\eta$ and $C$ depend only on $q, N, M$ and the exponents $(\alpha_{1,i}, \alpha_{2,i})_{i=1,\ldots,N}$ and $(\beta_{1,j}, \beta_{2,j})_{j=1,\ldots,M}$.
\end{lemma}

\begin{proof}
We note that, for each $h \in \N_0$ and $v \in \cP^h$,
\begin{equation*}
\delta \cI_{s,v} - \delta \Xi^{\cP^h}_{s,v} = \sum_{n \, : \, d^h_{n+1} \leq v} Z_{d^h_n, d^h_{n+1}},
\end{equation*}
where $Z_{d^h_n, d^h_{n+1}} := \delta \cI_{d^h_n, d^h_{n+1}} - \Xi_{d^h_n, d^h_{n+1}}$.

It then follows from Lemma~\ref{lemma: application of conditional BDG} (and Remark~\ref{remark: the conditional BDG including sup}) applied with the random variables $Z_{d^h_n, d^h_{n+1}} - \E_{d^h_n} [Z_{d^h_n, d^h_{n+1}}]$ for $0 \leq n < 2^h$, and the bounds in \eqref{eq: L^qr bound full generality sewing} and \eqref{eq: qrs bound full generality sewing}, that
\begin{align*}
\Big\| \sup_{v \in \cP^h} \big|\delta \cI_{s,v} - \delta \Xi^{\cP^h}_{s,v}\big| \Big\|_{q,r,s} &\lesssim \sum_{n=0}^{2^h-1} \big\| \E_{d_n^h} [Z_{d^h_n, d^h_{n+1}}] \big\|_{L^r} + \bigg( \sum_{n=0}^{2^h-1} \|Z_{d^h_n, d^h_{n+1}}\|_{q,r,d_n^h}^2 \bigg)^{\hspace{-2pt}\frac{1}{2}}\\
&\lesssim \sum_{n=0}^{2^h-1} \sum_{i=1}^N w_{1,i}(d^h_n, d^h_{n+1}-)^{\alpha_{1,i}} w_{2,i}(d^h_n+, d^h_{n+1})^{\alpha_{2,i}}\\
&\quad + \bigg( \sum_{n=0}^{2^h-1} \sum_{j=1}^M \bw_{1,j}(d^h_n, d^h_{n+1}-)^{2\beta_{1,i}} \bw_{2,j}(d^h_n+, d^h_{n+1})^{2\beta_{2,j}} \bigg)^{\hspace{-2pt}\frac{1}{2}}.
\end{align*}
By Lemma~\ref{lemma: alternating midpoints bounds for sum of non continuous controls}, there exists an $\eta > 0$ sufficiently small that, applying the bound in \eqref{eq: summable bounds for the allocation w1,...,wN for sums of controls}, we obtain the estimate in \eqref{eq: uniform bound sup_t cI- Xi^Ph more specific}.
\end{proof}

\begin{lemma}\label{lemma: major lemma for cadlag property of integrals}
Recall the setting of Theorem~\ref{theorem: mild stochastic sewing lemma}, and assume that the process $\Xi$ also satisfies the bound in \eqref{eq: uniform q,r,s bound for uniform convergence}. Let $\cP^h = \{d^h_n\}_{n=0}^{2^h}$, $h \in \N_0$, be the $(w_{1,i}(\cdot+,\cdot-), w_{2,i}(\cdot+,\cdot-))_{i=1,\ldots,N}$-$(\bw_{1,j}(\cdot+,\cdot-), \bw_{2,j}(\cdot+,\cdot-))_{j=1,\ldots,M}$-$(\hw_{1,k}(\cdot+,\cdot-), \hw_{2,k}(\cdot+,\cdot-))_{k=1,\ldots,K}$-$w$-alternating midpoints of the interval $[0,T]$ (in the sense of Definition~\ref{definition: alternating midpoints}), where $w$ is the control given by $w(s,t) = t - s$ for $(s,t) \in \Delta_{[0,T]}$. Then there exists an $\eta > 0$ such that
\begin{equation}\label{eq: uniform bound sup_i sup_t delta I - Xi}
\begin{split}
\Big\| \sup_{0 \leq n < 2^h} \, &\sup_{t \in \cP^\ell|_{[d^h_n,d^h_{n+1}]}} |\delta \cI_{d^h_n,t} - \Xi_{d^h_n,t}| \Big\|_{q,r,0} \leq C 2^{-\eta h} \bigg( \sum_{i=1}^N w_{1,i}(0,T-)^{\alpha_{1,i}} w_{2,i}(0+,T)^{\alpha_{2,i}}\\
&\qquad + \sum_{j=1}^M \bw_{1,j}(0,T-)^{\beta_{1,j}} \bw_{2,j}(0+,T)^{\beta_{2,j}} + \sum_{k=1}^K \hw_{1,k}(0,T-)^{\gamma_{1,k}} \hw_{2,k}(0+,T)^{\gamma_{2,k}} \bigg)
\end{split}
\end{equation}
for all $h, \ell \in \N_0$ with $h \leq \ell$, where the constants $\eta$ and $C$ depend only on $q, N, M, K$ and the exponents $(\alpha_{1,i}, \alpha_{2,i})_{i=1,\ldots,N}$, $(\beta_{1,j}, \beta_{2,j})_{j=1,\ldots,M}$ and $(\gamma_{1,k}, \gamma_{2,k})_{k=1,\ldots,K}$.
\end{lemma}

\begin{proof}
We argue in a similar spirit to the proof of \cite[Lemma~2.12]{FrizHocquetLe2021}. For some fixed $s < t$, we write $\cP^h = \{t^h_n\}_{n=0}^{2^h}$, $h \in \N_0$, for the $(w_{1,i}(\cdot+,\cdot-), w_{2,i}(\cdot+,\cdot-))_{i=1,\ldots,N}$-$(\bw_{1,j}(\cdot+,\cdot-), \bw_{2,j}(\cdot+,\cdot-))_{j=1,\ldots,M}$-$(\hw_{1,k}(\cdot+,\cdot-), \hw_{2,k}(\cdot+,\cdot-))_{k=1,\ldots,K}$-$w$-alternating midpoints of the interval $[s,t]$.

In the following, for $u \in [s,t]$ and $h \in \N_0$, we will write $\lfloor u \rfloor_h := \sup \{t^h_n \in \cP^h : t^h_n \leq u\}$, and write
\begin{equation*}
\Xi^h_{s,v} := \sum_{i \, : \, t^h_{n+1} \leq v} \Xi_{t^h_n,t^h_{n+1}}
\end{equation*}
for $v \in \cP^h$. In particular, we note that
\begin{equation}\label{eq: Xi^h = delta Xi^Ph}
\Xi^h_{s,v} = \delta \Xi^{\cP^h}_{s,v},
\end{equation}
for every $v \in \cP^h$, where $\Xi^{\cP^h}$ is defined as in \eqref{eq: defn Xi^cP}.

For any $u \in [s,t]$ and $\ell \in \N$, we have that
\begin{align*}
\Xi^\ell_{s,\lfloor u \rfloor_\ell} - \Xi_{s,\lfloor u \rfloor_0} &= \sum_{h=0}^{\ell-1} \big(\Xi^{h+1}_{s,\lfloor u \rfloor_{h+1}} - \Xi^h_{s,\lfloor u \rfloor_h}\big) = \sum_{h=0}^{\ell-1} \big(\Xi^{h+1}_{s,\lfloor u \rfloor_h} - \Xi^h_{s,\lfloor u \rfloor_h} + \Xi_{\lfloor u \rfloor_h,\lfloor u \rfloor_{h+1}}\big)\\
&= \sum_{h=0}^{\ell-1} \bigg(\sum_{n \, : \, t^h_{n+1} \leq \lfloor u \rfloor_h} - \delta \Xi_{t^h_n,t^{h+1}_{2n+1},t^h_{n+1}}\bigg) + \sum_{h=0}^{\ell-1} \Xi_{\lfloor u \rfloor_h,\lfloor u \rfloor_{h+1}}.
\end{align*}
If $u = t$, then $\lfloor u \rfloor_h = t$ for every $h \in \N$, and we simply have
\begin{equation*}
\Xi^\ell_{s,\lfloor u \rfloor_\ell} - \Xi_{s,\lfloor u \rfloor_\ell} = \sum_{h=0}^{\ell-1} \bigg(\sum_{n \, : \, t^h_{n+1} \leq \lfloor u \rfloor_h} - \delta \Xi_{t^h_n,t^{h+1}_{2n+1},t^h_{n+1}}\bigg).
\end{equation*}
Otherwise, we have $u \in [s,t)$, so that $\lfloor u \rfloor_0 = s$, and we obtain
\begin{equation*}
\Xi^\ell_{s,\lfloor u \rfloor_\ell} - \Xi_{s,\lfloor u \rfloor_\ell} = \sum_{h=0}^{\ell-1} \bigg(\sum_{n \, : \, t^h_{n+1} \leq \lfloor u \rfloor_h} - \delta \Xi_{t^h_n,t^{h+1}_{2n+1},t^h_{n+1}}\bigg) + \sum_{h=0}^{\ell-1} \Xi_{\lfloor u \rfloor_h,\lfloor u \rfloor_{h+1}} - \Xi_{s,\lfloor u \rfloor_\ell}.
\end{equation*}

Let $\eta > 0$ such that $\eta < \min_{i} (\alpha_{1,i} \wedge \alpha_{2,i})$, $2\eta < \min_{i} (\alpha_{1,i} + \alpha_{2,i} - 1)$, $\eta < \min_{j} (\beta_{1,j} \wedge \beta_{2,j})$, $2\eta < \min_{j} (\beta_{1,j} + \beta_{2,j} - \frac{1}{2})$, $\eta < \min_{k} (\gamma_{1,k} \wedge \gamma_{2,k})$ and $2\eta < \min_{k} (\gamma_{1,k} + \gamma_{2,k} - \frac{1}{q})$.

Let us write $Z^h_n := -\delta \Xi_{t^h_n,t^{h+1}_{2n+1},t^h_{n+1}}$. By Lemma~\ref{lemma: application of conditional BDG} (and Remark~\ref{remark: the conditional BDG including sup}), the bounds in \eqref{eq: stoch sewing E_s Lqr bound} and \eqref{eq: stoch sewing qrs bound}, and Lemma~\ref{lemma: alternating midpoints bounds for sum of non continuous controls}, we have that
\begin{align*}
&\bigg\| \sup_{u \in [s,t]} \bigg| \sum_{n \, : \, t^h_{n+1} \leq \lfloor u \rfloor_h} \big( Z^h_n - \E_{t^h_n} [Z^h_n] \big) \bigg| \bigg\|_{q,r,s} \lesssim \bigg( \sum_{n=0}^{2^h-1} \|Z^h_n\|_{q,r,t^h_n}^2 \bigg)^{\hspace{-2pt}\frac{1}{2}}\\
&\lesssim \bigg( \sum_{n=0}^{2^h-1} \sum_{j=1}^M \bw_{1,j}(t^h_n,t^h_{n+1}-)^{2\beta_{1,j}} \bw_{2,j}(t^h_n+,t^h_{n+1})^{2\beta_{2,j}} \bigg)^{\hspace{-2pt}\frac{1}{2}}\\
&\lesssim 2^{-\eta \lfloor \frac{h}{H} \rfloor} \bigg( \sum_{j=1}^M \bw_{1,j}(s,t-)^{2\beta_{1,j}} \bw_{2,j}(s+,t)^{2\beta_{2,j}} \bigg)^{\hspace{-2pt}\frac{1}{2}},
\end{align*}
where $H = 2N + 2M + 2K + 1$, and, similarly,
\begin{align*}
&\bigg\| \sup_{u \in [s,t]} \bigg| \sum_{n \, : \, t^h_{n+1} \leq \lfloor u \rfloor_h} \E_{t^h_n} [Z^h_n] \bigg| \bigg\|_{q,r,s} \leq \bigg\| \sum_{n=0}^{2^h-1} \big| \E_{t^h_n} [Z^h_n] \big| \bigg\|_{q,r,s} \leq \sum_{n=0}^{2^h-1} \big\| \E_{t^h_n} [Z^h_n] \big\|_{L^r}\\
&\leq \sum_{n=0}^{2^h-1} \sum_{i=1}^N w_{1,i}(t^h_n,t^h_{n+1}-)^{\alpha_{1,i}} w_{2,i}(t^h_n+,t^h_{n+1})^{\alpha_{2,i}} \lesssim 2^{-\eta \lfloor \frac{h}{H} \rfloor} \sum_{i=1}^N w_{1,i}(s,t-)^{\alpha_{1,i}} w_{2,i}(s+,t)^{\alpha_{2,i}}.
\end{align*}
Thus,
\begin{align*}
&\bigg\| \sup_{u \in [s,t]} \bigg|\sum_{h=0}^{\ell-1} \bigg( \sum_{n \, : \, t^h_{n+1} \leq \lfloor u \rfloor_h} - \delta \Xi_{t^h_n,t^{h+1}_{2n+1},t^h_{n+1}} \bigg) \bigg| \bigg\|_{q,r,s}\\
&\leq \sum_{h=0}^{\ell-1} \bigg\| \sup_{u \in [s,t]} \bigg| \sum_{n \, : \, t^h_{n+1} \leq \lfloor u \rfloor_h} \E_{t^h_n} [Z^h_n] \bigg| \bigg\|_{q,r,s} + \sum_{h=0}^{\ell-1} \bigg\| \sup_{u \in [s,t]} \bigg| \sum_{n \, : \, t^h_{n+1} \leq \lfloor u \rfloor_h} \big( Z^h_n - \E_{t^h_n} [Z^h_n] \big) \bigg| \bigg\|_{q,r,s}\\
&\lesssim \sum_{h=0}^\infty 2^{-\eta \lfloor \frac{h}{H} \rfloor} \sum_{i=1}^N w_{1,i}(s,t-)^{\alpha_{1,i}} w_{2,i}(s+,t)^{\alpha_{2,i}} + \sum_{h=0}^\infty 2^{-\eta \lfloor \frac{h}{H} \rfloor} \bigg( \sum_{j=1}^M \bw_{1,j}(s,t-)^{2\beta_{1,j}} \bw_{2,j}(s+,t)^{2\beta_{2,j}} \bigg)^{\hspace{-2pt}\frac{1}{2}}\\
&\lesssim \sum_{i=1}^N w_{1,i}(s,t-)^{\alpha_{1,i}} w_{2,i}(s+,t)^{\alpha_{2,i}} + \bigg( \sum_{j=1}^M \bw_{1,j}(s,t-)^{2\beta_{1,j}} \bw_{2,j}(s+,t)^{2\beta_{2,j}} \bigg)^{\hspace{-2pt}\frac{1}{2}}.
\end{align*}

Noting that $\sum_{h=0}^{\ell-1} \Xi_{\lfloor u \rfloor_h,\lfloor u \rfloor_{h+1}} - \Xi_{s,\lfloor u \rfloor_\ell} = -\sum_{h=0}^{\ell-1} \delta \Xi_{\lfloor u \rfloor_h,\lfloor u \rfloor_{h+1},\lfloor u \rfloor_\ell}$ for any $u \in [s,t)$, we have that
\begin{align*}
\sup_{u \in [s,t)} \bigg|\sum_{h=0}^{\ell-1} \Xi_{\lfloor u \rfloor_h,\lfloor u \rfloor_{h+1}} - \Xi_{s,\lfloor u \rfloor_\ell}\bigg| &\leq \sup_{u \in [s,t)} \sum_{h=0}^{\ell-1} |\delta \Xi_{\lfloor u \rfloor_h,\lfloor u \rfloor_{h+1},\lfloor u \rfloor_\ell}|\\
&\leq \sum_{h=0}^{\ell-1} \sup_{0 \leq n < 2^h} \, \sup_{u \in [t^{h+1}_{2n+1},t^h_{n+1}]} |\delta \Xi_{t^h_n,t^{h+1}_{2n+1},u}|.
\end{align*}
Using \eqref{eq: uniform q,r,s bound for uniform convergence} and Lemma~\ref{lemma: alternating midpoints bounds for sum of non continuous controls}, we have that
\begin{align*}
&\bigg\| \sup_{0 \leq n < 2^h} \, \sup_{u \in [t^{h+1}_{2n+1},t^h_{n+1}]} |\delta \Xi_{t^h_n,t^{h+1}_{2n+1},u}| \bigg\|_{q,r,s}^q \leq \bigg\| \E_s \bigg[ \sum_{n=0}^{2^h-1} \sup_{u \in [t^{h+1}_{2n+1},t^h_{n+1}]} |\delta \Xi_{t^h_n,t^{h+1}_{2n+1},u}|^q \bigg] \bigg\|_{L^{\frac{r}{q}}}\\
&\leq \sum_{n=0}^{2^h-1} \bigg\| \E_{t^h_n} \bigg[ \sup_{u \in [t^{h+1}_{2n+1},t^h_{n+1}]} |\delta \Xi_{t^h_n,t^{h+1}_{2n+1},u}|^q \bigg] \bigg\|_{L^{\frac{r}{q}}} = \sum_{n=0}^{2^h-1} \bigg\| \sup_{u \in [t^{h+1}_{2n+1},t^h_{n+1}]} |\delta \Xi_{t^h_n,t^{h+1}_{2n+1},u}| \bigg\|_{q,r,t^h_n}^q\\
&\leq \sum_{n=0}^{2^h-1} \bigg( \sum_{k=1}^K \hw_{1,k}(t^h_n,t^h_{n+1}-)^{\gamma_{1,k}} \hw_{2,k}(t^h_n+,t^h_{n+1})^{\gamma_{2,k}} \bigg)^{\hspace{-2pt}q}\\
&\lesssim \sum_{n=0}^{2^h-1} \sum_{k=1}^K \hw_{1,k}(t^h_n,t^h_{n+1}-)^{q\gamma_{1,k}} \hw_{2,k}(t^h_n+,t^h_{n+1})^{q\gamma_{2,k}}\\
&\lesssim 2^{-q \eta \lfloor \frac{h}{H} \rfloor} \sum_{k=1}^K \hw_{1,k}(s,t-)^{q\gamma_{1,k}} \hw_{2,k}(s+,t)^{q\gamma_{2,k}},
\end{align*}
so that
\begin{equation*}
\begin{split}
\bigg\| \sup_{u \in [s,t)} \bigg| \sum_{h=0}^{\ell-1} \Xi_{\lfloor u \rfloor_h,\lfloor u \rfloor_{h+1}} - \Xi_{s,\lfloor u \rfloor_\ell} \bigg| \bigg\|_{q,r,s} &\lesssim \sum_{h=0}^\infty 2^{-\eta \lfloor \frac{h}{H} \rfloor} \bigg( \sum_{k=1}^K \hw_{1,k}(s,t-)^{q\gamma_{1,k}} \hw_{2,k}(s+,t)^{q\gamma_{2,k}} \bigg)^{\hspace{-2pt}\frac{1}{q}}\\
&\lesssim \bigg( \sum_{k=1}^K \hw_{1,k}(s,t-)^{q\gamma_{1,k}} \hw_{2,k}(s+,t)^{q\gamma_{2,k}} \bigg)^{\hspace{-2pt}\frac{1}{q}}.
\end{split}
\end{equation*}

Combining the bounds derived above, we obtain
\begin{equation*}
\Big\|\sup_{u \in [s,t]} \big|\Xi^\ell_{s,\lfloor u \rfloor_\ell} - \Xi_{s,\lfloor u \rfloor_\ell}\big|\Big\|_{q,r,s} \lesssim w_1(s,t) + w_2(s,t)^{\frac{1}{2}} + w_3(s,t)^{\frac{1}{q}},
\end{equation*}
where we define
\begin{equation*}
\begin{split}
w_1(s,t) &= \sum_{i=1}^N w_{1,i}(s,t-)^{\alpha_{1,i}} w_{2,i}(s+,t)^{\alpha_{2,i}},\\
w_2(s,t) &= \sum_{j=1}^M \bw_{1,j}(s,t-)^{2\beta_{1,j}} \bw_{2,j}(s+,t)^{2\beta_{2,j}},\\
w_3(s,t) &= \sum_{k=1}^K \hw_{1,k}(s,t-)^{q\gamma_{1,k}} \hw_{2,k}(s+,t)^{q\gamma_{2,k}}.
\end{split}
\end{equation*}
Thus, for every $h \leq \ell$, since $\cP^h \subseteq \cP^\ell$, we have that
\begin{equation*}
\Big\| \sup_{v \in \cP^h} |\Xi^\ell_{s,v} - \Xi_{s,v}| \Big\|_{q,r,s} \leq \Big\| \sup_{v \in \cP^\ell} |\Xi^\ell_{s,v} - \Xi_{s,v}| \Big\|_{q,r,s} \lesssim w_1(s,t) + w_2(s,t)^{\frac{1}{2}} + w_3(s,t)^{\frac{1}{q}}.
\end{equation*}

Recalling \eqref{eq: Xi^h = delta Xi^Ph}, we see that, by Lemma~\ref{lemma: unifrom approximation in sewing on the partition},
\[ \Big\| \sup_{v \in \cP^\ell} |\delta \cI_{s,v} - \Xi^\ell_{s,v}| \Big\|_{q,r,s} = \Big\| \sup_{v \in \cP^\ell} |\delta \cI_{s,v} - \delta \Xi^{\cP^\ell}_{s,v}| \Big\|_{q,r,s} \longrightarrow 0 \]
as $\ell \to \infty$. It follows that, for every $h \in \N$,
\begin{equation}\label{eq: uniform bound delta I - Xi by w345}
\Big\| \sup_{v \in \cP^h} |\delta \cI_{s,v} - \Xi_{s,v}| \Big\|_{q,r,s} \lesssim w_1(s,t) + w_2(s,t)^{\frac{1}{2}} + w_3(s,t)^{\frac{1}{q}}.
\end{equation}

We now let $\cP^h = \{d^h_n\}_{n=0}^{2^h}$, $h \in \N_0$, be the alternating midpoints of the interval $[0,T]$, as specified in the statement of Lemma~\ref{lemma: major lemma for cadlag property of integrals}. We note that, for any $h \in \N$ and any $n = 0, 1, \ldots, 2^h - 1$, the sequence of partitions $\cP^\ell|_{[d^h_n,d^h_{n+1}]}$ for $\ell \geq h$ are actually the $(w_{1,i}(\cdot+,\cdot-), w_{2,i}(\cdot+,\cdot-))_{i=1,\ldots,N}$-$(\bw_{1,j}(\cdot+,\cdot-), \bw_{2,j}(\cdot+,\cdot-))_{j=1,\ldots,M}$-$(\hw_{1,k}(\cdot+,\cdot-), \hw_{2,k}(\cdot+,\cdot-))_{k=1,\ldots,K}$-$w$-alternating midpoints of the interval $[d^h_n,d^h_{n+1}]$. (The order in which we alternate the controls might have changed, but this is irrelevant for our estimates.) Thus, using the bound in \eqref{eq: uniform bound delta I - Xi by w345} applied on the interval $[s,t] = [d^h_n,d^h_{n+1}]$, together with Lemma~\ref{lemma: alternating midpoints bounds for sum of non continuous controls} again, we have, for any $h \leq \ell$, that
\begin{align*}
&\Big\| \sup_{0 \leq n < 2^h} \, \sup_{t \in \cP^\ell|_{[d^h_n,d^h_{n+1}]}} |\delta \cI_{d^h_n,t} - \Xi_{d^h_n,t}| \Big\|_{q,r,0}^q \leq \bigg\| \E_0 \bigg[ \sum_{n=0}^{2^h-1} \sup_{t \in \cP^\ell|_{[d^h_n,d^h_{n+1}]}} |\delta \cI_{d^h_n,t} - \Xi_{d^h_n,t}|^q \bigg] \bigg\|_{L^{\frac{r}{q}}}\\
&\leq \sum_{n=0}^{2^h-1} \bigg\| \E_{d^h_n} \bigg[ \sup_{t \in \cP^\ell|_{[d^h_n,d^h_{n+1}]}} |\delta \cI_{d^h_n,t} - \Xi_{d^h_n,t}|^q \bigg] \bigg\|_{L^{\frac{r}{q}}} = \sum_{n=0}^{2^h-1} \Big\| \sup_{t \in \cP^\ell|_{[d^h_n,d^h_{n+1}]}} |\delta \cI_{d^h_n,t} - \Xi_{d^h_n,t}| \Big\|_{q,r,d^h_n}^q\\
&\lesssim \sum_{n=0}^{2^h-1} \big( w_1(d^h_n,d^h_{n+1})^q + w_2(d^h_n,d^h_{n+1})^{\frac{q}{2}} + w_3(d^h_n,d^h_{n+1}) \big)\\
&\lesssim \bigg( \sum_{n=0}^{2^h-1} \sum_{i=1}^N w_{1,i}(d^h_n,d^h_{n+1}-)^{\alpha_{1,i}} w_{2,i}(d^h_n+,d^h_{n+1})^{\alpha_{2,i}} \bigg)^{\hspace{-2pt}q}\\
&\quad + \bigg( \sum_{n=0}^{2^h-1} \sum_{j=1}^M \bw_{1,j}(d^h_n,d^h_{n+1}-)^{2\beta_{1,j}} \bw_{2,j}(d^h_n+,d^h_{n+1})^{2\beta_{2,j}} \bigg)^{\hspace{-2pt}\frac{q}{2}}\\
&\quad + \sum_{n=0}^{2^h-1} \sum_{k=1}^K \hw_{1,k}(d^h_n,d^h_{n+1}-)^{q\gamma_{1,k}} \hw_{2,k}(d^h_n+,d^h_{n+1})^{q\gamma_{2,k}}\\
&\lesssim 2^{-q \eta \lfloor \frac{h}{H} \rfloor} \big( w_1(0,T)^q + w_2(0,T)^{\frac{q}{2}} + w_3(0,T) \big),
\end{align*}
from which it follows that the estimate in \eqref{eq: uniform bound sup_i sup_t delta I - Xi} holds for some constant $C$ and some possibly new $\eta > 0$.
\end{proof}

\begin{proof}[Proof of Theorem~\ref{theorem: sample path regularity sewing}]
Let $\eta > 0$ and $\cP^h = \{d^h_n\}_{n=0}^{2^h}$, $h \in \N_0$, be the constant and sequence of partitions specified in Lemma~\ref{lemma: major lemma for cadlag property of integrals}. For $h \in \N$ and $t \in [0,T]$, we adopt the shorthand
\begin{equation*}
\Xi^h_t := \Xi^{\cP^h}_t = \sum_{n=0}^{2^h-1} \Xi_{d^h_n \wedge t,d^h_{n+1} \wedge t}.
\end{equation*}
For every $\ell \geq h$, we have that
\begin{equation*}
\sup_{t \in \cP^\ell} |\cI_t - \Xi^h_t| \leq \sup_{v \in \cP^h} |\cI_v - \Xi^h_v| + \sup_{0 \leq n < 2^h} \, \sup_{t \in \cP^\ell|_{[d^h_n,d^h_{n+1}]}} |\delta \cI_{d^h_n,t} - \Xi_{d^h_n,t}|.
\end{equation*}
It then follows from Lemma~\ref{lemma: unifrom approximation in sewing on the partition}\footnote{The alternating midpoints we consider here involve additional controls to the ones in Lemma~\ref{lemma: unifrom approximation in sewing on the partition}, but it is easy to see that the result is nonetheless true, albeit with possibly different constants.} and the bound in \eqref{eq: uniform bound sup_i sup_t delta I - Xi} that
\begin{equation}\label{eq: combined bound for I_t - Xi^h_t}
\begin{split}
\Big\| \sup_{t \in \cP^\ell} |\cI_t &- \Xi^h_t| \Big\|_{q,r,0} \lesssim 2^{-\eta h} \bigg( \sum_{i=1}^N w_{1,i}(0,T-)^{\alpha_{1,i}} w_{2,i}(0+,T)^{\alpha_{2,i}}\\
&+ \sum_{j=1}^M \bw_{1,j}(0,T-)^{\beta_{1,j}} \bw_{2,j}(0+,T)^{\beta_{2,j}} + \sum_{k=1}^K \hw_{1,k}(0,T-)^{\gamma_{1,k}} \hw_{2,k}(0+,T)^{\gamma_{2,k}} \bigg)
\end{split}
\end{equation}
for some possibly new $\eta > 0$. We infer that there exists a subsequence $(h_m)_{m \in \N}$, such that, for each $m \in \N$ and $\ell \geq h_m$,
\begin{equation*}
\Big\| \sup_{t \in \cP^\ell} |\cI_t - \Xi^{h_m}_t| \Big\|_{q,r,0} < 2^{-\eta m}.
\end{equation*}
Let $\cP = \bigcup_{h \in \N_0} \cP^h$. Since the sequence of partitions $(\cP^\ell)_{\ell \in \N}$ is nested, it follows from part~(v) of Proposition~\ref{prop: Lqrs is Banach space} that
\begin{equation*}
\Big\| \sup_{t \in \cP} |\cI_t - \Xi^{h_m}_t| \Big\|_{q,r,0} \leq \liminf_{\ell \to \infty} \Big\| \sup_{t \in \cP^\ell} |\cI_t - \Xi^{h_m}_t| \Big\|_{q,r,0} \leq 2^{-\eta m}.
\end{equation*}
By Markov's inequality, $\P(\sup_{t \in \cP} |\cI_t - \Xi^{h_m}_t| \geq 2^{-\frac{\eta m}{2}}) \lesssim 2^{-\frac{\eta m}{2}}$, and it follows from the Borel--Cantelli lemma that, almost surely, $\Xi^{h_m} \to \cI$ uniformly on $\cP$ as $m \to \infty$. Since each $\Xi^{h_m}$ is almost surely c\`adl\`ag, the same is thus true of $\cI$ on $\cP$.

Since we included the $w$-midpoints when defining $(\cP^h)_{h \in \N_0}$, we have from Lemma~\ref{lemma: alternating-midpoints may replace dyadics} that $|\cP^h| \leq 2^{-\lfloor \frac{h}{H} \rfloor} T$ for all $h \in \N_0$, so that in particular $|\cP^h| \to 0$ as $h \to \infty$. We infer that $\cP$ is dense in $[0,T]$, and we can define
\begin{equation*}
\widetilde{\cI}_s := \lim_{t \in \cP, \, t \searrow s} \cI_t
\end{equation*}
as an almost sure limit for every $s \in [0,T)$, and let $\widetilde{\cI}_T := \cI_T$.

It follows from \eqref{eq: qrs bound full generality sewing} that $\delta \cI_{s,t} - \Xi_{s,t} \to 0$ in $L^{q,r}_s$ as $t \searrow s$, and we have by assumption that $\Xi_{s,t} \to 0$ almost surely as $t \searrow s$. Together, these imply that $(\cI_s)_{s \in [0,T]}$ is right-continuous in probability. By uniqueness of limits in probability, we deduce that $\widetilde{\cI}_s = \cI_s$ almost surely for every $s \in [0,T]$, so that $\widetilde{\cI}$ is indeed a version of $\cI$.

Finally, to see the estimate in \eqref{eq: uniform estimate in full generality sewing}, we note that
\begin{equation*}
\Big\|\sup_{t \in [0,T]} |\widetilde{\cI}_t - \Xi^h_t|\Big\|_{q,r,0} = \Big\|\sup_{t \in \cP} |\cI_t - \Xi^h_t|\Big\|_{q,r,0} \leq \liminf_{\ell \to \infty} \Big\|\sup_{t \in \cP^\ell} |\cI_t - \Xi^h_t|\Big\|_{q,r,0},
\end{equation*}
and combine this with the bound in \eqref{eq: combined bound for I_t - Xi^h_t}.
\end{proof}

\section{Additional technical results}

\begin{lemma}\label{lemma: convergence to 0 of two param processes}
Let $q \in [2,\infty)$, and let $(R_{s,t})_{(s,t) \in \Delta_{[0,T]}}$ be an $L^q$-valued two-parameter process, such that $R_{s,t}$ is $\cF_t$-measurable for every $(s,t) \in \Delta_{[0,T]}$. Let $(w_{1,i}, w_{2,i})_{i=1,\ldots,N}$ and $(\bw_{1,j}, \bw_{2, j})_{j=1,\ldots,M}$ be controls. If
\begin{align*}
\big\| \E_s [R_{s,t}] \big\|_{L^q} &\leq \sum_{i=1}^N w_{1,i}(s,t-)^{\alpha_{1,i}} w_{2,i}(s+,t)^{\alpha_{2,i}},\\
\|R_{s,t}\|_{L^q} &\leq \sum_{j=1}^M \bw_{1,j}(s,t-)^{\beta_{1,j}} \bw_{2,j}(s+,t)^{\beta_{2,j}}
\end{align*}
for every $(s,t) \in \Delta_{[0,T]}$, where $\alpha_{1,i}, \alpha_{2,i}, \beta_{1,j}, \beta_{2,j} > 0$ are such that $\alpha_{1,i} + \alpha_{2,i} > 1$ for every $i = 1, \ldots, N$, and $\beta_{1,j} + \beta_{2,j} > \frac{1}{2}$ for every $j = 1, \ldots, M$, then
\begin{equation}\label{eq: limit to zero in technical result}
\bigg\| \sum_{[s,t] \in \cP} R_{s,t} \bigg\|_{L^q} \, \longrightarrow \, 0
\end{equation}
in the sense of RRS convergence.
\end{lemma}

\begin{proof}
Let $\epsilon > 0$. We let $\delta > 0$ and $\cP^\epsilon$ be the constant and partition as defined in the proof of Theorem~\ref{theorem: mild stochastic sewing lemma}. In particular, by the construction of the partition $\cP^\epsilon$, we have that, for any refinement $\cP \supseteq \cP^\epsilon$, and any $[s,t] \in \cP$, either $w_{1,i}(s,t-) \leq \epsilon^{\frac{1}{\delta}} w_{1,i}(0,T-)$ or $w_{2,i}(s+,t) \leq \epsilon^{\frac{1}{\delta}} w_{2,i}(0+,T)$ for each $i$, and either $\bw_{1,j}(s,t-)\leq \epsilon^{\frac{1}{\delta}}  \bw_{1,j}(0,T-)$ or $\bw_{2,j}(s+,t) \leq \epsilon^{\frac{1}{\delta}} \bw_{2,j}(0+,T)$ for each $j$.

Then, for any refinement $\cP \supseteq \cP^\epsilon$, we may apply Lemma~\ref{lemma: application of conditional BDG} with $r = q$ and the random variables $R_{s,t} - \E_s [R_{s,t}]$ for each $[s,t] \in \cP$ to obtain
\begin{align*}
&\bigg\| \sum_{[s,t] \in \cP} R_{s,t} \bigg\|_{L^q} \lesssim \sum_{[s,t] \in \cP} \big\|\E_s [R_{s,t}]\big\|_{L^q} + \bigg(\sum_{[s,t] \in \cP} \|R_{s,t}\|_{L^q}^2\bigg)^{\hspace{-2pt}\frac{1}{2}}\\
&\lesssim \epsilon \sum_{[s,t] \in \cP} \sum_{i=1}^N w_{1,i}(0,T-)^{\delta} w_{2,i}(0+,T)^{\delta} w_{1,i}(s,t-)^{\alpha_{1,i}-\delta} w_{2,i}(s+,t)^{\alpha_{2,i} - \delta}\\
&\quad + \epsilon \bigg( \sum_{[s,t] \in \cP} \sum_{j=1}^M \bw_{1,j}(0,T-)^{2\delta} \bw_{2,j}(0+,T)^{2\delta}\bw_{1,j}(s,t-)^{2\beta_{1,j}-2\delta}  \bw_{2,j}(s+,t)^{2\beta_{2,j} - 2\delta} \bigg)^{\hspace{-2pt}\frac{1}{2}}\\
&\lesssim \epsilon \sum_{i=1}^N w_{1,i}(0,T-)^{\alpha_{1,i}} w_{2,i}(0+,T)^{\alpha_{2,i}} + \epsilon \sum_{j=1}^M \bw_{1,j}(0,T-)^{\beta_{1,j}} \bw_{2,j}(0+,T)^{\beta_{2,j}},
\end{align*}
from which we deduce the RRS convergence in \eqref{eq: limit to zero in technical result}.
\end{proof}

\begin{lemma}\label{lemma: pure jump stochastic sewing lemma}
Let $q \in [2,\infty)$, let $(\Xi_{s,t})_{(s,t) \in \Delta_{[0,T]}}$ be a two-parameter process such that $\Xi_{s,t}$ is $\cF_t$-measurable for every $(s,t) \in \Delta_{[0,T]}$, and let $w_1, w_2$ be controls such that $w_i(t-,t) = 0$ for each $i = 1, 2$ and $t \in (0,T]$. Let $\delta \colon [0,T] \to [0,\infty)$ such that $\delta(t) \neq 0$ for only countably many $t \in [0,T]$, and such that $\sum_{t \in [0,T]} \delta(t) < \infty$, and suppose that
\begin{align*}
\big\|\E_s [\Xi_{s,t}]\big\|_{L^q} &\leq \delta(t)^{\alpha_1} w_1(s,t)^{\alpha_2},\\
\|\Xi_{s,t}\|_{L^q} &\leq \delta(t)^{\beta_1} w_2(s,t)^{\beta_2}
\end{align*}
for all $(s,t) \in \Delta_{[0,T]}$, where $\alpha_1, \alpha_2, \beta_1, \beta_2 > 0$ are such that $\alpha_1 + \alpha_2 > 1$ and $\beta_1 + \beta_2 > \frac{1}{2}$. Then
\begin{equation*}
\lim_{|\cP| \to 0} \bigg\|\sum_{[s,t] \in \cP} \Xi_{s,t}\bigg\|_{L^q} = 0.
\end{equation*}
\end{lemma}

\begin{proof}
We argue similarly to the proof of \cite[Theorem~2.11]{FrizZhang2018}. We may assume without loss of generality that $\alpha_2 < 1$ and $\beta_2 < \frac{1}{2}$.

Let $\epsilon > 0$, and let $\cP$ be a partition of $[0,T]$. Applying Lemma~\ref{lemma: application of conditional BDG} with $r = q$ and the random variables $\Xi_{s,t} - \E_s [\Xi_{s,t}]$ for $[s,t] \in \cP$, we have that
\begin{equation*}
\bigg\|\sum_{[s,t] \in \cP} \Xi_{s,t}\bigg\|_{L^q} \lesssim \sum_{[s,t] \in \cP} \big\|\E_s [\Xi_{s,t}]\big\|_{L^q} + \bigg(\sum_{[s,t] \in \cP} \|\Xi_{s,t}\|_{L^q}^2\bigg)^{\hspace{-2pt}\frac{1}{2}} \leq S_1(\cP) + S_2(\cP),
\end{equation*}
where
\begin{align*}
S_1(\cP) &:= \sum_{[s,t] \in \cP} \delta(t)^{\alpha_1} w_1(s,t)^{\alpha_2} \1_{\{\delta \geq \epsilon\}}(t) + \bigg(\sum_{[s,t] \in \cP} \delta(t)^{2\beta_1} w_2(s,t)^{2\beta_2} \1_{\{\delta \geq \epsilon\}}(t)\bigg)^{\hspace{-2pt}\frac{1}{2}},\\
S_2(\cP) &:= \sum_{[s,t] \in \cP} \delta(t)^{\alpha_1} w_1(s,t)^{\alpha_2} \1_{\{\delta < \epsilon\}}(t) + \bigg(\sum_{[s,t] \in \cP} \delta(t)^{2\beta_1} w_2(s,t)^{2\beta_2} \1_{\{\delta < \epsilon\}}(t)\bigg)^{\hspace{-2pt}\frac{1}{2}}.
\end{align*}
As there are only finitely many $t \in [0,T]$ with $\delta(t) \geq \epsilon$, we have that $\lim_{|\cP| \to 0} S_1(\cP) = 0$ by our assumption on the controls.

Since clearly $(1 - \alpha_2) + \alpha_2 = 1$ and $(1 - 2\beta_2) + 2\beta_2 = 1$, we can apply H\"older's inequality to obtain
\begin{align*}
S_2(\cP) &\leq \bigg(\sum_{[s,t] \in \cP} \delta(t)^{\frac{\alpha_1}{1 - \alpha_2}} \1_{\{\delta < \epsilon\}}(t)\bigg)^{\hspace{-2pt}1 - \alpha_2} \bigg(\sum_{[s,t] \in \cP} w_1(s,t)\bigg)^{\hspace{-2pt}\alpha_2}\\
&\quad + \bigg(\sum_{[s,t] \in \cP} \delta(t)^{\frac{\beta_1}{\frac{1}{2} - \beta_2}} \1_{\{\delta < \epsilon\}}(t)\bigg)^{\hspace{-2pt}\frac{1}{2} - \beta_2} \bigg(\sum_{[s,t] \in \cP} w_2(s,t)\bigg)^{\hspace{-2pt}\beta_2}\\
&\leq \epsilon^{\alpha_1 + \alpha_2 - 1} \bigg(\sum_{t \in [0,T]} \delta(t)\bigg)^{\hspace{-2pt}1 - \alpha_2} w_1(0,T)^{\alpha_2} + \epsilon^{\beta_1 + \beta_2 - \frac{1}{2}} \bigg(\sum_{t \in [0,T]} \delta(t)\bigg)^{\hspace{-2pt}\frac{1}{2} - \beta_2} w_2(0,T)^{\beta_2}.
\end{align*}
Since $\epsilon > 0$ was arbitrary, we infer the desired convergence.
\end{proof}

\begin{lemma}\label{lemma: contraction for integral against quad. variation for Ito approach}
Let $p \in [2,3)$, $q \in [2,\infty)$, and let $\sigma$ be a predictable bounded Lipschitz function. Suppose that $Y, \tY \in V^p L^q$ are adapted and almost surely c\`adl\`ag, and let $A = (A_t)_{t \in [0,T]}$ be a matrix-valued process which is almost surely c\`adl\`ag and non-decreasing, in the sense that $A_t - A_s$ is positive semi-definite whenever $s \leq t$. Then
\begin{equation}\label{eq: estimate for sigma(Y) - sigma(tY) ^2 dA}
\bigg\|\int_s^t \big(\sigma_{u}(Y_{u-}) - \sigma_{u}(\tY_{u-})\big)^{\otimes 2} \dd A_u\bigg\|_{L^{\frac{q}{2}}} \leq C \big(\|Y_s - \tY_s\|_{L^q}^2 + \|Y - \tY\|^2_{p,q,[s,t)}\big) \|A\|_{\frac{p}{2},\frac{q}{2},\infty,[s,t]}
\end{equation}
for every $(s,t) \in \Delta_{[0,T]}$, where the integral is understood in the Lebesgue--Stieltjes sense, and the constant $C$ depends only on $p$ and $\|\sigma\|_{C^1_b}$.
\end{lemma}

\begin{proof}
For notational simplicity, we will suppose that $\sigma$, $Y$ and $A$ are one-dimensional, but the following argument generalizes easily to the  multidimensional case. We may also suppose that $Y_{s-} = Y_s$ and $\tY_{s-} = \tY_s$, as this does not affect the value of the Lebesgue--Stieltjes integral in \eqref{eq: estimate for sigma(Y) - sigma(tY) ^2 dA}. Since $\sigma$ is predictable bounded Lipschitz, we find by the bound in \eqref{eq: random Lipschitz constant} that
\begin{align*}
\bigg|\int_s^t \big(\sigma_u(Y_{u-}) - \sigma_u(\tY_{u-})\big)^2 \dd A_u\bigg| \lesssim \int_s^t \big(1 \wedge |Y_{u-} - \tY_{u-}|^2\big) \dd A_u.
\end{align*}

For each $n \in \N$, we let $t^n_i := s + i 2^{-n} (t - s)$ for $i = 0, 1, \ldots, 2^n$, and
\begin{equation*}
Z^n_u := \big(1 \wedge |Y_s - \tY_s|^2\big) \1_{\{s\}}(u) + \sum_{i=0}^{2^n-1} \big(1 \wedge |Y_{t^n_i-} - \tY_{t^n_i-}|^2\big) \1_{(t^n_i,t^n_{i+1}]}(u).
\end{equation*}
Since $Y_-$ and $\tY_-$ are almost surely left-continuous, we have that $Z^n_u \to 1 \wedge |Y_{u-} - \tY_{u-}|^2$ as $n \to \infty$ for each $u \in [s,t]$. By dominated convergence, we then deduce that, almost surely,
\begin{equation}\label{eq: int Z dA a.s. limit}
\int_s^t Z^n_u \dd A_u = \sum_{i=0}^{2^n-1} \big(1 \wedge |Y_{t^n_i-} - \tY_{t^n_i-}|^2\big) \delta A_{t^n_i,t^n_{i+1}} \longrightarrow \int_s^t \big(1 \wedge |Y_{u-} - \tY_{u-}|^2\big) \dd A_u
\end{equation}
as $n \to \infty$. For any $G \in V^p L^q$, by H\"older's inequality, we have that
\begin{align*}
\big\| \delta (G^2)_{u,v} \big\|_{L^{\frac{q}{2}}} &= \big\| G_v^2 - G_u^2 \big\|_{L^{\frac{q}{2}}} = \big\| \delta G_{u,v} (G_v + G_u) \big\|_{L^{\frac{q}{2}}}\\
&\leq \|\delta G_{u,v}\|_{L^q} \|G_u + G_v\|_{L^q} \lesssim \|\delta G_{u,v}\|_{L^q} \sup_{r \in [s,t]} \|G_r\|_{L^q},
\end{align*}
so that
\begin{align*}
\| G^2 \|_{p,\frac{q}{2},[s,t]} &\lesssim \| G \|_{p,q,[s,t]} \sup_{r \in [s,t]} \|G_r\|_{L^q} \leq \| G \|_{p,q,[s,t]} \big( \|G_s\|_{L^q} + \| G \|_{p,q,[s,t]} \big)\\
&\lesssim \|G_s\|_{L^q}^2 + \| G \|_{p,q,[s,t]}^2.
\end{align*}
Letting $G = Y - \tY$, we thus have that
\begin{equation*}
\big\|1 \wedge |Y_- - \tY_-|^2\big\|_{p,\frac{q}{2},[s,t]} \leq \big\||Y - \tY|^2\big\|_{p,\frac{q}{2},[s,t]} \lesssim \|Y_s - \tY_s\|_{L^q}^2 + \| Y - \tY \|_{p,q,[s,t]}^2 < \infty,
\end{equation*}
and we may assume without loss of generality that $\|A\|_{\frac{p}{2},\frac{q}{2},\infty,[s,t]} < \infty$. Thus, by \cite[Theorem~2.2]{FrizZhang2018} (Young sewing), the limit
\begin{equation}\label{eq: int Z dA Lq/2 limit}
\int_s^t \big(1 \wedge |Y_{u-} - \tY_{u-}|^2\big) \dd A_u = \lim_{n \to \infty} \int_s^t Z^n_u \dd A_u
\end{equation}
exists in $L^{\frac{q}{2}}$ as a Young integral, and comes with the estimate
\begin{align*}
&\bigg\|\int_s^t \big(1 \wedge |Y_{u-} - \tY_{u-}|^2\big) \dd A_u - \big(1 \wedge |Y_s - \tY_s|^2\big) \delta A_{s,t}\bigg\|_{L^{\frac{q}{2}}}\\
&\lesssim \big\|1 \wedge |Y_- - \tY_-|^2\big\|_{p,\frac{q}{2},[s,t)} \|A\|_{\frac{p}{2},\frac{q}{2},\infty,[s,t]} \lesssim \big( \|Y_s - \tY_s\|_{L^q}^2 + \| Y - \tY \|_{p,q,[s,t]}^2 \big) \|A\|_{\frac{p}{2},\frac{q}{2},\infty,[s,t]}.
\end{align*}
By the uniqueness of limits, the Young integral in \eqref{eq: int Z dA Lq/2 limit} coincides with the pathwise Lebesgue--Stieltjes integral in \eqref{eq: int Z dA a.s. limit}. Since
\begin{align*}
&\big\|\big(1 \wedge |Y_s - \tY_s|^2\big) \delta A_{s,t}\big\|_{L^{\frac{q}{2}}} = \E \Big[\big(1 \wedge |Y_s - \tY_s|^2\big)^{\frac{q}{2}} \E_s \big[|\delta A_{s,t}|^{\frac{q}{2}}\big]\Big]^{\frac{2}{q}}\\
&\leq \|Y_s - \tY_s\|_{L^q}^2 \|\delta A_{s,t}\|_{\frac{q}{2},\infty,s} \leq \|Y_s - \tY_s\|_{L^q}^2 \|A\|_{\frac{p}{2},\frac{q}{2},\infty,[s,t]},
\end{align*}
we deduce that
\begin{align*}
\bigg\|\int_s^t \big(1 \wedge |Y_{u-} - \tY_{u-}|^2\big) \dd A_u\bigg\|_{L^{\frac{q}{2}}} \lesssim \big(\|Y_s - \tY_s\|_{L^q}^2 + \|Y - \tY\|^2_{p,q,[s,t)}\big) \|A\|_{\frac{p}{2},\frac{q}{2},\infty,[s,t]},
\end{align*}
from which we infer the estimate in \eqref{eq: estimate for sigma(Y) - sigma(tY) ^2 dA}.
\end{proof}

\bibliographystyle{abbrv}
\bibliography{References}

\end{document}